\documentclass[a4paper,reqno]{amsart}

\usepackage{amscd,amsthm,amsfonts,amssymb,amsmath,esint}
\usepackage{amsmath,tikz-cd}
\usepackage[colorlinks=true]{hyperref}
\usepackage{fullpage}
\usepackage[all]{xy}
\usepackage{mathtools}
 \usepackage{enumitem}

\newcommand{\wtd}{\widetilde}
\newcommand{\Z}{\mathbb{Z}}

\newcommand{\x}{\times}
\newcommand{\C}{\BC}

\newcommand{\CMP}{\theta}

\newcommand{\cmptv}{\varsigma}

\newcommand{\bksl}{\backslash}

\newcommand{\Thetam}{\Theta}

    \newcommand{\BA}{{\mathbb {A}}} 
    \newcommand{\BC}{{\mathbb {C}}} 
     \newcommand{\BF}{{\mathbb {F}}}
    \newcommand{\BG}{{\mathbb {G}}} \newcommand{\BH}{{\mathbb {H}}}

    \newcommand{\BQ}{{\mathbb {Q}}}

     \newcommand{\BZ}{{\mathbb {Z}}}

    \newcommand{\CA}{{\mathcal {A}}} 
     
    \newcommand{\CE}{{\mathcal {E}}} 
     \newcommand{\CH}{{\mathcal {H}}}
     
     \newcommand{\CL}{{\mathcal {L}}}
    \newcommand{\CM}{{\mathcal {M}}} 
    \newcommand{\CO}{{\mathcal {O}}} 
     \newcommand{\CR}{{\mathcal {R}}}
    \newcommand{\CS}{{\mathcal {S}}} 
     
     \newcommand{\CX}{{\mathcal {X}}}

    \newcommand{\fa}{{\mathfrak{a}}} 
    \newcommand{\fc}{{\mathfrak{c}}} 
     \newcommand{\ff}{{\mathfrak{f}}}

     \newcommand{\fl}{{\mathfrak{l}}}
    \newcommand{\fm}{{\mathfrak{m}}} \newcommand{\fn}{{\mathfrak{n}}}
     \newcommand{\fp}{{\mathfrak{p}}}
    \newcommand{\fq}{{\mathfrak{q}}} \newcommand{\fr}{{\mathfrak{r}}}

     \newcommand{\fN}{{\mathfrak{N}}}

    \newcommand{\alg}{{\mathrm{alg}}}
    \newcommand{\Aut}{{\mathrm{Aut}}}

    \newcommand{\cond}{\mathrm{cond}}\newcommand{\Cond}{{\mathrm{Cond}}}

    \newcommand{\End}{{\mathrm{End}}} \newcommand{\Eis}{{\mathrm{Eis}}}
    
    \newcommand{\Frac}{{\mathrm{Frac}}}

    \newcommand{\Gal}{{\mathrm{Gal}}} \newcommand{\GL}{{\mathrm{GL}}}
    
    \newcommand{\Hom}{{\mathrm{Hom}}}
    
    \renewcommand{\Im}{{\mathrm{Im}}}

    \newcommand{\ord}{{\mathrm{ord}}} 
    \newcommand{\PGL}{{\mathrm{PGL}}} \newcommand{\Pic}{\mathrm{Pic}}
    \newcommand{\pr}{{\mathrm{pr}}}
    
    \renewcommand{\mod}{\ \mathrm{mod}\ }
    
    \newcommand{\ol}{\overline}

    \newcommand{\Sel}{{\mathrm{Sel}}}

    \newcommand{\SL}{{\mathrm{SL}}}

    \newcommand{\sgn}{{\mathrm{sgn}}}
    \newcommand{\Stab}{{\mathrm{Stab}}}

    \newcommand{\tr}{{\mathrm{tr}}}

    \newcommand{\vol}{{\mathrm{vol}}}

    \newcommand{\wh}{\widehat}
    
    \newcommand{\pair}[1]{\langle {#1} \rangle}

    \newcommand{\ov}{\overline}

    \newcommand{\ra}{\rightarrow} 
    \newcommand{\bs}{\backslash}
    \newcommand{\nequiv}{\equiv\hspace{-10pt}/\ }

\newcommand{\CMspace}{{\rm{CM}}}

    \theoremstyle{plain}
     \newtheorem{thm}{Theorem}[section] \newtheorem{cor}[thm]{Corollary}
    \newtheorem{lem}[thm]{Lemma}  \newtheorem{prop}[thm]{Proposition}
    \newtheorem {conj}[thm]{Conjecture}
    \newtheorem{defn-lem}[thm]{Definition and Lemma} \newtheorem{defn}[thm]{Definition}

\theoremstyle{remark} \newtheorem{remark}[thm]{Remark}
\theoremstyle{remark} 
\theoremstyle{remark} \newtheorem{example}{Example}

    \newcommand{\Neron}{N\'{e}ron~}
    \newcommand{\adeles}{ad\'{e}les~}

    \numberwithin{equation}{section}

    \newcommand{\dfn}[1]{\textsf{\color{magenta}#1}}
\begin{document}

\title{Mod $\ell$ non-vanishing of self-dual Hecke $L$-values over CM fields and applications}
\author{Ashay Burungale, Wei He, Ye Tian and Xiangdong Ye} \address{Ashay A. Burungale: Department of Mathematics, The university of Texas at Austin,
2515 Speedway, Austin TX 78712, USA.} 
\email{ashayburungale@gmail.com}

\address{Wei He: School of Mathematics and Statistics, Xi'an Jiaotong University, Xi'an 710049, P.R. China.} 
\email{hewei0714@xjtu.edu.cn}

\address{Ye Tian: Morningside Center of Mathematics; Academy of Mathematics and Systems Science, Chinese
Academy of Sciences, Beijing 100190, China; and         School of Mathematical Sciences,  University of the Chinese Academy of Sciences, Beijing 100049, China.}
\email{ytian@math.ac.cn}

\address{Xiangdong Ye: Department of Mathematics,
University of Science and Technology of China, Hefei, Anhui 230026, P.R. China.}
\email{yexd@ustc.edu.cn}
\maketitle

\begin{abstract}
Let $\lambda$ be a self-dual Hecke character over a CM field $K$. Let $\fp$ be a degree one prime of the maximal totally real subfield $F$ of $K$ and $\Gamma_{\fp}$ the Galois group of the  anticyclotomic $\BZ_p$-extension of $K$ unramified outside $\fp$. 
We prove that 
$$L(1,\lambda\nu)\neq 0$$
for all but finitely many finite order characters $\nu$ of $\Gamma_\fp$ such that $\varepsilon(\lambda\nu)=+1$. For an ordinary prime $\ell$ with respect to the CM quadratic extension $K/F$, we also determine the $\ell$-adic valuation of the normalised Hecke $L$-values 
$L^{\alg}(1,\lambda\nu)$. As an application, we complete Hsieh's proof of Eisenstein congruence divisibility towards the CM Iwasawa main conjecture over $K$. 

Our approach and results complement the prior work initiated by Hida's ideas on the arithmetic of Hilbert modular Eisenstein series, studied via mod $\ell$ analogue of the Andr\'e--Oort conjecture. The previous results  
established the non-vanishing only for {infinitely many characters $\nu$}. Our approach is based on the arithmetic of a CM modular form on a Shimura set, studied via arithmetic of the CM field and Ratner's ergodicity of unipotent flows.  
\end{abstract}
\setcounter{tocdepth}{1}
\tableofcontents \section{Introduction} 
Special values of $L$-functions mysteriously encode arithmetic. As underlying motives vary in a
family, a fundamental problem is whether the $L$-values are generically non-zero. If so, a finer problem:
mod $\ell$ non-vanishing of algebraic part of the $L$-values for a fixed prime $\ell$. 

A basic example arises from CM motives or Hecke characters over a CM field. The aim of this paper is to establish mod $\ell$ non-vanishing of central Hecke $L$-values in a $p$-adic self-dual family over a CM field for primes $\ell \neq p$ {(see Theorems~\ref{gnv}~and~\ref{mainthm})}. {It has an application to CM Iwasawa theory: the completion of Hsieh's proof of Eisenstein congruence divisibility towards the CM main conjecture \cite{Hs2} (see Theorem~\ref{mc}).} 

\subsection{Context} Let 
 $\ell$ and $p$ be distinct primes.

The study of mod $\ell$ non-vanishing of $L$-values in a $p$-adic family goes back to the seminal work of Washington \cite{Was} in the late 70's. He proved that the $\ell$-adic valuation of Dirichlet $L$-values in the $p$-adic cyclotomic family is bounded (cf.~\cite{Sin}), implying the boundedness of the $\ell$-part of class numbers in the $\BZ_p$-cyclotomic extension of $\BQ$. 
It has various arithmetic consequences, for example it is an ingredient in the modularity of residually reducible Galois representations \cite{SW}. 

We surmise that the $\ell$-part of critical $L$-values in a $p$-adic family of motives is bounded rather generally.
 The study of such a non-vanishing for Hecke $L$-values over an imaginary quadratic field was
initiated in the mid 80's. 
The first result is due to Gillard \cite{Gi}, who 
proved the boundedness for $p$-adic family arising from Coates--Wiles 
$\BZ_p$-deformation\footnote{In this setting $p$ splits in the CM field $K$ as $(p)=\fp \ov{\fp}$, and Coates--Wiles $\Z_p$-extension arises from the $\fp^\infty$ or $\ov{\fp}^\infty$-torsion points of the CM curve.} of a CM elliptic curve 
over the CM field. 
It is 
based on Zariski density of torsion points on a self-product of the CM elliptic curve modulo $\ell$. 

A couple of decades later, Hida initiated and extensively studied \cite{Hi1,Hi2,Hi3,Hi4} 
the mod $\ell$ non-vanishing for the anticyclotomic $\BZ_p$-deformation of an arithmetic Hecke character $\lambda$ over a general CM field $K$ (see also \cite{Hsieh:nv,He}). 
More precisely, let $F$ be the maximal totally real subfield of $K$, $\fp$ a prime of $F$ above $p$,   
$\Gamma_\fp$ the Galois group of the maximal anticyclotomic $\BZ_p^{\deg \fp}$-extension of $K$ unramified outside $\fp$ and $\ell$ an odd ordinary prime for the CM quadratic extension $K/F$. 
Then the mod $\ell$ non-vanishing was proved
for critical Hecke $L$-values $$L^{\alg}(m,\lambda\nu)$$ for a Zariski dense subset of finite order characters $\nu$ of $\Gamma_\fp$. Moreover, if $\deg \fp =1$, then the non-vanishing was proved for all except finitely many $\nu$.   
Hida's strategy is based on the arithmetic of a $\GL_{2}(F)$-Eisenstein series,
studied via 
 geometry of mod $\ell$ Hilbert modular Shimura variety $\CX$ and Zariski density of CM points on a self-product of $\CX$. 
 The latter is a mod $\ell$ analogue of the Andr\'e--Oort conjecture \cite{Ch} (Chai--Oort rigidity principle).

 In the mid 2000's Finis established  mod $\ell$ non-vanishing of central Hecke $L$-values over an imaginary quadratic field in a self-dual family arising from the anticyclotomic $\BZ_p$-deformation of a self-dual Hecke character  
 \cite{Fin1,Fin2}. He proved the non-vanishing for all except finitely many $\nu$. 
His notably different approach relies on the arithmetic of $U(1)$-theta functions with complex multiplication, studied via Mumford's theory of geometric theta functions, the Manin--Mumford conjecture and elliptic units. It has applications to Iwasawa theory of elliptic curves, specifically to Iwasawa main conjectures (cf.~\cite{SU,BKO}). In light of its dependence on elliptic units the approach does not yet generalise to CM fields.  

 Hida's mod $\ell$ non-vanishing of Hecke $L$-values over a general CM field has found applications to the arithmetic of the CM field (cf.~\cite{Hi:MC,Oh}). It is an ingredient  
in Hsieh's 2014 proof of Eisenstein congruence divisibility \cite{Hs2} towards the  CM Iwasawa main conjecture over a CM field. 
However, in 2018 a gap was found in Hida's strategy, specifically in his proof of Zariski density of CM points on a self-product of $\CX$. In a recent paper \cite{Hi4} Hida recovered\footnote{For the case $F=\BQ$, see also \cite{nv:CM,Oh'}.} several of the results of \cite{Hi1,Hi2,Hsieh:nv}. 
However, if $\deg \fp =1$, then the non-vanishing is only proved for infinitely many characters $\nu$ 
and {\it not} all except finitely many $\nu$ as originally claimed.  
Consequently, Hsieh's proof \cite{Hs2} of $\ell$-indivisibility of an ancillary Eisenstein series on the unitary group $U(2,1)$, and hence that of the main result of \cite{Hs2} towards the CM Iwasawa main conjecture are incomplete.

In this paper we establish a uniform non-vanishing of self-dual Hecke $L$-values over a general CM field $K$ by a new method. For a self-dual Hecke character $\lambda$ over $K$ and a degree one prime $\fp$ of $F$, we prove that 
$$
L(1,\lambda\nu)\neq 0
$$
for {\it all except finitely many} finite order characters $\nu$ of 
$\Gamma_\fp$ such that $\varepsilon(\lambda\nu)=+1$~(see~Theorem~\ref{gnv}). We also prove an analogous mod $\ell$ non-vanishing for primes $\ell \neq p$ ordinary with respect to the CM quadratic extension $K/F$ {(see~Theorem~\ref{mainthm})}. 
The latter leads to completion of Hsieh's proof of Eisenstein congruence divisibility towards the CM main conjecture (see~ Theorem~\ref{mc}).

Our approach is rooted in the arithmetic of a CM modular form on Shimura set arising from a totally definite quaternion algebra over $F$ (cf.~\cite{nv:CM}). It relies on the arithmetic of the CM field, automorphic representation theory, the Waldspurger formula and Ratner's ergodicity of unipotent flows ~\cite{Ra}. 
To link arithmetic of the Shimura set with that of the CM field, a key is 
an $\ell$-integral comparison of quaternionic and CM periods (see~Theorem~\ref{thmE}), which maybe of independent interest.

\subsection{Main results}\label{in1}
\subsubsection{Set-up} 
Let $K$ be a CM field. Let $F$ be its maximal totally real subfield, $\eta_{K/F}$ the associated quadratic character over $F$ and $c\in \Gal(K/F)$ the non-trivial element. Let $D_F$ be the absolute discriminant of $F$ and $\BA$ the ring of \adeles of $F$.

Let $\lambda$ be a self-dual Hecke over $K$, i.e. 
\begin{equation}\label{theta}\lambda|_{\BA^\times}\cdot |\ |_{\BA^\times}=\eta_{K/F}.\end{equation}
Note that 
{$\lambda \cdot |\ |_{\BA_K}^{1/2}$} is conjugate symplectic self-dual in the sense of \cite{GGP}.
The infinity type of $\lambda$ is of the form $\Sigma+\kappa(1-c)$ for 
$\Sigma$ a CM type of $K$ and 
$\kappa=\sum_{\sigma\in \Sigma}\kappa_\sigma\cdot\sigma\in \BZ_{\geq 0}[\Sigma]$, i.e. $$\lambda_\infty(z)=\prod_{\sigma\in \Sigma}\sigma(z)^{-(1+\kappa_\sigma)}\sigma\circ c(z)^{\kappa_\sigma}$$ for $z\in K^\times$. 
Let $L(s,\lambda)$ be the associated Hecke $L$-function as in Tate's thesis~\cite{Ta}, and so with the center of the functional equation at $s=1$. Let $\Cond(\lambda)$ denote the conductor of $\lambda$. 

Fix an algebraic closure $\ov{\BQ}$ of $\BQ$ and an embedding $\iota_\infty: \ov{\BQ}\hookrightarrow \BC$.

Let $\fp$ be a prime of $F$ above a rational prime $p$.
Let $K_{\fp^\infty}$ be the maximal $\BZ_{p}$-free anticyclotomic extension of $K$ unramified outside $\fp$ and $\Gamma_{\fp}=\Gal(K_{\fp^\infty}/K)$. Note that $\Gamma_\fp \simeq \BZ_p^{\deg \fp}$.
Let $\Xi_\fp$ be the set of finite order $\ov{\BQ}^\times$-valued characters of $\Gamma_{\fp}$ and $\Xi_{\lambda,\fp}^+\subset \Xi_{\fp}$ the subset given by
$$
\Xi_{\lambda,\fp}^+ = \{ \nu \in \Xi_\fp\ |\ \varepsilon(\lambda\nu)=+1\},
$$
where $\varepsilon(\lambda\nu)$ denotes the global epsilon factor.

A basic problem:
 \begin{equation}\label{Q1}\tag{Q1}\text{How often are the central $L$-values $L(1,\lambda\nu)$ non-zero as $\nu \in \Xi_{\lambda,\fp}^+$ varies?}\end{equation}
The literature on related problems \cite{Hi1,Hi2,Hi3,Hi4,Vatsal_Cornut:Documenta,Vatsal_Cornut:London} suggests that such a non-vanishing holds for a Zariski dense\footnote{Let $\{\gamma_1, \cdots, \gamma_{\deg\fp}\}$ be a set of topological generators of $\Gamma_\fp$.
Regard the set of continuous characters $\Hom(\Gamma_\fp, \mu_{p^\infty})$ as a subset of $\BG_{m}^{\deg \fp}(\ov{\BQ}_p)$ by sending a
character $\nu$ to $$(\nu(\gamma_1), \cdots, \nu(\gamma_{\deg\fp})) \in \mu_{p^\infty}^{\deg \fp}(\ov{\BQ}_p) \subset \BG_m^{\deg \fp}(\ov{\BQ}_p).$$  A subset $X$ of $\Hom(\Gamma_\fp , \mu_{p^\infty})$ is said to be
Zariski dense if it is so in $\BG_{m}^{\deg\fp}$ over $\ov{\BQ}_p$.
} 
subset of characters $\nu$ of $\Gamma_\fp$. Our study suggests that a finer non-vanishing holds (cf.~Theorem~\ref{gnv}). 
 
One may also seek mod $\ell$ analogue of this non-vanishing problem. 
To be precise, 
let $\ell\neq p$ be another rational prime. Fix an algebraic closure $\ov{\BQ}_\ell$ and  an embedding $\iota_\ell: \ov{\BQ}\hookrightarrow \ov{\BQ}_\ell$. Let ${v}_{\ell}: \ov{\BQ}_{\ell}\ra \BQ$ be the valuation such that ${v}_{\ell}(\ell)=1$. 

In this paper we assume that $\ell$ is an ordinary prime for the extension $K/F$, i.e. 
\begin{equation}\label{ord}\tag{ord}
\text{Each prime of $F$ above $\ell$ splits in $K$.} 
\end{equation} 
Then there exists an $\ell$-adic CM type $\Sigma$ of $K$, i.e. 
$$
\Sigma_\ell \cap \Sigma_\ell^c = \emptyset
$$
for $\Sigma_\ell:=\{\iota_\ell\circ\sigma\ |\ \sigma\in \Sigma \}$ and 
$\Sigma_\ell^c :=\{\iota_\ell\circ\sigma\circ c\ |\ \sigma\in \Sigma \}$.
In the following we first introduce the algebraic part $L^{\alg}(1,\lambda\nu)\in \ov{\BQ}$ of the central $L$-value $L(1,\lambda\nu)$.

Let $\Omega_\infty=(\Omega_\sigma)_{\sigma\in \Sigma}\in (\BC^\times)^\Sigma$ be the CM period associated to a \Neron differential on an abelian scheme  of CM type $(K,\Sigma)$ over $\ov{\BZ}_{(\ell)}$ (cf.~\cite{Kz,Hi1}).  
For a prime $w| {\rm Cond}(\lambda)$ of $K$, let $G(\lambda_w)$ be the Gauss sum \[G(\lambda_w):=|\varpi|_w^{e}\cdot \sum_{u\in\varpi^{-e}\cdot(O_{K_w}/\varpi^{e}O_{K_w})^\times}\lambda_w(u)\psi_w(u),\] 
where $\varpi$ is a uniformiser of $K_w$, $|\cdot|_w$ the valuation on $K_w^\times$ such that $|\varpi|_{w}^{-1}=\#\kappa_w$ for $\kappa_w$ the residue field of $K_w$, $e$ the {exponential} conductor of $\lambda_w$ and $\psi_w$ a non-trivial additive character of $K_w$ with trivial conductor. Put $$\Gamma_{\Sigma}(\sum_{\sigma\in \Sigma} n_\sigma \sigma)=\prod_{\sigma\in \Sigma}\Gamma(n_\sigma), \qquad \Omega_\infty^{\sum_{\sigma\in \Sigma} n_\sigma \sigma}=\prod_{\sigma\in \Sigma}\Omega_\sigma^{n_\sigma}.$$ 
By the work of Shimura, Katz \cite{Kz} and Hida--Tilouine \cite{HT'}, we {have}
\begin{equation}\label{eq:alg}
L^{{\rm alg}}(1,\lambda\nu):=[O_K^\times:O^\times] \prod_{w\in \Sigma_\ell }G((\lambda\nu)_w) \cdot 
\frac{\pi^\kappa\Gamma_\Sigma(\Sigma+\kappa) \cdot L(1,\lambda\nu)}{\Omega_\infty^{\Sigma+2\kappa}}\in \ov{\BZ}_{(\ell)}.
\end{equation}

The underlying mod $\ell$ non-vanishing {problem}:  
 \begin{equation}\tag{Q2}\label{Q2}\text{How does 
    ${v}_\ell\left(L^{\alg}(1,\lambda\nu)\right)$ vary with $\nu\in \Xi_{\lambda,\fp}^+$?}\end{equation}
The following invariant will appear in our result towards it: 
\[\mu_{\ell}(\lambda)=\sum_{\substack{v|{\rm{N}}_{K/F}(\Cond(\lambda)) \text{ {non-split}}\\v\nmid \fp}}\mu_{\ell}(\lambda_v), \]where $\mu_\ell(\lambda_v):=\inf_{x\in K_v^\times}{v}_{\ell}(\lambda_v(x)-1)$ and `non-split' refers to being non-split in $K$. This invariant is closely related to the $\ell$-part of the Tamagawa number associated to $\lambda$. 
 Typically, we have  $${v}_\ell \left(L^{\alg}(1,\lambda\nu)\right) \geq \mu_{\ell}(\lambda)$$ as per the Bloch--Kato conjecture (cf.~\cite{Fin1,Hsieh:nv,He}).

\subsubsection{} Our main result towards the problem \eqref{Q1} is the following {(cf.~Corollary~\ref{maincor})}.

\begin{thm}\label{gnv}
Let $K$ be a CM field and $F$ its maximal totally real subfield. 
Let $\lambda$ be a self-dual Hecke character over $K$ and  
$\fp$ a {degree one} prime of $F$ unramified in $K$.
Then 
\[L(1,\lambda\nu)\neq 0 \]
for all except finitely many $\nu \in\Xi_{\lambda,\fp}^+$.
\end{thm}
An application to Mordell--Weil ranks: 
\begin{cor}\label{cor:MW} Let $\lambda$ be a self-dual Hecke character over a CM field $K$ of infinity type $\Sigma$, which is a CM type of $K$. Let $A_\lambda$ be an associated CM abelian variety over $K$. 
Let $\fp$ be a degree one prime of $F$ split in $K$. Suppose that $\varepsilon(\lambda)=+1$. {Then} the Mordell--Weil group $A_\lambda(K_{\fp^\infty})$ is finitely generated.
\end{cor}
\begin{proof}
This is a consequence of Theorem~\ref{gnv} and \cite[Thm.~B']{Ne1}. 
\end{proof}
\begin{remark}  If $\fp$ is inert in $K$, then a variant of Corollary~\ref{cor:MW} also holds, whose formulation we leave to the interested reader.

\end{remark}

\subsubsection{Mod $\ell$ non-vanishing}
Our main result towards the problem \eqref{Q2} is the following (cf.~Theorem~\ref{mm}). 
\begin{thm}\label{mainthm} Let $K$ be a CM field and $F$ its maximal totally real subfield. 
Let $\lambda$ be a self-dual Hecke character over $K$ 
of infinity type $\Sigma+\kappa(1-c)$ for $\Sigma$ a CM type of $K$ and 
$\kappa\in \BZ_{\geq 0}[\Sigma]$. 
Let $\fp$ be a degree one prime of $F$
such that $\fp$ is {either coprime to ${\rm N}_{K/F}(\Cond(\lambda))$ or split in $K$.} 
Let $\ell$ be a prime such that $(\ell, 6\fp D_F)=1$ and $\Sigma$ is an $\ell$-adic CM type. Then  
$$
 {v}_\ell\left(L^{\rm alg}(1,\lambda\nu)\right)=\mu_{\ell}(\lambda)
 $$for all except finitely many $\nu \in\Xi_{\lambda,\fp}^+$.
\end{thm} 
\begin{remark}
\
\begin{itemize} 
\item[i)] 
In view of the BSD conjecture (cf.~\cite{BF}), the above mod $\ell$ non-vanishing would imply that 
the $\ell$-part of Tate--Shafarevich groups 
in the $\fp$-anticyclotomic family is bounded (cf.~\cite[Cor.~1.2]{nv:CM}). 
\item[ii)] We expect to remove the hypotheses of Theorem~\ref{mainthm} via  refinement of the strategy elsewhere. 
\end{itemize}

\end{remark}

\subsection{Applications to CM Iwasawa theory}
 
\subsubsection{Iwasawa main conjecture for CM fields} \label{mc:CM}
Let $K/F$ be a CM quadratic extension and 
$h_K^-$ the relative class number. 
Let $p$ be an ordinary odd prime. 

Fix an algebraic closure $\ov{\BQ}_p$ and an embedding 
$\iota_p: \ov{\BQ}\hookrightarrow \ov{\BQ}_{p}$.
Let $\Sigma$ be a $p$-ordinary CM type of $K$, and $\Sigma_p$ the associated $p$-adic CM type via the embedding $\iota_p$. For a number field $L$, let $\omega_L:\Gal(\ov{L}/L)\ra \BZ_p^\times$ denote the associated Teichm\"uller character.

Let $K_\infty$ be the compositum of the cyclotomic $\BZ_p$-extension and the anticyclotomic $\BZ_p^{[F:\BQ]}$-extension of $K$. Put $\Gamma_K=\Gal(K_\infty/K)$. Let $K'$ be a finite abelian extension of {$K$ that contains $K(\mu_p)$ and} is disjoint with $K_\infty$. Put $\Delta=\Gal(K'/K)$ and $K_\infty'=K_\infty K'$. Let $\psi:\Delta\ra \ov{\BZ}_p^\times$ be a finite order Hecke character over $K$.  
Let $R=W(\ov{\BF}_p)[\psi]$ and $\Lambda=R[\!|\Gamma_K|\!]$.  

Let $M_\Sigma$ be the maximal $p$-abelian $\Sigma_p$-ramified extension
of $K_\infty'$. A basic Iwasawa-theoretic object: $$X_\Sigma:=\Gal(M_\Sigma/K'_\infty)\otimes_{\BZ_p[\Delta][\![\Gamma_K]\!]} R[\Delta][\![\Gamma_K]\!].$$ 
Define $X_\Sigma^{(\psi)}$ as the maximal $\psi$-isotypic quotient of $X_\Sigma$. It is a finitely generated torsion $\Lambda$-module and let $F_{\Sigma}(\psi)\in \Lambda$ denote its characteristic power series.

On the analytic side, we have the associated Katz $p$-adic $L$-function
$$L_{\Sigma}(\psi)\in \Lambda$$ interpolating 
algebraic part of critical Hecke $L$-values associated to twists of $\psi$ by certain $p$-adic characters of $\Gamma_K$ (cf.~\cite{Kz,HT}).

Based on the work of Hsieh \cite{Hs2}, a key arithmetic consequence of Theorem~\ref{mainthm} to the arithmetic of $K$ is the following divisibility towards the CM Iwasawa main conjecture (cf.~\cite[p.~90]{HT'}).

\begin{thm}\label{mc}
Let $K/F$ be a $p$-ordinary CM quadratic extension and $\Sigma_p$ a $p$-adic CM type. 
For $K'$ a finite extension of $K$ as above, let $\psi:\Delta:=\Gal(K'/K) \ra \ov{\BZ}_p^\times$ be a character.
Suppose that
  \begin{itemize}
    \item [\tiny{$\bullet$}] $p\nmid 6 h_K^- D_F \# \Delta$,
    \item [\tiny{$\bullet$}]$\psi$ is unramified at $\Sigma_p^c$ and $\psi\omega_K^{-a}$ is unramified at $\Sigma_p$ for some integer $a\nequiv 2\pmod{p-1}$.
  \end{itemize}
Then 
 \[L_{\Sigma}(\psi) \big{|} F_{\Sigma}(\psi).\]  
\end{thm} 
\begin{remark}\noindent
\begin{itemize}
    \item[i)]  In view of the non-vanishing of Hecke $L$-values  in Theorem~\ref{mainthm}, it may be possible to remove the hypothesis $p\nmid \# \Delta$ (cf.~Remark~\ref{rmk:hyp}).
    \item[ii)] If $F=\BQ$, then the CM main conjecture is a theorem of Rubin \cite{Ru91}, whose proof does not rely on the anticyclotomic non-vanishing of Hecke $L$-values. 
\end{itemize}
   
\end{remark}

One has the following application to the BSD conjecture (cf.~\cite[Cor.~1]{Hs2}).
\begin{cor}
 Let $E$ be a CM elliptic curve defined over a totally real field $F$ with CM by an order of an imaginary quadratic field $M$. Let $K=FM$ and $p\nmid 6h_K^- D_F$ be a prime such that $E$ has good ordinary reduction at all primes above $p$. 
 \begin{itemize}
 \item[a)] If the central $L$-value $L(E/F,1)$ vanishes, then the $p^\infty$-Selmer group $\Sel_{p^\infty}(E/F)$ is infinite.
 \item[b)] If $L(E/F,1)\neq 0$, then 
 $$
 {\rm length}_{\BZ_p}\Sel_{p^\infty}(E/F) \geq {v}_{p}\bigg{(}\frac{L(E/F,1)}{\Omega_E}\bigg{)}
 $$
 where $\Omega_E$ is the N\'eron period. 
 \end{itemize}
\end{cor}

Theorem~\ref{mc} is asserted in \cite{Hs2}. 
The approach of {\it loc. cit.} is based on the method of Eisenstein congruence on the unitary group $U(2,1)$ over $K$. It begins with  construction of a $\Lambda$-adic Eisenstein series $\CE(\psi)$ on $U(2,1)$ whose constant term is closely related to the Katz $p$-adic $L$-function $L_{\Sigma}(\psi)$. A key is to show $p$-primitivity of the Eisenstein series $\CE(\psi)$. In {\it loc. cit.} it is approached by showing $p$-indivisibility of a certain Fourier--Jacobi coefficient of $\CE(\psi)$. Since a toric period of the Fourier--Jacobi coefficient is essentially a product of two Hecke $L$-values, the $p$-indivisibility results of \cite{Hi1,Hi2,Hsieh:nv} are invoked in the case $\deg \fp=1$ to show {\it simultaneous} $p$-indivisibility of these Hecke $L$-values. Due to the subsequently found gap in Hida's strategy and consequent weakening \cite{Hi4} of the $p$-indivisibility results of \cite{Hi1,Hi2,Hsieh:nv}  in the case $\deg \fp=1$, Hsieh's proof of the $p$-primitivity of $\CE(\psi)$ as in \cite[\S7.6]{Hs2} is incomplete (see also the introduction of \cite{Hi4}). 

As detailed in the main text, our Theorem~\ref{mainthm} fixes this gap and leads to completion of the proof of the main result of \cite{Hs2}  (cf.~\S \ref{s:IMC}). 

\subsection{Strategy}
The non-vanishing problems \eqref{Q1} and \eqref{Q2} concerns the arithmetic of a self-dual Hecke character or that of the unitary group $U(1)$ over $K$. 
Via theta correspondence, we recast them as  problems on a totally definite unitary group $U(2)$ over $K$. The latter turn out to be amenable to automorphic and ergodic tools.

Some of the following notation differs from the rest of the paper.
\subsubsection{Shimura set} 
For a CM field $K$ and $F$ the maximal totally real subfield, we have a canonically associated totally definite quaternion algebra $B$ over $F$ characterised by  \[\varepsilon(B_v)=\eta_{K_v/F_v}(-1)\] 
for any place $v$ of $F$. 
Note that $K$ embeds into $B$ as an $F$-algebra, and we fix such an embedding.

In our strategy $B$ is foundational to the arithmetic of $K$ 
in the guise of associated Shimura set 
(cf.~\cite{Ti,CST2,TYZ,nv:CM}).

\subsubsection{Ancillary results} 
{Our mod $\ell$ non-vanishing of Hecke $L$-values is based on Theorems~\ref{thmD} and~\ref{thmE} below, which concern the arithmetic of the Shimura set.}

For a self-dual Hecke character $\lambda$ over $K$ of infinity type $\Sigma+\kappa(1-c)$, let 
$\phi_\lambda$ be the associated $\GL_{2}(F)$-theta series of weight $k:=2\Sigma+2\kappa$ and level $U_0(\mathfrak{N}_{\lambda})$ for $\mathfrak{N}_{\lambda}:=D_{K/F}{\rm N}_{K/F}(\Cond(\lambda))$. Write $\mathfrak{N}_{\lambda}=\mathfrak{N}_{\lambda}^+\mathfrak{N}_{\lambda}^-$ with $\mathfrak{N}_{\lambda}^+$ (resp. $\mathfrak{N}_{\lambda}^-$) only divisible by split (resp.~non-split) primes \dfn{of $F$} in $K$.
Let 
$\pi_\lambda$ be the Jacquet--Langlands transfer of the automorphic representation generated by $\phi_\lambda$ to $B^\times$. 
Let $\ell$ be a prime such that 
$(\ell, 6\fp D_K)=1$. 
Let {$$f_{\lambda}\in \pi_\lambda^{\wh{R}^\times}$$} be 
the CM form as in subsection \ref{ss:ntv}. 
{It is a test vector in the sense of Gross and Prasad \cite{GP}, which is $\ell$-primitive and $K_v^\times$-invariant for all places $v|\mathfrak{N}_{\lambda}^-$.} 
We emphasise that $f_\lambda$ is {not} a newform\footnote{In this setting the finite part of the discriminant of $B$ divides $D_{K/F}$. Consequently, at places dividing $D_{K/F}$, the newform is not a test vector under an optimal embedding $K\hookrightarrow B$.}.  

We have a factorisation
\begin{equation}\label{$L$-fac'}
L(1/2,\pi_{\lambda,K}\otimes\chi_0\nu)=L(1,\lambda\chi_0\nu)L(1,\lambda(\chi_0\nu)^{-1})
\end{equation}
of the Rankin--Selberg central $L$-value $L(1/2,\pi_{\lambda,K}\otimes\chi_0\nu)$, where $\nu\in\Xi_\fp$ and $\chi_0$ is an auxiliary finite order anticyclotomic character over $K$ such that 
$\chi_0|_{\prod_{v|\mathfrak{N}_{\lambda}^-}K_v^\times}=1$. We recast the non-vanishing problems for Hecke $L$-values in terms of these Rankin--Selberg $L$-values. 

The Rankin--Selberg $L$-values $L(1/2,\pi_{\lambda,K}\otimes\chi_0\nu)$ maybe studied via the Waldspurger formula.  
While employing it, the period 
\[\Omega_{\lambda}:=\frac{\pi^k(\phi_\lambda,\phi_\lambda)_{U_0(\mathfrak{N}_{\lambda})}}{\Gamma_{\Sigma}(k)\pair{f_\lambda,f_\lambda}_{\wh{R}^\times}}
  \]
appears naturally,  where 
$(\phi_{\lambda},\phi_{\lambda} )_{U_0(\mathfrak{N}_{\lambda})}$ and $\pair{ f_{\lambda}, f_{\lambda}}_{\wh{R}^\times}$ are Petersson norms. 
In view of the explicit Waldspurger formula \cite{CST} the normalised $L$-value $L(1/2,\pi_{\lambda,K}\otimes \chi_0\nu)/\Omega_{\lambda}$ is $\ell$-integral.

A key is the following $(\ell,p)$ non-vanishing (cf.~Theorems~\ref{T:b.W}~and~\ref{T:3.W}).

\begin{thm}\label{thmD}  
  Let $\lambda$ be a self-dual Hecke character over $K$  
  and $\pi_\lambda$ the associated cuspidal automorphic representation of conductor $\mathfrak{N}_{\lambda}$. Let $\ell$ be a prime such that $(\ell, 6D_{K})=1$ and $\fp\nmid \ell\mathfrak{N}_{\lambda}^-$ {a degree one prime of $F$}. Let $\chi_0$ be a finite order anticyclotomic character such that the condition \eqref{Ass(D)} holds. Then for all except finitely many $\nu\in\Xi_{\lambda\chi_0,\fp}^{+}$, we have 
  \[{v}_{\ell}\left(\frac{ L(1/2, \pi_{\lambda, K}\otimes\chi_{0}\nu)}{\Omega_{\lambda}^{}}\right)=0.\]
  \end{thm}
As for the period $\Omega_\lambda$, note that 
$$\pi^{2\kappa}
\Omega_{\lambda}/\Omega_\infty^{2\Sigma+4\kappa}\in\ov{\BQ}^\times 
$$
since
the normalised $L$-values $\pi^\kappa L(1,\lambda\chi_0\nu)/\Omega_\infty^{\Sigma+2\kappa}$ and $\pi^\kappa L(1,\lambda(\chi_0\nu)^{-1})/\Omega_\infty^{\Sigma+2\kappa}$ are algebraic and non-zero for some $\nu\in \Xi_{\lambda\chi_0,\fp}^{+}$. The $\ell$-adic valuation of this ratio of periods is given by the following (cf.~Theorem~\ref{cpp}).
\begin{thm} \label{thmE}
  Let $\lambda$ be a self-dual Hecke character over a CM field $K$ of infinity type $\Sigma+\kappa(1-c)$ for $\Sigma$ a CM type of $K$ and $\kappa\in \BZ_{\geq 0}[\Sigma]$. Let {$\ell\nmid 6D_F$} be a prime such that $\Sigma$ is an $\ell$-adic CM type. 
{Then} 
\[{v}_{\ell}\left(\frac{\pi^{2\kappa}\Gamma_\Sigma(\Sigma+\kappa)^2 \cdot \Omega_{\lambda}^{}}{\Omega_\infty^{2\Sigma+4\kappa}}\prod_{w\in \Sigma_\ell }G(\lambda_w)^2\right)=2\sum_{\text{$v|{\rm{N}}_{K/F}(\Cond(\lambda))$ inert}}\mu_\ell(\lambda_v).\]
\end{thm}

The above two results\footnote{taking $\chi_0=1$ in Theorem~\ref{thmD}} yield Theorem~\ref{mainthm} 
in light of the factorisation \eqref{$L$-fac'}
of $L$-values, and the lower bound for $\ell$-adic valuation of Hecke $L$-values as in ~Theorem~\ref{lm:lb}.

\begin{remark}
The above strategy also leads to determination of the $\mu$-invariant of anticyclotomic Katz $p$-adic $L$-function associated to a self-dual Hecke character  (cf.~\cite[Thms.~1.3~$\&$~1.5]{nv:CM}), giving a different proof of results of \cite{Hi3,Hs1}. 
\end{remark}
\subsubsection{About Theorem~\ref{thmD}}\label{ss:lp-nvs}
This mod $\ell$ non-vanishing of the Rankin--Selberg $L$-values $L^{\alg}(1/2,\pi_{\lambda,K}\otimes\chi_0\nu):=L(1/2,\pi_{\lambda,K}\otimes\chi_0\nu)/\Omega_{\lambda}$ is based on explicit Waldspurger formula, Fourier analysis on ring class groups in conjunction with automorphic representation theory  and equidistribution of special points. 
{For simplicity of notation, 
we describe the case of $\lambda$ with infinity type $\Sigma$.}

Let $X_{U}$ be the Shimura set associated to $B^\times$ and an open subgroup $U=\wh{R}^\times\subset \wh{B}^\times$ corresponding to the test vector $f_\lambda$. Let $\fc$ denote the conductor of $\chi_0$, which is coprime to $\fp$. An apt choice of embedding $\iota: K\hookrightarrow B$ leads to special points $$x_{\fc,n}(a)\in X_U$$ for $[a]\in {{\rm G}_{n,\fc}}:=\Gal(H_{\fc\fp^n}/K)$ and $H_{\fm}$ the ring class field of conductor {$\fm$}. {Based on  explicit Waldspurger formula of Cai--Shu--Tian (cf.~Theorem~\ref{T:central.W}) and test vector theory (cf.~\S\ref{test}), the mod $\ell$ non-vanishing of 
$L^{\alg}(1/2,\pi_{\lambda,K}\otimes\chi_0\nu)$ is equivalent to that of toric periods  
 \[P_{f_\lambda}(\nu):=\sum_{[a]\in {{\rm G}_{\fc,n}}}\nu\chi_0(a)f_\lambda(x_{\fc,n}(a)),\]
where $\nu$ factors through {${\rm G}_{\fc, n}$} {(cf.~\S\ref{ss:lp})} . The generality of the work of Cai--Shu--Tian \cite{CST} is essential in our study since  
$D_{K/F}$ divides the conductor of $\pi_\lambda$ and $B$ is ramified at primes dividing $D_{K/F}$.

We study $\ell$-indivisibility of the toric periods $P_{f_{\lambda}}(\nu)$ via automorphic representation theory, Fourier analysis on the ring class groups ${\rm G}_{\fc,n}$ and equidistribution of images of special points $x_{\fc,n}(a)$ in a self-product of the Shimura set $X_U$. The equidistribution is a consequence of Ratner's seminal ergodicity of unipotent flows ~\cite{Ra}. The relevance of equidisitribution to mod $\ell$ non-vanishing is an insight of Vatsal \cite{Vatsal:nonvanishing}.   {His strategy was generalised to totally real fields by Cornut--Vatsal \cite{Vatsal_Cornut:London}. However, their results essentially excludes the CM case, referred to as exceptional therein.} To approach the CM case, we refine the strategy by considering toric periods over Galois orbits and by employing {automorphic} representation theory peculiar to the CM case (see more below). It then suffices to show that $f_{\lambda} \mod{\ell}$ is non-Eisenstein on specific components of $X_U$, which encode the $\varepsilon$-factor {$\varepsilon(\lambda)$}. 
The non-Eisenstein arguments in the work of Vatsal \cite{Vatsal:nonvanishing} and Cornut--Vatsal \cite{Vatsal_Cornut:London} do not apply to our CM setting.

In fact, a new phenomenon happens. One has a partition 
$
X_U=X_U^+ \sqcup X_U^-
$
where {$$X_{U}^+:=\{[h]\in X_{U}\ \Big|\ \eta_{K/F}\circ {\rm N}(h)=+1\}$$ with ${\rm N}$ the reduced norm. }
{For} $\fp$ inert in $K$, $f_\lambda \mod{\ell}$ is non-Eisenstein on {\it exactly one} of the components $X_U^\pm$ depending on $\varepsilon(\lambda)$, unlike the work of Vatsal and Cornut--Vatsal where the test vector is non-Eisenstein on all the components. Our non-Eisenstein argument is indirect: $f_\lambda \mod{\ell}$ is non-zero by definition, and consequently non-Eisenstein on $X_U$ (see Lemma~\ref{ee} which is specific to the CM setting). So it is non-Eisenstein on at least one of the components $X_U^{\tilde{\varepsilon}}\in\{X_{U}^{+},X_{U}^{-}\}$. As the non-vanishing strategy applies on $X_U^{\tilde{\varepsilon}}$ and $L(1/2,\pi_{\lambda,K}\otimes \chi_0\nu)=0$ for $\nu\in\Xi_{\lambda\chi_0,\fp}^-$, we deduce that
$\tilde{\varepsilon}$ has the desired parity. Since the test vector $f_\lambda$ is independent of $\fp$, the $\fp$ inert case allows us to treat\footnote{See~\S\ref{S:restest}.} the $\fp$ split case! }\

We now elaborate uniform non-vanishing of the toric periods $P_{f_\lambda}(\nu)$ (cf.~\S\ref{deg1}). As $\nu$ varies in a Galois orbit, the average of toric periods $P_{f_\lambda}(\nu)$ leads to a partial toric period over a thin subgroup ${\rm G}_{\fc,n}' \subset {\rm G}_{\fc,n}$ for $n_{}:=\ord_\fp\cond(\nu)$, whose cardinality is independent of $n$ since $\deg \fp=1$ (see~\eqref{avered} and Remark~\ref{geq1}). We analyse the period over ${\rm G}_{\fc,n}'$ by partitioning ${\rm G}_{\fc,n}'$ into  a locally defined subgroup ${\rm H}_{\fc,n}$ and $D_1\subset \wh{K}^\times$, the latter representing ${\rm G}_{\fc,n}'/{\rm H}_{\fc,n}$. The corresponding periods are studied via local representation theory and ergodic theory respectively, as outlined below.

 We first introduce a partition of $\Xi_\fp$ into finitely many subfamilies $\CS$ (see the discussion~above~\eqref{tv-tw}) and a Hecke-modification $f_{\CS}$  of the test vector $f_\lambda$ as in~\eqref{tv-tw} such that the ${\rm H}_{\fc,n}$-action has the following uniform eigen-behaviour:

\begin{itemize}
\item[\tiny{$\bullet$}] $f_{\CS}(x_{\fc,n}(\cdot))$ is $(\chi_0\nu)^{-1}$-eigen under the ${\rm H}_{\fc,n}$-action for each $n$ with $\nu\in\CS$ and $n=\ord_\fp\cond(\nu)$. 
\end{itemize}
{By the Tunnell--Saito theorem and its mod $\ell$ variant}, $f_{\CS}$ is still a test vector for the subfamily $\CS$ and has the same $\ell$-divisibility as $f_\lambda$ (cf.~Lemmas~\ref{locavetest}~and~\ref{tnn}). In view of the eigen-property, the toric period of $f_\lambda$ equals that of $f_{\mathcal{S}}$, and so in our strategy the former maybe replaced by the latter. A key feature of $f_{\CS}$: the associated toric period over ${\rm G}_{\fc,n}'$ factors through ${\rm G}_{\fc,n}'/{\rm H}_{\fc,n}$ .
Hence, its non-vanishing can be studied via the equidistribution of skewed-diagonal images of special points on the 
$\#D_1$-fold self-product of the Shimura set $X_U$, leading to uniform non-vanishing of the toric periods over the subfamily $\CS$  (cf.~Lemma~\ref{ind}~and~Proposition~\ref{C:Vatsal_Cornut}). Since there are finitely many subfamilies, Theorem~\ref{thmD} follows.

\subsubsection{About Theorem~\ref{thmE}}\label{ss:per} This $\ell$-integral period relation compares automorphic and motivic periods, which is a basic problem (cf.~\cite{Ha,PW,Pr1,IP}). 

For weight two elliptic newforms with square-free conductor, such a comparison of periods is a consequence of Ribet's seminal level raising (cf.~\cite{PW,Pr1}). For elliptic newforms with weights in the Fontaine--Laffaille range and square-free conductors, it may also be approached via `$R=T$' theorems (cf.~\cite{CH,KO,BKM}). These methods apply to Hecke eigenforms, and require additional hypotheses over totally real fields. 
However, our CM setting is neither semistable nor does it involve an eigenform on the Shimura set, and the weights are not necessarily in the Fontaine--Laffaille range\footnote{As far as we know, our result is the first result in the non Fontaine--Lafaille range.}.  It is also excluded by the conjectures of Ichino and Prasanna \cite{Pr1,IP} which pertain to the arithmetic of Petersson norms under the Jacquet--Langlands correspondence.

Our roundabout strategy is based on a tenuous link between the $\ell$-adic valuation of $\pi^{2\kappa}
\Omega_{\lambda}/\Omega_\infty^{2\Sigma+4\kappa}$ and mod $\ell$ non-vanishing as in Theorems~\ref{thmD}~and~\ref{mainthm}. 
The key: there exist  
\begin{itemize}
  \item [\tiny{$\bullet$}] an auxiliary degree one prime $\fq\mid 2\ell{\rm N}_{K/F}(\Cond(\lambda))$ of $F$ inert in $K$ and a Hecke character $\chi_0\in \Xi_{\lambda,\fq}^+$ satisfying the condition \eqref{Ass(D)} {so that $\varepsilon(\lambda\chi_0)=\varepsilon(\lambda\chi_0^{-1})=+1$}, and 
  \item [\tiny{$\bullet$}] an auxiliary degree one
 prime $\fq'\nmid \fq 2\ell{\rm N}_{K/F}(\Cond(\lambda))$ of $F$ split in $K$ with ${\rm N}_{F/\BQ}(\fq')\gg 0$ {and} two characters $\nu_0,\nu\in\Xi_{\lambda,\fq'}^{+}$ such that \end{itemize}  
\begin{equation}\label{nv-cmp}
{v}_{\ell}\left(L^{\alg}(1,\lambda\nu_0(\chi_0\nu)) \cdot
L^{\alg}(1,\lambda\nu_0(\chi_0\nu)^{-1})\right)=2\mu_{\ell}(\lambda).
\end{equation}
Consequently, Theorem \ref{thmE} for the base character $\lambda\nu_0$, and hence 
for $\lambda$ holds. 

We now elaborate the strategy.

\subsubsection*{{Preliminary mod $\ell$ non-vanishing \eqref{nv-cmp}}}

If $\varepsilon(\lambda)=+1$, then we {may} take $\chi_0=1$ and {simply} require no condition on $\fq$.

Since $\varepsilon(\lambda\chi_0)=\varepsilon(\lambda\chi_0^{-1})=+1$, the non-vanishing \eqref{nv-cmp} holds by a recent result of the second-named author \cite{He} based on Hida's strategy (cf.~\cite{Hsieh:nv}). 
More precisely, for any prime 
$\fq'\nmid 2\ell$ of $F$ split in $K$, 
\cite[Thm.~1.5]{He} for Hecke characters $\lambda\chi_0$ and $\lambda\chi_0^{-1}$ provides the existence of $\nu_1,\nu_2\in \Xi_{\fq'}$ such that the non-vanishing \eqref{nv-cmp} holds for $$\nu_0:=\sqrt{\nu_1}\sqrt{\nu_2}, \quad \nu:=\sqrt{\nu_1}\sqrt{\nu_2}^{-1},$$ where $\sqrt{\nu_i} \in \Xi_{\fq'}$ such that $\sqrt{\nu_i}^2=\nu_i$.

    \subsubsection*{{Theorem \ref{thmE} for $\lambda\nu_{0}$}}

We first employ: 
\begin{itemize}
  \item [(a1)] An $\ell$-integrality of  Rankin--Selberg $L$-values: for $\nu\in \Xi_{\lambda\chi_0,\fq'}^+$ 
  we have \[{v}_{\ell}\left(\frac{ L(1/2, \pi_{\lambda\nu_0, K}\otimes\chi_{0}\nu)}{\Omega_{\lambda}^{}}\right)\geq 0 \]  
  if $\ord_{\fq'}\mathfrak{N}_{\lambda}< \ord_{\fq'}\cond(\nu_{0,\fq'})\leq \ord_{\fq'}\cond(\nu_{\fq'})$, {where} 
  $\nu_{\fq'}:=\nu|_{K_{\fq'}^\times}$ and $\cond(\nu_{\fq'})$ is the maximal integral ideal $(\fq')^k$ of $O_{\fq'}$ such that {$\nu_{\fq'}((O_{\fq'}+{\fq'}^k O_{K_{\fq'}})^\times)=1$.} 
\end{itemize}
\begin{itemize}
  \item [(a2)] The {preliminary} mod $\ell$ non-vanishing of Hecke $L$-values as in \eqref{nv-cmp}.
\end{itemize}

In view of (a1) and (a2) we have
\begin{equation}\label{com1'}{v}_{\ell}\left(\frac{\pi^{2\kappa}\Gamma_\Sigma(\Sigma+\kappa)^2 \cdot   \Omega_{\lambda}^{}}{\Omega_\infty^{2\Sigma+4\kappa}}\prod_{w\in \Sigma_\ell }G(\lambda_w)^2\right)\leq 2\sum_{\text{$v|{\rm{N}}_{K/F}(\Cond(\lambda))$ inert}}\mu_\ell(\lambda_v).\end{equation}

 Next, we employ:
 \begin{itemize}
  \item [(b1)] The mod $\ell$ non-vanishing of Rankin--Selberg $L$-values: Theorem~\ref{thmD} for the base character $\lambda\nu_0$, prime $\fq'$ and $\chi_0\in\Xi_{\lambda,\fq}^+$.
  \item [(b2)] A lower bound for $\ell$-divisibility of {Hecke} $L$-values: for $\nu\in \Xi_{\lambda\chi_0,\fq'}^+$  we have
  \[{v}_{\ell}\left(L^{\alg}(1,\lambda\nu_0(\chi_0\nu)) \cdot 
  L^{\alg}(1,\lambda\nu_0(\chi_0\nu)^{-1})\right)\geq 2\mu_{\ell}(\lambda).\]  
   
 \end{itemize}
In view of (b1) and (b2) note that 
\begin{equation}\label{com2}{v}_{\ell}\left(\frac{\pi^{2\kappa}\Gamma_\Sigma(\Sigma+\kappa)^2 \cdot \Omega_{\lambda}^{}}{\Omega_\infty^{2\Sigma+4\kappa}}\prod_{w\in \Sigma_\ell }G(\lambda_w)^2\right)\geq 2\sum_{\text{$v|{\rm{N}}_{K/F}(\Cond(\lambda))$ inert}}\mu_\ell(\lambda_v).\end{equation}

 Hence,   
 Theorem \ref{thmE} holds for the Hecke character $\lambda\nu_0$.

    \subsubsection*{{Theorem \ref{thmE} for $\lambda$}}

 The equivalence of Theorem \ref{thmE} for Hecke characters $\lambda$ and $\lambda\nu_0$ is a consequence of the equivalence between non-vanishing of the corresponding Hecke $L$-values and comparison of periods as in~Proposition~\ref{cpnv1}.

\begin{remark} This paper maybe viewed as a sequel to \cite{nv:CM}, which considered mod $\ell$ non-vanishing of self-dual Hecke $L$-values over an imaginary quadratic field via arithmetic of a Shimura set. A new feature of the strategy of this paper: the use of an auxiliary Hecke character $\chi_0$. It leads to technical simplifications, for example it bypasses the results of section 7 of {\it loc. cit.} on explicit test vectors for supercuspidal representations.
\end{remark}
\subsection{Vistas} The mod $\ell$ non-vanishing of self-dual Hecke $L$-values has various applications. For example, it is ancillary to the main result of \cite{Lee} on primitivity of Hida family of theta lifts from the unitary group $U(1)$ to a totally definite $U(2)$ over a CM field $K$. 
It is also a key ingredient in the joint work \cite{ABL} of the first-named author with Alonso Rodrigues and Lai on primitivity of the Eisenstein series 
$\CE(\psi)$ on the unitary group $U(2,1)$ over $K$ by a new method. The {\it loc. cit.} considers periods of $\CE(\psi)$ against the subgroup $U(1) \times U(2) \subset U(2,1)$ {instead of Fourier--Jacobi coefficients as in \cite{Hs2}} (cf.~\S\ref{mc:CM}). Such a period against a CM form on $U(2)$ is a product of Hecke $L$-values. This approach may lead to removal of some of the hypotheses from the CM main conjecture divisibility (cf.~Theorem~\ref{mc}). 

In a sequel \cite{BHTY} we will approach the non-vanishing problems \eqref{Q1} and \eqref{Q2} when the prime $\fp$ is of degree greater than one via a refinement of the current strategy\footnote{The current strategy readily yields the non-vanishing for a Zariski dense subset of characters.}. The connection between epsilon factor of a self-dual Hecke character and non-vanishing of associated test vector on a specific component of the Shimura set established in \S\ref{S:restest}
 can be viewed as a global refinement of the local Tunnell--Saito theorem. It suggests a new structure on the endoscopic cases of the global Gan--Gross--Prasad conjecture, which we will also explore.

A basic problem: mod $\ell$ non-vanishing of Hecke $L$-values over a CM field for non-ordinary primes $\ell$. Note that our mod $\ell$ non-vanishing of  Rankin--Selberg Hecke $L$-values 
as in Theorem~\ref{thmD} allows $\ell$ to be non-ordinary. In conjunction with Finis' ideas \cite{Fin1} on the arithmetic of theta functions with complex multiplication, we plan to study this problem.

\subsection{Plan of the paper}
In section \ref{Shimuraset} we recall generalities regarding modular forms and special points on a Shimura set over a totally real field. Then section \ref{exp:Wald} describes an explicit Waldspurger formula connecting  
Rankin--Selberg central $L$-values in an anticyclotomic family and toric periods on the Shimura set, leading to an $\ell$-integrality of these $L$-values. Sections \ref{s:nv} and \ref{s:nv-Hecke} constitute the core of the  paper. In section \ref{s:nv} we prove mod $\ell$ non-vanishing of the Rankin--Selberg central $L$-values in the CM case, via automorphic representation theory and  Ratner's ergodicity of unipotent flows.
Section \ref{s:nv-Hecke} establishes the main result on  mod $\ell$ non-vanishing of self-dual Hecke $L$-values.  In the last section we describe an application to the CM Iwasawa main conjecture.
\subsection{Notation}\label{nota} 
\begin{itemize}
\item Throughout the paper, we fix totally real number field $F$. Let $D_F$, $O$ and $\BA$ respectively denote the absolute discriminant, the ring of integers and the ring of \adeles of $F$.
Let $\BA_f$ denote the ring of finite \adeles of $F$.
{For a finite set of places $S$ of $F$, let $\BA^{(S)}\subset \BA$ denote the adele ring with the factors at $S$ omitted.}
\item We also fix a totally imaginary quadratic extension $K/F$, and analogously define $D_K$, $O_K$ and $\BA_K$. Let $D_{K/F}$ denote the relative discriminant of $K/F$.

\item  For an integral prime $\fa$ of $F$, write $\fa=\fa^+\fa^-$ so that $\fa^-$ (resp. $\fa^+$) is only divisible by the non-split (resp.~split) primes in $K$.

\item For an $O$-module $M$, let $M_v=M\otimes_{O} F_v$. For an abelian group 
$N$, let $\wh{N}=N\otimes_{\BZ}\wh{\BZ}$. 
\item For a finite place $v$ of $F$, let $F_v$ be the completion of $F$ at $v$. Let $O_v$ denote the ring of integers of $F_v$, $\varpi_v$ a uniformiser of $O_v$ and $\fp_v$ the prime ideal of $O_v$. Let $q_v$ be the cardinality of residue field of $O_v$. 
We also fix a generator $d_v$ of the different $D_v$ of $F_v$ and $d_{K_v/F_v}$ a generator of the relative different $D_{K_v/F_v}$ of $K_v/F_v$. For an ideal $\fc$ of $F$, let $O_{K,\fc}=O+\fc O_K$ be the $O$-{order} of conductor $\fc$. 

\item For a finite place $v$ of $F$, let $\ord_v:F_v^\times\ra \BZ$ be the valuation such that $\ord_v(\varpi_v)=1$. We also view 
$\ord_v$ as a homomorphism from the group of fractional ideals of $F_v$ to $\BZ$.

\item For a finite place $v$ of $F$ and $\chi_v$ a character of $K_v$, define its conductor $\cond(\chi_v)$ with respect to $F_v$ as the maximal integral ideal $\fr_v$ of $F_v$ such that 
{$$\chi_v((1+\fr_vO_{K_v})^\times)=1.$$} For a Hecke character $\chi$ over $K$, the global conductor $\cond(\chi)$ is analogously defined. 
{We let $\Cond(\chi)$ denote the usual notion of conductor of a Hecke character $\chi$ over $K$, which is the maximal integral ideal $\ff$ of $K$ such that $\chi((1+\wh{\ff})^\times)=1$.}

\item In the main text we fix a totally definite quaternion algebra $B$ over $F$, and let $G=B^\times$ denote the associated algebraic group.

\item{Throughout this paper, we adopt the convention that an $L$-function does not include the archimedean $L$-factors.}

\end{itemize}

\subsection*{Acknowledgments}{We thank Haruzo Hida and Christopher Skinner for numerous instructive discussions. We also thank Shilin Lai and Wei Zhang for helpful comments on the preprint, and Raul Alonso Rodrigues, Li Cai, Shinichi Kobayashi, Philippe Michel and Lei Yang for discussions. 

The work of A.B. is partially supported by the NSF grant DMS 2302064, W.H. by the NSFC grant No. 12501019, Y.T. by the NSFC grant No. 12288201  and X.Y. by NSFC grant No. 12031019. 
}

\section{Gross points and modular forms on a Shimura set}\label{Shimuraset}
This section introduces a totally definite quaternion algebra, modular forms and Gross points on the associated Shimura set.
\subsection{Set-up}\label{SS:setup}
\subsubsection{CM field}\label{CMfd}
Let $K$ be a CM field and $F$  its maximal totally real subfield. 

Let $c\in \Gal(K/F)$ be the non-trivial element. Fix a CM type $\Sigma$ of $K$. We often identify it with the set of infinite places of $F$. 

{Let $S$ be an auxiliary finite set of places of $F$.} Let $\theta\in K$ be such that 
\begin{itemize}
    \item [\tiny{$\bullet$}]$ \Im(\sigma(\theta))>0$ for all $\sigma\in \Sigma$,
    \item [\tiny{$\bullet$}]$ \{1,\theta\}$ is an $O_v$-basis of $O_{K_v}$ for all $v|D_{K}$ or $v\in S$,
    \item [\tiny{$\bullet$}] $\theta$ is a uniformiser of $K_v$ for all $v|D_{K/F}$.
\end{itemize}
Put $\delta=\theta-c(\theta)$. 
\subsection{Totally definite quaternion algebra}
\label{S:totq}
 Let $B$ be a totally definite quaternion algebra over $F$. Let $S_B$ be the set of finite places of $F$ at which $B$ ramifies. Put $D_{B}=\prod_{v\in S_B} \fp_v$.

Write ${\rm T}$ and ${\rm N}$ for the reduced trace and norm of $B$ respectively. Let $S_0$ be a finite set of finite places at which $B$ splits. We will consider a family of special points in \S\ref{spepts} whose conductors only vary at places in $S_0$. Suppose that the set $S$ as in \S\ref{CMfd} contains $S_B$ and $S_0$. In applications, we often take $$S=S_0\cup S_1\cup \{v\ |\ v|\ell\}$$ with $S_1\supset S_B$ consisting of places dividing the conductor of a given cuspidal automorphic representation and $\ell$ a rational prime that is coprime to the places in $S_0$. 

Suppose that\footnote{The hypotheses will be satisfied in our setting (cf.~Lemma~\ref{lm:disc}).}
\begin{itemize}
  \item [\tiny{$\bullet$}] $K$ can be embedded into $B$,
  \item [\tiny{$\bullet$}] $-1\in F_v^\times$ is a norm from $K_v^\times$ for all $v\in S\bs S_B$.
\end{itemize}
We often fix an embedding $\iota: K\hookrightarrow B$, and then choose a basis $\{1, J\}$ of $B$ as a $K$-algebra such that $$Jt=c(t)J$$ for all $t\in K$, and
\begin{itemize}
    \item [\tiny{$\bullet$}] $\beta:=J^2\in F^\times$ with $\sigma(\beta)<0$ for all $\sigma\in \Sigma$,
    \item [\tiny{$\bullet$}] $\beta\in O_v^{\times 2}$ for all $v\in S\bs{S_B}$.
\end{itemize}
In view of the strong approximation 
the existence of such a $\beta$ is a local problem, and the existence may be seen as follows: recall that $B_v$ splits if and only if ${\rm{N}}(J)\in {\rm N}(K^\times
_v)$. So it suffices to check the existence of a $k_v\in K_v^\times$ for each finite $v\in S\bs S_B$ such that $-{\rm{N}}(Jk_v)\in O_v^{\times 2}$. Since we are assuming that $-1\in F_v^\times$ is a norm from $K_v^\times$ for all $v\in S\bs S_B$, the existence follows.

In the following we fix an isomorphism 
\begin{equation}\label{eq:sp}
i=\prod i_v: B(\BA_{f}^{(S_B)})\simeq {M_2}(\BA_{f}^{(S_B)}). 
\end{equation}

For a finite place $v\notin S$, choose $i_v: B(F_v)\simeq M_{2}(F_v)$ such that
 \begin{equation}i_{v}(O_{K_v})\subset M_2(O_v).
 \end{equation} 
Fix a square root $\sqrt{\beta}\in \ov{\BQ}$ of $\beta$.
For $v\in S\bs{S_B}$,
 define $i_v: B(F_v)\simeq M_{2}(F_v)$ by
\begin{equation}\label{E:embedding.W}i_v(\theta)=
\begin{pmatrix}
{\rm T(\theta)} & {-\rm {\rm{N}}(\theta)}\\
1 & 0\\
\end{pmatrix};\quad i_v(J)=
\sqrt{\beta}\cdot 
\begin{pmatrix}
-1 &\rm T(\theta)\\
0 & 1\\
\end{pmatrix}\quad(\sqrt{\beta}\in O_v^\times).
\end{equation}
We further choose $i_v$ so that $i_v(J)\in i_v(K_v^\times) \GL_2(O_v)$ for all finite $v\notin S$. 

Let $G= B^\times$ be an algebraic
group over $F$. 
 From now, we often identify $B(F_v)$ with $M_{2}(F_v)$ via the above isomorphism $i_v$ for $v\notin S_B$ finite, and in turn $G(F_v)$ with $\GL_2(F_v)$.

For convenience, we will typically choose an isomorphism $\iota: B(K_\infty)\simeq M_2(K_\infty)$ defined over $K$, which is induced from an  embedding $\iota: K\hookrightarrow B$ such that $\iota(t)=\begin{pmatrix}
    t& \\ &c(t)
\end{pmatrix}$ for $t\in K$.
That is, we fix an isomorphism \begin{equation}\label{eq:sp1}\iota:B(K)\simeq M_2(K),\quad a+bJ\mapsto \begin{pmatrix}
    a& b\beta\\ c(b)&c(a)\end{pmatrix},\quad a,b\in K.\end{equation}
\subsection{{Definite quaternionic modular forms}}\label{S:defform}{In this subsection we briefly describe  algebraic and $\ell$-integral theory of modular forms on the multiplicative group of a totally definite quaternion algebra.}
\subsubsection{Algebraic representations of $G(F)$}\label{algrep}
Let $A$ be a commutative ring and $h\in 2\BZ_{\geq 1}$. 

Denote by {$V_{h-2}(A)\subset A[X,Y]$} the space of homogeneous polynomials of degree $h-2$ with coefficients in $A$. Let $(\rho_h, V_{h-2}(A))$ be the representation of $\PGL_2(A)$ given by 

\[(\rho_h(\gamma) P)(X,Y):=\det(\gamma)^{-\frac{h-2}{2}}P((X,Y)\gamma),\quad \gamma\in \PGL_2(A),\quad P(X,Y)\in V_{h-2}(A).\]

If $A$ is a $\BZ[\frac{1}{(h-2)!}]$ algebra, then there is a $\PGL_2(A)$-invariant non-degenerate pairing on $V_{h-2}(A)$ given by: for $-\frac{h-2}{2}\leq i\leq \frac{h-2}{2}$ and $x_i:=X^{\frac{h-2}{2}+i}Y^{\frac{h-2}{2}-i}$, define 
\[\pair{x_{i},x_j}_{h-2}:=\delta_{i,-j}\cdot (-1)^{\frac{h-2}{2}+i}\frac{\Gamma(h/2+i)\Gamma(h/2-i)}{\Gamma(h-1)}.\]
{Let $B$ be a totally definite quaternion algebra over $F$.}
Fix an isomorphism $$\iota : B(\ov{\BQ})\simeq M_2(\ov{\BQ}).$$
Put $k=\sum_{\sigma\in \Sigma}k_\sigma\cdot\sigma \in 2\BZ_{\geq 1}[\Sigma]$ and define
\[\rho_{k}:G(F)\ra \prod_{\sigma\in \Sigma}\GL_2(\ov{\BQ})\xrightarrow{\prod_{\sigma\in \Sigma} \rho_{k_\sigma}} \prod_{\sigma\in \Sigma}\Aut ({ V}_{k_\sigma-2}(\ov{\BQ})),\]
where the first map is induced by $g\mapsto \prod_{\sigma\in \Sigma}\sigma(g)$. 

We have a ${\rm PLG_2}(\ov{\BQ})$-invariant bilinear pairing \[\displaystyle\pair{\ ,\ }_{k-2\Sigma}:=\otimes_{\sigma\in \Sigma} \pair{\ ,\ }_{k_\sigma-2}\] on $\bigotimes_{\sigma\in \Sigma} V_{k_\sigma-2}(\ov{\BQ})$ given by \[\pair{x,y}_{k-2\Sigma}:=\prod_{\sigma\in \Sigma} \pair{x_\sigma,y_\sigma}_{k_\sigma-2},\] for pure tensors $\displaystyle x=\otimes_{\sigma\in \Sigma}x_\sigma,  y=\otimes_{\sigma\in \Sigma}y_\sigma\in \bigotimes_{\sigma\in \Sigma}{ V}_{k_\sigma-2}(\ov{\BQ})
.$

For any subalgebra $A$ of $\ov{\BQ}$ that contains the Galois closure of $F$ so that $\iota$ is defined over $A$, note that $\rho_k$ is also defined over $A$.
Put $$V_{k-2\Sigma}(A)=\bigotimes_{\sigma\in \Sigma} V_{k_{\sigma} -2}(A),$$  and let $\rho_k$ also denote the representation of $G(F)$ on $V_{k-2\Sigma}(A)$. {Here the restriction of $\rho_k$ to the 
$\sigma$-component is given by the composite $G(F)\xrightarrow{\sigma}G(A)\xrightarrow{\rho_{k_\sigma}} V_{k_\sigma-2}(A)$.}

{We may and do choose $\iota$ such that \eqref{eq:sp1} holds and $x_0$ is $K^\times$-invariant.}
\subsubsection{Classical modular forms}\label{auto:rep.wt=k}
Fix an embedding $\iota_\infty: \ov{\BQ}\hookrightarrow \BC$.

Now $\rho_k$ induces a {continuous} representation 
 \[\rho_{k,\infty}:G(F_\infty)\ra \prod_{\sigma\in \Sigma}\GL_2(\BC)\xrightarrow{\prod_{\sigma\in \Sigma} \rho_{k_\sigma}} \prod_{\sigma\in \Sigma}\Aut ({ V}_{k_\sigma-2}(\BC))\] so that $\rho_{k,\infty}|_{G(F)}=\iota_\infty\circ \rho_k$.

For an open compact subgroup $U\subset G(\BA_f)$, let $\CA_k(U)$ be the subspace of automorphic forms on $G(\BA)$ of weight $k$, level $U$ and trivial central character. Here an automorphic form is said to be of weight $k$ if for each $\sigma\in \Sigma$ the $G(F_\sigma)$-representation generated by the form is a direct sum of $\rho_{k_\sigma}$. Put $$\CA_k=\varinjlim \CA_k(U).$$

For a subalgebra $A\subset \BC$ that contains the Galois closure of $F$ and so that $\iota$ as in \eqref{eq:sp1} is defined over $A$, {we have a representation $(\rho_k, V_{k-2\Sigma}(A))$ of $\rho(F)$ defined similarly as in \S\ref{algrep}.}
The space of weight $k$ modular forms with coefficients in $A$ is defined by
\[ {\rm M}_{k}(U,A)=\{\varphi:G(\BA_f)\ra V_{k-2\Sigma}(A)\ |\ \varphi(\gamma gu)=\rho_{k}(\gamma)\varphi(g), \forall u\in U\}.\] Via the right translation, ${\rm M}_{k}(A)=\varinjlim {\rm M}_{k}(U,A)$ is an admissible $G(\BA_f)$-representation. 

Define the $G(\BA)$-equivariant map: 
\[\Psi:V_{k-2\Sigma}(\BC)\otimes_{\BC}{\rm M}_{k}(\BC) \ra \CA_k,\quad \quad v\otimes \varphi\mapsto f: g\mapsto \pair{\rho_{k,\infty}(g_\infty)v,\varphi(g_f)}_{k-2\Sigma},\quad g=g_\infty\cdot g_f\in G(F_\infty)\cdot G(\BA_f).\]
\subsubsection{$\ell$-adic avatar of modular forms}\label{avatar}
Let $\ell$ be a rational prime.
Fix an embedding $\iota_\ell: \ov{\BQ}\hookrightarrow \ov{\BQ}_\ell$. 

Then $\rho_k$ induces a {continuous} representation
 \[\rho_{k,\ell}:G(F_\ell)\ra \prod_{ \sigma\in \Sigma}\GL_2(\ov{\BQ}_\ell)\xrightarrow{\prod_{\sigma\in \Sigma}\rho_{ k_\sigma}} \prod_{\sigma\in \Sigma}\Aut ({ V}_{k_\sigma-2}(\ov{\BQ}_\ell)),\]
 so that $\rho_{k,\ell}|_{G(F)}=\iota_\ell\circ \rho_k$.

Let $A\subset \ov{\BQ}_{\ell}$ be a complete algebra that contains $\iota_\ell\circ\sigma(O)$ for all $\sigma\in \Sigma$ and so that $\iota$ as in \eqref{eq:sp1} is defined over $\Frac(A)\cap \ov{\BQ}$. {Then we have a representation $(\rho_k, V_{k-2\Sigma}(A))$ of $\rho(F)$  defined similarly as in \S\ref{algrep}.} Let $U=\prod_{v<\infty} U_v\subset \wh{B}^\times$ be an open compact subgroup such that $\rho_{k,\ell}|_{\prod_{v|\ell}{U_{v}}}$ {preserves} $V_{k-2\Sigma}(A)$.

The $\ell$-adic avatar of the space of weight $k$ modular forms with level $U$ is defined by
 \[{\CM}_{ k}(U,A)=\{\wh{\varphi}:G(F)\bs G(\BA_f) 
 \ra V_{k-2\Sigma}(A)\ |\ \wh{\varphi}(\gamma gu)=\rho_{{k},\ell}(u_\ell^{-1})\wh{\varphi}(g), \forall u\in U\}.\]
 If $A$ is the integer ring of a local field, then ${\CM}_{ k}(U,A)\otimes_{A}\Frac(A)={\CM}_{ k}(U,\Frac(A))$.
In our applications, $A$ will typically be the $\ell$-adic completion  of a number field $L$ or its integer ring $O_L$ with respect to the embedding $\iota_\ell$.

Let $L$ be a number field that contains the Galois closure of $F$ so that $\iota$ as in \eqref{eq:sp1} is defined over $L$. Let $\tilde{L}$ be its $\ell$-adic completion with respect to the embedding $\iota_\ell$. Then the following $\ell$-adic avatar map  
\[{\rm M}_k(U,L)\otimes_{L}\tilde{L}\ra {\CM}_k(U,\tilde{L}),\quad \varphi\mapsto \wh{\varphi}: g\mapsto \rho_{{{ k}},\ell}(g_\ell^{-1})\varphi(g)\] is an isomorphism.
In particular, for $x\in V_{{k-2\Sigma}}(L)$, $\varphi\in {\rm M}_{k}(U,L)$ and  $g\in \wh{B}^\times$, we have 
\[\pair{x,\varphi(g)}_{k-2\Sigma}=\pair{\rho_{{k},\ell}(g_\ell^{-1})x,\wh{\varphi}(g)}_{k-2\Sigma}\]
\begin{defn}
A modular form in $\CM_k(U,\tilde{L})$ is said to be 
$\ell$-integral if it lies in $\CM_k(U,O_{\tilde{L}})$.
\end{defn}

\subsubsection{$\ell$-optimal form}
{Let $L$ be as in \S\ref{avatar}.}

Let $x\in V_{k-2\Sigma}(L)$ be a pure tensor that is
$\ell$-primitive, i.e. $x$ lies in $V_{k-2\Sigma}(O_{\tilde{L}})$ and {it} is $\ell$-indivisible. 

\begin{defn}\label{opt:ellava}A modular form $\wh{\varphi} \in {\CM}_k(U, \tilde{L})$ is $\ell$-optimal with respect to $x$ if {
$$\pair{x,\wh{\varphi}(g)}_{k-2\Sigma}\in O_{\tilde{L}}$$} for all 
$ g\in \wh{B}^\times$, and there exists a $g\in \wh{B}^\times$ such that $\pair{x,\wh{\varphi}(g)}_{k-2\Sigma}$ is an $\ell$-adic unit.\end{defn}
\begin{remark}\label{rmk:opt}If $\ell\Sigma> (k-2\Sigma)\in \BZ[\Sigma]$, then $\wh{\varphi}$ is $\ell$-optimal with respect to $x$ if and only if $\wh{\varphi}$ is $\ell$-integral and is non-zero\footnote{Since $G(\BF_{\ell^r})$ representation $V_{k_\sigma-2}(\BF_{\ell^r})$ is irreducible for $r\in \BZ_{\geq 0}$ and $\sigma|\infty$.} modulo $\ell$.\end{remark}

 \subsection{Special points}
 \label{spepts}
This subsection describes an analogue of CM points in the totally definite setting, which arises from the 
maximal anticyclotomic extension of $K$ unramified outside $S_0$.

 \subsubsection{Toric embedding}
 \label{gpt}
We identify $G(\BA_{f}^{ (S_B)})$ with $\GL_2(\BA_f^{(S_B)})$ as in \eqref{eq:sp}.
For ${g,h\in G(\BA)}$, put
\[\iota_{h}(g)=h^{-1}g h.\]

 Let $S_1^+\subset S_1$ (resp. $S_1^-\subset S_1$) be the subset consisting of finite places of $F$ that are split (resp.~non-split) in $K$. Note that $S_1^{+}\cap S_{B}=\emptyset$ since we have assumed that $K\hookrightarrow B$.
For a finite place $v\notin S_0$, define $\cmptv_{v}\in G(F_v)$ by
\begin{equation}\label{E:cmptv.W}\begin{aligned}
\cmptv_{v}=&1\text{ if $v\nmid S_0\cup S_1^+$,}\\
\cmptv_{v}=&\delta^{-1}\begin{pmatrix}
  \CMP & c(\CMP)\\
1 & 1\\
\end{pmatrix}
\in{\SL_2(O_{K_w})}={\SL_2(O_{F_v})}\text{ if $v\in S_1^+\bs S_0$,}
\end{aligned}\end{equation} {where $\theta$ and $\delta$ are as in \S\ref{CMfd}.}
For each $v\in S_1^+$, choose a place $w$ of $K$ above $v$. If $v\in S_1^+\bs S_0$, for $t=(t_1,t_2)\in K_v:=K\otimes_{F}F_v=K_w\oplus K_{\ov{w}}$ {we have}
\begin{equation}\label{E:cm2.W}\iota_{\cmptv_{v}}(t)=\begin{pmatrix}
t_1 & 0\\
0 & t_2\\
\end{pmatrix}.
\end{equation}

For each $v\in S_0$ and a non-negative integer $r_v$, define $\cmptv_v^{(r_v)}\in G(F_{v})$ as follows.

{To begin, for $r_v\geq 0$, put $\delta_{r_v}=\begin{pmatrix}
  \varpi_{v}^{r_v}&\\ &1
 \end{pmatrix} $.}

If $v$ splits in $K$ as $v=\wp\ov{\mathfrak{\wp}}$, where $\wp$ is a fixed place of $K$ as above, then
\begin{align}\label{E:op1.W}
\cmptv_v^{(r_v)}:=&\begin{pmatrix}
\CMP & -1\\
1 & 0\\
\end{pmatrix}
\delta_{r_v}
\in\GL_2(K_\wp)=\GL_2(F_v).
\intertext{If $v$ is non-split in $K$, then}
\label{E:op2.B}\cmptv_v^{(r_v)}:=&
\begin{pmatrix}
0 & 1\\
-1 & 0\\
\end{pmatrix}
\delta_{r_v}.
\end{align}
We take $\varsigma_\infty=1$ and for $\fr=\prod_{v\in S_0} \fp_v^{r_v}$, let $\varsigma^{(\fr)}=\prod_{v\notin S_0}\varsigma_v\prod_{v\in S_0}\varsigma_v^{(r_v)}$.

The above local embeddings lead to a family of embeddings \[\iota_{\varsigma^{(\fr)}}: \BA_K\hookrightarrow B(\BA).\]
The above explicit embeddings satisfy~\eqref{opti}~and~Lemma~\ref{unco}, which will be used in our non-vanishing arguments. 
\subsubsection{Quaternionic order}\label{ss:qo} We introduce an order of the totally definite quaternion algebra, with respect to which special points will be introduced in the next subsection. 

Let $R\subset B$ be an order such that 
\begin{itemize}
  \item[\tiny{$\bullet$}] For $v\notin S_1\cup S_0$, $R_v=M_2(O_v)$;
  \item[\tiny{$\bullet$}] For $v\in S_1\bs S_0$, $R_v$ contains $\iota_{\varsigma_v} O_{K_v}$;
  \item[\tiny{$\bullet$}] For each $v\in S_0$, $R_v$ is the standard Eichler order $$M_0(\fp_v^{n_v})_v=\left\{\begin{pmatrix}
      a& b \\ c & d
  \end{pmatrix}\in M_2(O_v)\ |\ c\in \fp_v^{n_v}\right\}$$ of discriminant $\fp_v^{n_v}$ for an integer $n_v\geq 0$.

  \end{itemize}
  Let $\fr=\prod_{v\in S_0} \fp_v^{r_v}$ such that $r_v\geq n_v$ for any $v\in S_0$.
  
   Together with the choice of $\varsigma^{(\fr)}$ as above, we have \begin{equation}\label{opti}\wh{R}\cap \iota_{\varsigma^{(\fr)}} \wh{K}=\iota_{\varsigma^{(\fr)}} \wh{O}_{K,\fr}.\end{equation}
  
\subsubsection{Special points}\label{sppt}
Define $x_{\fr}:\BA_K^\times\to G(\BA) $ by
\begin{equation}\label{E:op3.W}x_{\fr}(a):=a\cdot \cmptv^{(\fr)},\quad \cmptv^{(\fr)}:=\prod_{v\notin S_0}\varsigma_v\prod_{v\in S_0}\varsigma_v^{(\ord_v\fr)}\cmptv_{v}.
\end{equation}
This gives a family of special points 
$\{x_{\fr}(a)\}_{a\in\BA_K^\times}$.

\section{Explicit Waldspurger formula}\label{exp:Wald}
This section presents an explicit Waldspurger formula in a general context. It is based on the work
of Cai--Shu--Tian \cite{CST} to which we refer for an introduction. The main result is an anticyclotomic
twist family version of the Waldspurger formula for Rankin--Selberg $L$-values which involve a fixed test vector (cf.~Theorem~\ref{T:central.W}). We also consider its $\ell$-integral analogue, leading to an $\ell$-integrality of the Rankin--Selberg $L$-values (cf.~Proposition~\ref{lbd:RS}).
\subsection{Backdrop}\label{setting:Wald}
\subsubsection{Setting}\label{ss:set}

Let $K$ be a CM field and let $F$ be its maximal totally real subfield. Let $\Sigma$ be a CM type of $K$, which we identify with the set of infinite places of $F$.

Let $\sigma$ be a cuspidal irreducible automorphic representation of $\GL_2(\BA)$ with trivial central character such that
\begin{itemize}
\item[(H1)] The archimedean component $\sigma_\infty$ is the discrete series of weight $k$ for $k\in 2\BZ_{\geq 1}[\Sigma]$;
\item[(H2)] $\varepsilon(\sigma_{K})=+1$.
\end{itemize} 
Here $\varepsilon(\sigma_K)$ is the epsilon factor\footnote{See \cite[Chap.~2]{Langlands} for base change representations and \cite[Thm.~2.18]{Jacquet_Langlands:GLtwo} for definition of epsilon factors of $\GL_2$ representations.} of base change $\sigma_K$. 

The following example will be of particular interest for the paper. 
\begin{example}\label{ex1}
  Let $\lambda$ be a self-dual Hecke character over $K$ in the sense of \eqref{theta} of infinity type $\Sigma+\kappa(1-c)$ for $\kappa\in \BZ_{\geq 0}[\Sigma]$. Then the automorphic representation $\sigma_{\lambda}$ of $\GL_2(\BA)$ generated by the associated theta series {$\phi_{\lambda}$} satisfies the hypotheses (H1) and (H2) with weight $k=2\Sigma+2\kappa$ and 
  we have  
  $$L(s+1/2,\lambda)=L(s,\sigma_\lambda)$$ (cf.~\cite[\S12]{Jacquet_Langlands:GLtwo}).
  \end{example}
Let $B$ be the quaternion algebra over $F$ such that the Tunnell--Saito condition 
\begin{equation}\label{TS}
\varepsilon(\sigma_{K,v})=\varepsilon(B_v)
\end{equation}
holds for any place $v$ of $F$, where $\varepsilon(\sigma_{K,v})$ denotes the local base change epsilon factor and $\varepsilon(B_v)$ is the Hasse invariant of $B_v:=B(F_v)$. It is a totally definite quaternion algebra. 
Put $G=B^\times$. 

Let $\pi$ be the cuspidal irreducible representation of $G(\BA)$ whose Jacquet--Langlands transfer to $\GL_2(\BA)$ is $\sigma$. Note that $\pi$ exists in view of the Tunnell--Saito theorem \cite{Tu, Sa} and the condition \eqref{TS}. It is a unitary irreducible cuspidal automorphic representation with trivial central character and has weight $k$ in the sense of \S\ref{auto:rep.wt=k}. {We say that $\pi$ has CM by $K$ if its Jacquet--Langlands transfer $\sigma$ arises from a Hecke character\footnote{Since $\sigma$ has trivial central character, the associated Hecke character is self-dual.} over $K$.} 

Denote by $\mathfrak{N}$ the conductor of $\sigma$. Write $\mathfrak{N}=\mathfrak{N}^+\mathfrak{N}^-$, where $\mathfrak{N}^+$ and $\mathfrak{N}^-$ are only divisible by places that are split and non-split in $K$ respectively. 

{For an ideal $\fr$ of $F$, put\[{\rm P}_\fr=K^\times\wh{F}^\times\bs \wh{K}^\times/\wh{O}_{K,\fr}^{\times}.\] Let $S_0$ be a set of finite places at which $B$ splits and put \[{\rm P}_{S_0}=\varprojlim_{\fr}{\rm P}_\fr,\] where $\fr$ runs over all ideals of $F$ that are supported on $S_0$.
Let $\CX_{S_0}$ be the set of $\ov{\BQ}^\times$-valued finite order characters $\chi$ of ${\rm P}_{S_0}$ such that 
$$
    \text{$\chi_v=1$ for all $v|\mathfrak{N}^-, v\notin S_0$}
    $$ and $$\mathfrak{N}_{S_0}\mid \cond(\chi).$$}

\subsubsection{Some consequences of the Tunnell--Saito theorem}\label{S:conTS}

{By the eponymous theorem \cite{Tu,Sa}, the condition \eqref{TS} holds for all places $v$ if and only if $$\Hom_{K_{\BA}^\times}(\pi,\BC)\neq 0.$$}

In general, for a cuspidal irreducible automorphic representation $\pi$ of $G(\BA)$ with central character $\omega_\pi$, a Hecke character $\chi$ over $K$ with  $\chi|_{\BA^\times}\omega_{\pi}=1$, it is known that $\dim \Hom_{\BA_K^\times}(\pi,\chi^{-1})\leq 1$. Note that 
\[\Hom_{\BA_K^\times}(\pi,\chi^{-1})\neq 0 \iff \Hom_{K_v^\times}(\pi_v,\chi_v^{-1}) \neq 0 \text{ for each place $v$.}\] 
The Tunnell--Saito theorem ~\cite{Tu, Sa} asserts that 
\begin{equation}\label{TSloc} \Hom_{K_v^\times}(\pi_v,\chi_v^{-1}) \neq 0 \iff
\varepsilon(\pi_{K,v}\otimes\chi_v)=\chi_v(-1)\varepsilon(B_v).\end{equation} 

\begin{prop}\label{char}Under the hypotheses (H1), (H2) and the above choice of $B$, any $\chi\in \CX_{S_0}$ satisfies
    \[\Hom_{\BA_K^\times}(\pi,\chi^{-1})\neq 0\]and \[\varepsilon(\pi_{K}\otimes\chi)=+1.\] 
\end{prop}
\begin{proof}
  Let $v$ be a place of $F$. 
  
  If $v\notin S_0$ is finite, and either $v$ is split in $K$ or $v\nmid \mathfrak{N}$, then $B_v$ splits. Let $B_v'$ be the non-split quaternion algebra over $F_v$ and $\pi_v'$ be the irreducible representation of $B_v'^\times$ corresponding to $\pi_v $. By the Tunnell--Saito theorem, we have \[\dim \Hom_{K_v^\times}(\pi_v,\chi_v^{-1})+\dim \Hom_{K_v^\times}(\pi_v',\chi_v^{-1})=1.\]
  If $K_v$ is split or $v\nmid \mathfrak{N}$, then either an embedding $K_v\hookrightarrow B_v$ does not exist or the representation $\pi_v'$ does not exist. In particular,  $\Hom_{K_v^\times}(\pi_v',\chi_v^{-1})=0$ and hence $\Hom_{K_v^\times}(\pi_v,\chi_v^{-1})\neq 0.$

If $v\in S_0$, then the non-vanishing of 
 $\Hom_{K_v^\times}(\pi_v, \chi_v)$ 
 is a consequence of the explicit Tunnell--Saito theorem as in ~\cite[Prop.~3.1]{CST}.

If $v|\mathfrak{N}^-\infty$ but $v\notin S_0$, then  $\chi_v=1$ and hence $\Hom_{K_v^\times}(\pi_v^\times,\chi_v^{-1})=\Hom_{K_v^\times}(\pi_v^\times,\BC)\neq 0$ by the hypothesis \eqref{TS} and the Tunnell--Saito theorem \eqref{TSloc}.

The second assertion that $\varepsilon(\pi_{K}\otimes\chi)=+1$ also follows from the Tunnell--Saito theorem \eqref{TSloc}.
\end{proof}

In the CM case we have the following key. 
\begin{lem}\label{lm:disc}
    {Suppose that {$\sigma$} has CM by $K$ in the sense of Example \ref{ex1}. Then} the quaternion algebra $B$ as in \eqref{TS} satisfies\footnote{ {In particular, the hypothesis of subsection \ref{S:totq} is satisfied.}}
     $$
     D_{B}=\prod_{\eta_{K_{v}/F_v}(-1)=-1}\fp_v.
     $$
       \end{lem}
     
\begin{proof}
 
Suppose that $\sigma=\sigma_\lambda$ is associated to $\lambda$ as in Example~\ref{ex1}.
 For each $v$, let $\psi_v$ be a non-trivial additive character of $F_v$ and $\psi_{K_v}=\psi_v\circ\tr_{K_v/F_v}$. Let $\varepsilon(\sigma_{v},\psi_{v})$ be the epsilon factor of $\sigma_v$ with respect to $\psi_v$ defined in \cite[Thm.~2.18]{Jacquet_Langlands:GLtwo} and 
$\varepsilon(\sigma_{K_v},\psi_{K_v})$ the base change epsilon factor.

Then we have the following relation between base change epsilon factors and product of quadratic twist epsilon factors 
    \[\begin{aligned}
\varepsilon(\sigma_{K,v},\psi_{K_v})
=&\varepsilon(\sigma_{v},\psi_{v})\varepsilon(\sigma_{v}\otimes\eta_{K_v/F_v},\psi_{v})\eta_{K_v/F_v}(-1)\\
=&\varepsilon(\lambda_v^*,\psi_{K_v})^2(\lambda_{K_v}(\psi_v)^2\eta_{K_v/F_v}(-1))\\
=&\varepsilon(\lambda_v^*,\psi_{K_v})^2\\
=&\eta_{K_v/F_v}(-1). 
    \end{aligned}\] 
      Here $\lambda_v^* :=\lambda_v\cdot |\cdot |_{K_v}^{1/2}$ denotes the  unitarisation of $\lambda_v$, the second equality follows from  \cite[Thm.~4.7]{Jacquet_Langlands:GLtwo} and 
     the last from 
    \[\begin{aligned}
\overline{\varepsilon(\lambda_v^*,\psi_{K_v})}
=&\varepsilon((\lambda_v^*)^{-1},\psi_{K_v})\cdot (\lambda_v^*)^{-1}(-1)\\
=&\varepsilon(\lambda_v^{*,c},\psi_{K_v})\cdot \eta_{K_v/F_v}(-1)\\
=&\varepsilon(\lambda_v^*,\psi_{K_v})\cdot \eta_{K_v/F_v}(-1).
    \end{aligned}\]
   Hence, in view of \eqref{TS} the proof concludes.
\end{proof}
\begin{remark}
By Lemma~\ref{lm:disc}, the totally definite quaternion algebra $B$ is an intrinsic invariant of the CM field $K$.
\end{remark} 
\subsection{Test vector} \label{test} 
In this subsection we describe an explicit uniform test vector for self-dual pairs $(\pi,\chi)$  as $\chi$ varies in an anticyclotomic family based on the work of Cai--Shu--Tian \cite{CST}.

Let $B$ be a quaternion algebra over a number field $F$ and $G=B^\times$. Let $\pi$ be a unitary irreducible cuspidal automorphic representation of $G(\BA)$ with central character $\omega_{\pi}$ and let $\chi$ be a Hecke character over $K$. The pair $(\pi,\chi)$ is called self-dual if $$\omega_{\pi}\cdot \chi|_{\BA^\times}=1.$$

For a self-dual pair $(\pi,\chi)$, we have $\dim \Hom_{K_{\BA}^\times}(\pi,\chi^{-1})\leq 1$ and a  criterion for the equality is {given by \eqref{TSloc}.}
Following Gross and Prasad \cite{GP},  $f\in \pi$ is said to be a test vector for $\Hom_{K_{\BA}^\times}(\pi,\chi^{-1})$ {with respect to an embedding $K_{\BA}^\times\hookrightarrow B_{\BA}^\times$} if any basis of $\Hom_{K_{\BA}^\times}(\pi,\chi^{-1})$ takes a non-zero value on $f$. 
This is essentially a local notion: 
$f$ is a test vector if it is so locally\footnote{defined analogously} for any place.
The Waldspurger formula \cite{Wa, YZZ} links the Rankin–Selberg $L$-value \[L(1/2, \pi_K\otimes\chi)\] with $K^\times$-toric period of a test vector in $\pi$ {against $\chi$} {with respect to an embedding $K\hookrightarrow B$.}

{Let the setting, and in particular $B$ and $\pi$, be as in \S\ref{ss:set}.} This subsection describes an explicit  uniform test vector for the pairs $(\pi,\chi)$ as $\chi$ varies in the family $\CX_{S_0}$ of Hecke characters. The next subsection describes an explicit relation between the $L$-values $L(1/2,\pi_K\otimes \chi)$ and toric periods of the test vector.

\subsubsection{Definitions}\label{deff}
For a finite place $v$, recall that $M_0(\fp_v^n)_v$ is 
 the Eichler of discriminant $\fp_v^n$. Put $U_0(\fp_v^n)_{v}=M_0(\fp_v^n)_{v}^\times$. Recall that we have fixed an embedding $\iota: K\hookrightarrow B$ and a family of embeddings $\iota_{\varsigma^{(\fr)}}: K_{\BA}^\times\hookrightarrow G(\BA)$ in \S\ref{gpt}.
Let $S_1$ denote the set of places dividing the conductor $\mathfrak{N}$ and $S_0$ a finite set of places at which $B$ splits.
\begin{defn}\label{D:1.W} For a place $v$, define a non-zero vector $f_v\in\pi_v$ as follows. 
\begin{itemize}
\item[a)] If $v \mid \mathfrak{N}^-\infty$ and $v\notin S_0$, then $f_v$ is invariant under the action of $\iota_{\varsigma_v}K
_v^\times$.
\item[b)] If either $v\nmid \mathfrak{N}^-$ is finite or $v\in S_0$, then {$f_v$ is fixed by $ U_0(\mathfrak{N}_v)_v$.}
\end{itemize}
\end{defn}
\begin{defn}\label{D:2.W}
Let $R\subset B$ be an order of discriminant $\mathfrak{N}$ satisfying the following.
\begin{itemize}
\item[a)] If $v|\mathfrak{N}^-$ and $v\notin S_0$, then $R_v \subset B_v$ is an order so that $R_v \cap\iota_{\varsigma_v} K_{v}=\iota_{\varsigma_v} O_{K_v}$. 
\item[b)] If either $v\nmid \mathfrak{N}^-$ finite or $v\in S_0$, then $R_v\subset B_v$ is the Eichler order $M_0(\mathfrak{N}_v)_v$.
\end{itemize}
\end{defn}

\subsubsection{Existence}
\begin{lem}\noindent
\begin{itemize}
\item [i)]For any finite $v$, an order $R_v \subset B_v$ with discriminant $\mathfrak{N}_v$ as in Definition \ref{D:2.W} exists. Moreover, it is unique up to conjugation by $\iota_{\varsigma_v}K_v^\times$ for $v|\mathfrak{N}^-$ and $v\notin S_0$.
   \item [ii)] {If} $v\notin S_0$, then $R_v \cap\iota_{\varsigma_v} K_{v}=\iota_{\varsigma_v} O_{K_v}$, and if $v\in S_0$ and $\mathfrak{N}_v|\fr_v=\fp_v^{r_v}$, then $$R_v\cap \iota_{\varsigma_v^{(r_v)}} O_{K_v}=\iota_{\varsigma_v^{(r_v)}} O_{K_{v},\fr_v}.$$ 
    \end{itemize}
\end{lem}
\begin{proof}
Part i) just follows from \cite[Lem.~3.3]{CST}.

Now we consider ii). The case (a) $v|\mathfrak{N}^-$ and $v\notin S_0$ or (b) $v\notin S_0$ and $v\nmid \mathfrak{N}$ follows from definition.

Now consider the case $v\mid \mathfrak{N}^+$ and $v\notin S_0$, Since 
$\iota_{\varsigma}: K_v\hookrightarrow B_v\simeq M_{2,v}$ is the diagonal embedding (cf.~Equation~\eqref{E:cm2.W}), the result holds.

The remaining case that $v\in S_0$ follows by a direct calculation (cf.~\S\ref{ss:qo}).
 \end{proof}
Note that $R$ satisfies the properties in \S\ref{ss:qo}.
\begin{lem}\label{lm:t-mult}\
\begin{itemize}
\item[i)] 
For any $v$, there exists $f_v\in\pi_v$ as in Definition \ref{D:1.W}.
 Moreover, it is unique up to scalars. 
\item[ii)] For any finite $v$, we have $f_v \in \pi_{v}^{R_v^\times}$. 
\end{itemize}
\end{lem}
\begin{proof} 
The first part follows from newform theory and the Tunnell--Saito theorem (cf.~\cite[Prop.~3.7]{CST}).
 For the second part, we only need to consider the case that $v|\mathfrak{N}^-$ but $v\notin S_0$, which is the content of \cite[Prop.~3.8]{CST}.
\end{proof}
\begin{defn}[Test vector]\label{D:tv}

Define $f\in\pi$ by $$f=\otimes_{v}f_v \in \otimes \pi_v = \pi$$ for $f_v \in \pi_v$ as in Definition \ref{D:1.W}. 

\end{defn}
\begin{prop}\label{testv}
{For any $\chi\in \CX_{S_0}$ of conductor $\fc$, the vector
\[\pi(\varsigma^{(\fc)})f\in \pi
\]
is a test vector for $\Hom_{K_{\BA}^\times}(\pi,\chi^{-1})$ with respect to the embedding $\iota$.} 
\end{prop}
\begin{proof}
This is essentially a local assertion, and follows from Proposition~3.7 of \cite{CST}.

 {For $v\nmid \mathfrak{N}^-$ and $v\notin S_0$, note that $B_v$ splits and $\chi_v$ is unramified. It follws that the local newform $f_v$ is a test vector for $\chi_v$ and
 so $f_v$ is a test vector.
For $v$ such that $v|\mathfrak{N}^-$ and $v\nmid S_0$, one has $\chi_v=1$ and the assertion just follows from the definition. Lastly, for $v\in S_0$, the assertion follows from the fact that under the condition $\mathfrak{N}_v\mid \cond(\chi_v)$, the local newform is a test vector for $\chi_v$.}
\end{proof}
\begin{remark}\label{uniform}
{Note that the above test vector $f$ only depends on the subset of primes $\{\fq\ |\ \fq\mid \mathfrak{N}^-\ \text{and}\ \fq\notin S_0\}$. So if $S_0$ is coprime to $\mathfrak{N}^-$, then $f$ does not depend on $S_0$.}
\end{remark}
The following preliminary will be used in our later arguments. 

  \begin{lem}\label{lm:lev}  For any place $\fq|D_{K/F}$ of $F$, suppose that $ord_{\fq}(\mathfrak{N}_\fq)> \ord_{\fq}(D_{K_\fq/F_\fq})$. 
Then ${\rm{N}}(R_v^\times)= {\rm{N}}(O_{K_v}^\times)$ for any finite place $v$. 
\end{lem}
\begin{proof}
    Let $v\nmid D_{K/F}$ and $v\notin S_0$ be a finite place. Then we have 
    $O_{K_v}\subset R_v$. 
    Since ${\rm{N}}: O_{K_v}^\times\rightarrow O_v^\times$ is surjective for $v\nmid D_{K/F}$, it follows that 
     ${\rm{N}}(R_v^\times)={\rm{N}}(O_{K_v}^\times)$.
    For the case $v\in S_0$, note that $B_v$ splits and $R_v=M_0(\mathfrak{N}_v)_v$ is an Eichler order, and so ${\rm{N}}(R_v^\times)\ra O_v^\times$ is surjective.

    Now consider the remaining case, namely $v$ is ramified in $K$.
    Note that $R_v$ is of the form $O_{K_v}+\varpi^{\ord_v(\mathfrak{N}_v/D_{B_v})}O_{B_v}$, where $\varpi\in K_v$ is a uniformiser and $O_{B_v}$ the maximal order of $B_v$. 
    We have \[O_{B_v}=O_{K_v}+\varpi^{\ord_{v}(D_{B_v}/D_{K_v/F_v})}O_{K_v}(1+J),\] where $J\in B_v$ is such that 
    \begin{itemize}
        \item[\tiny{$\bullet$}] $JkJ=\ov{k}$ for all $k\in K_v$ and ${\rm{N}}(J)\in O_v^\times$ and,
        \item[\tiny{$\bullet$}]for $v|2$ and $B_v$ ramified, we have $J^2\equiv 1\pmod{D_{K_v/F_v}/\fp_v}$ and it is {not} a norm of $K_{v}^\times$.
        \end{itemize}
    In this {case} note that $\ord_v(\mathfrak{N}_v)> \ord_{v}(D_{K_v/F_v})$ and so $R_v=O_{K_v}+\varpi^{\ord_v(\mathfrak{N}_v/D_{K_v/F_v})}O_{K_v}J$. 
    
     For $a\in O_{K_v}$ and $b\in \varpi^{\ord_v(D_{K_v/F_v})} O_{K_v}$, we have
     \[\begin{aligned}
      {\rm{N}}(a+bJ)
    \equiv& {\rm{N}}(a)\pmod{D_{K_v/F_v}}.\\
    \end{aligned}\]
    Hence, ${\rm{N}}(R_v^\times)={\rm{N}}(O_{K_v}^\times)$.
    \end{proof}
We have the following simple yet key observation. 
\begin{lem}\label{locavetest}
  {
  Let $M$ be a non-archimedean local field of characteristic $0$. Let $D$ be a quaternion algebra over $M$ and $E\subset D$ be a quadratic subalgebra over $M$. Let $\pi$ be an irreducible admissible unitary representation of $D^\times$ and let $\chi$ be a character of $E^\times$ such that $\dim_{\BC} \Hom_{E^\times}(\pi,\chi^{-1})=1$. Let $f\in \pi$ be a test vector\footnote{in the sense of \cite{GP}} for $\Hom_{E^\times}(\pi,\chi^{-1})$ of level $U$ for some open compact subgroup $U$ of $D^\times$. Let $\ov{U}$ denote the image of $U$ in $D^\times/M^\times$. Then for any open compact subgroup $V$ of $E^\times/M^\times$,  
 \[h:=\frac{1}{\#(V/V\cap \ov{U})}\sum_{t\in V/V\cap \ov{U}}\chi(t) \pi(t)f\in \pi\] is also a test vector for $\Hom_{E^\times}(\pi,\chi^{-1})$ and so non-zero.  Furthermore, for any basis $P\in \Hom_{E^\times}(\pi,\chi^{-1})$, we have 
  \[P(h)= P(f).\]}
\end{lem}

\begin{proof}
{
  Let $T_{\chi^{-1}}$ be the maximal quotient of $\pi$ on which $E^\times$ acts via $\chi^{-1}$, which is non-zero by assumption. Then each element in $\Hom_{E^\times}(\pi,\chi^{-1})$ factors through the natural projection map $\pr:\pi\ra  T_{\chi^{-1}}$. Note that the map $\pr$ is $E^\times$-equivalent and $V/V\cap \ov{U}$ is finite, thus $\pr(h)=\pr(f)$ and the result follows.}

\end{proof}
\begin{remark} Our mod $\ell$ non-vanishing of Hecke $L$-values relies on mod $\ell$ analogue of the above result as in~Lemma~\ref{tnn}.\end{remark}

\subsection{Explicit Waldspurger formulas} \label{S:exW}
The aim of this subsection is to explicitly link Rankin--Selberg $L$-values with toric periods of the test vector introduced in \S\ref{test}.
\subsubsection{Setting}
\label{ss:set-W} 
We begin with generalities regarding Waldspurger formula and then specialise to the prior setting. 

Let $B$ be a quaternion algebra over a number field $F$ and $G=B^\times$. Let $\pi$ be an irreducible cuspidal automorphic representation of $G(\BA)$ with trivial central character and $\sigma$ its base change to $\GL_2(\BA)$. 
Let $K$ be a quadratic field extension of $F$ with an embedding $\iota: K\hookrightarrow B$ and $\chi$ a finite order Hecke character over $K$ such that \[\chi|_{\BA^\times}=1.\]  
Suppose that \[\varepsilon(\pi_K \otimes \chi)=+1.\] 
Then the Waldspurger formula \cite{Wa,YZZ} connects the toric periods
\[P_{\chi}(f)=\int_{K^\times \BA^\times\bs \BA_{K}^\times}\chi(t)f(t)dt,\quad f\in \pi \]
with the Rankin--Selberg $L$-value $L(1/2,\pi_K\otimes \chi)$. {Here $dt$ denotes the Tamagawa measure with volume $2L(1,\eta_{K/F})$. }

Now let the setting, and in particular $B$ and $\pi$, be as in \S\ref{ss:set}. Recall that $\pi$ is of weight $k\in 2\BZ_{\geq 1}[\Sigma]$. Let $\chi\in \CX_{S_0}$ be a finite order anticyclotomic Hecke character over $K$. 
Let $R\subset B$ be an order as in Definition \ref{D:2.W}, and 
pick a test vector $f=\Psi(x_0\otimes\varphi)$ as in Definition \ref{D:tv}, where\[x_0=\otimes_{\sigma\in \Sigma}x_{0_\sigma}\in \ \bigotimes_{\sigma\in \Sigma} V_{k_\sigma-2}(\BC), \qquad x_{0_\sigma}=X^{\frac{k_\sigma-2}{2}}Y^{\frac{k_\sigma-2}{2}}.\]

 Let {$X_{\wh{R}^\times}$} be the Shimura set $ B^\x\wh{F}^\times\bksl \widehat{B}^\times/\widehat{R}^\x$, whose elements may be chosen as a set of representatives in $\widehat{B}^\times=G(\BA_f)$. Define the inner product of $f=\Psi(x_0\otimes\varphi)$ with itself by 
\begin{equation}\label{E:normvf.W}
\pair{ f, f }_{\wh{R}^\times}:=\frac{\pair{x_0,x_{0}}_{k-2\Sigma}}{\prod_{\sigma\in \Sigma}(k_\sigma-1)}\sum_{[g_i]\in X_{\wh{R}^\times}}\frac{1}{w_i}\cdot\pair{\varphi(g_i),\varphi(g_i)}_{k-2\Sigma},\quad w_i:=[B^\x\cap g_i\widehat{R}^\x g_i^{-1}:O^\times].
\end{equation}
For $\phi$ the $\GL_{2}(F)$-newform as normalised in~\cite[Thm.~1.5]{CST} of level $$U_0(\mathfrak{N}):=\left\{\begin{pmatrix}
  a&b\\ c&d
\end{pmatrix}\in M_2(\wh{O})\ |\ \mathfrak{N}|c\right\}$$ associated to $\pi$, the Petersson norm $(\phi,\phi)_{U_0(\mathfrak{N})}$ is defined by {the inner product on the Hilbert modular variety {$\GL_2(F)\bs \BH^{\Sigma}\times \GL_2(\wh{F})/(\wh{F}^\times U_0(\mathfrak{N}))$} arising from the invariant measure $dxdy/y^2$ on each component $\BH$, which denotes the upper half plane.}

{Let $\iota: K\hookrightarrow B$ be the embedding as in \S\ref{deff}.}
Recall that {$\wh{R}\cap\iota_{\varsigma^{(\fr)}}\wh{O}_K=\iota_{\varsigma^{(\fr)}}\wh{O}_{K,\fr}$. For a Hecke character $\chi\in \CX_{S_0}$  of conductor $\fr$, 
the toric period $P_{\chi}(\varsigma^{(\fr)}f)$} with respect to embedding $\iota$ is essentially given by\footnote{{Note that it equals $P_{\chi}^0(\pi(\varsigma^{(\fr)})f):=\sum_{[a]\in {\rm P}_\fr}\chi(a)\pair{x_0, \pi(\varsigma^{(\fr)})\varphi(\iota(a))}_{k-2\Sigma}$ in \cite[Thm.~1.8]{CST}.}}
 \begin{equation}\label{tperd}P(\varsigma^{(\fr)},\varphi,\chi):=\sum_{[a]\in {\rm P}_\fr}\chi(a)\pair{x_0, \varphi(x_\fr(a))}_{k-2\Sigma}, \quad f=\Psi(x_0\otimes\varphi)\end{equation} where  ${\rm P}_\fr=K^\times\wh{F}^\times\bs\wh{K}^\times/\wh{O}_{K,\fr}^\times$ (cf.~\cite[Lem.~2.3]{CST}). 

{If $\pi=\pi_\lambda$ has CM by $K$ in the sense of Example \ref{ex1}, we will also use the notation $\phi_\lambda$, $f_\lambda$ for $\phi$, $f$ respectively to emphasise the dependence on $\lambda$.}
\subsubsection{Explicit Waldspurger formula }\label{S:central.W}
  The main result of this subsection is the following {Waldspurger formula for self-dual pairs $(\pi,\chi)$ as $\chi$ varies in anticyclotomic family with restricted places of ramification}.
\begin{thm}\label{T:central.W}
Let $(\pi,\chi)$ be a self-dual pair as in \S\ref{setting:Wald} with {$\chi\in \CX_{S_0}$} of conductor $\fr=\prod_{v\in S_0}\fp_v^{r_v}$. Let $f=\Psi(x_0\otimes\varphi)$ be an associated test vector as in Definition \ref{D:tv}.  
    Then 
    \[\begin{aligned}{\rm N}(\fr_1)^{-1}\cdot P(\varsigma^{(\fr)},\varphi,\chi)^2=&\frac{\Gamma_{\Sigma}(k)\pair{f,f}_{\wh{R}^\times}}{\pi^k(\phi,\phi)_{U_0(\mathfrak{N})}}\sqrt{|D_K|}\cdot L^{(\Sigma_0)}(\frac{1}{2},\pi_K\otimes \chi) \cdot \\
        & [O_{K,\fr_1}^\times:O^\times]^2\frac{\kappa_{\fr_1}^2}{2^c} \cdot \frac{\varepsilon(\pi)}{\prod_{\fp\in S_0}\varepsilon(\pi_\fp)}\chi_{S_1^+\bs S_0}(\fn^+).\end{aligned}\]  
        Here the notation is as follows: 
        \begin{itemize}
        \item[\tiny{$\bullet$}] {$L^{(\Sigma_0)}(\frac{1}{2},\pi_K\otimes \chi)=L(\frac{1}{2},\pi_K\otimes \chi)\prod_{v\in S_0}L(\frac{1}{2},\pi_{K_v}\otimes \chi_v)^{-1}$ is the incomplete L-value,}
        \item[\tiny{$\bullet$}] $\fr_1=\prod_{{\substack{v\in S_0\\ \text{non-split}}}}\fp_v^{r_v}$, 
        \item[\tiny{$\bullet$}] $\Sigma_0$ is the set of places $v$ such that $v|(\cond(\chi)D_{K/F}, \mathfrak{N})$ 
        and if $v\notin S_0$, then $\ord_v(\mathfrak{N})>1$,
        \item[\tiny{$\bullet$}] $\kappa_{\fr_1}=\#\ker \left(\Pic(O)\ra {\rm P}_{\fr_1}\right)$\footnote{It equals $1$ or $2$.}, 
        \item[\tiny{$\bullet$}] $$c={-\#\Sigma_D+2}+{2\sum_{\sigma\in \Sigma}k_\sigma-3[F:\BQ]}$$ for $\Sigma_D$  the set of places dividing $(D_{K/F},\mathfrak{N})$ which are coprime to $S_0$,
        \item[\tiny{$\bullet$}] $\mathfrak{N}^+=\fn^+\ov{\fn^+}$ with $w|\fn^+$ for $v=w\ov{w}$, {and $\chi_{T}=\prod_{v\in T}\chi_v$.}
 \end{itemize}
\end{thm}
The above result is a consequence of a general explicit Waldspurger formula \cite[Thm.~1.8]{CST}. 
In {\it loc. cit.} the absolute value square of a toric period appears, and the following analysis
relates it to the square of the toric period.

Put \[\tau=\prod_{v|\mathfrak{N}^+, v\notin S_0}\begin{pmatrix}
  \varpi_v^{\ord_v(\mathfrak{N})} & \\ &1
\end{pmatrix}\in \prod_{v|\mathfrak{N}^+, v\notin S_0}\GL_2(F_v).\]
\begin{lem}\label{sc}
  For $\chi\in \CX_{S_0}$ of conductor $\fr$ such that $\mathfrak{N}_{S_0}|\fr_{S_0}$, we have
  \[P(J\varsigma^{(\fr)}\tau,\varphi,\chi)=P(\varsigma^{(\fr)},\varphi,\chi) \frac{\varepsilon(\pi)}{\prod_{\fp\in S_0}\varepsilon(\pi_\fp)}.\] 
  \end{lem}
  \begin{proof}
  Let $S_1$ be a set of places as in \S\ref{spepts} given by the prime factors of $\mathfrak{N}$. 
  
  If $v\notin S_1\cup S_0$ finite, note that $J\in K_v^\times\GL_2(O_v)$ and $\varsigma_{v}=1$, and so the Hecke action of $J$ at $v$ does not change the toric period. 
  For $v\in S_1^+\bs S_0$, we have $$\pi_v(\varsigma_v^{-1}J\varsigma_v \tau) \varphi_v=\pi_v(w_{\pi_v})\varphi_v=\varepsilon(\pi_v)\varphi_{v},$$ where $w_{\pi_v}:=\begin{pmatrix}
    &1 \\ \varpi_v^{\ord_v(\mathfrak{N})}
  \end{pmatrix}\in \GL_2(F_v)$ is the Atkin--Lehner operator (cf.~\cite[Thm.~3.2.2]{Schmidt:newform}).
  If $v| S_1^-\bs S_0$, then we have $\chi_v=1$, $\varsigma_v=1$ and $\Hom_{K_v^\times}(\pi_v,\BC)\neq 0$. Let $P_v$ be a basis of the latter space. Then $J$ acts on $P_v$ by $\varepsilon(\pi_v)\varepsilon(B_v)$ (cf.~\cite[Thm.~4]{DP}). Thus $$P_v(\pi_v(J\varsigma_v )\varphi_v)=\varepsilon(\pi_v)\varepsilon(B_v) P_v(\pi_v(\varsigma_v )\varphi_v).$$ 
  
  Now consider the case $v=\fp\in S_0$. 
Then $\varsigma_\fp^{(\ord_\fp\fr),-1}J\varsigma_\fp^{(\ord_\fp \fr)}$ stabilises $\varphi$
  and so
  \[P(J\varsigma^{(\fr)},\varphi,\chi)=\prod_{v\in S,v\notin S_0}\varepsilon(\pi_v)\varepsilon(B_v)P(\varsigma^{(\fr)},\varphi,\chi).\]
 The result then follows from the fact that $\prod_{v\in S\bs S_0}\varepsilon(\pi_v)\varepsilon(B_v)=\frac{\varepsilon(\pi)}{\prod_{\fp\in S_0}\varepsilon(\pi_\fp)}$.

  \end{proof}

\begin{proof}[Proof of Theorem \ref{T:central.W}]
  The following is based on \cite[Thm.~1.8]{CST}.
  
  Taking the test vector in {\it ibid.} to be $\pi(\varsigma^{(\fr)})f$ with $f=\Psi(x_0\otimes\varphi)$ and letting {for $P_{\chi}^0(\pi(\varsigma^{(\fr)})f):=\sum_{[a]\in {\rm P}_\fr}\chi(a)\pair{x_0, \pi(\varsigma^{(\fr)})\varphi(\iota(a))}_{k-2\Sigma}$, we have}
\[\frac{P_\chi^0(\pi(\varsigma^{(\fr)})f) P^0_{\chi^{-1}}(\ov{\pi}(\varsigma^{(\fr)})\ov{f})}{[O_{K,\fr_1}^\times:O^\times]^2}{\rm N}(\fr_1)^{-1}=\frac{\#\kappa_{\fr_1}^2}{2^c}\cdot \frac{\Gamma_{\Sigma}(k)\pair{f,\ov{f}}_{\wh{R}^\times}}{\pi^k(\phi,\phi)_{U_0(\mathfrak{N})}}\sqrt{|D_K|}\cdot L^{(\Sigma_0)}(1/2,\pi_K\otimes\chi).\]
Indeed, this follows from \cite[Thm.~1.8]{CST} since 
\begin{itemize}
\item[\tiny{$\bullet$}] \[\pair{f_1,f_2}_{H,\wh{R}^\times}=\pair{f_1,\ov{f}_2}_{\wh{R}^\times} \] for $\pair{f_1,f_2}_{H,\wh{R}^\times}$ being the $\wh{R}^\times$-level Hermitian pairing,
\item[\tiny{$\bullet$}] $C_\infty=4^{k-\Sigma}\pi^{k+\Sigma}\Gamma_{\Sigma}(k)^{-1}$, 
\item[\tiny{$\bullet$}] $(2\pi)^{-\Sigma}(\phi,\phi)_{U_0(\mathfrak{N})}=\pair{\phi_0,\phi_0}_{U_0(\mathfrak{N})}$
   and 
   \item[\tiny{$\bullet$}]
\[\nu_{\fr_1}=\# \kappa_{\fr_1}[O_{K,\fr_1}^\times:O^\times]\] as in \cite[Thm.~1.8]{CST}.
\end{itemize}

We also have \[P_{\chi}^0(\pi(\varsigma^{(\fr)})f)=P(\varsigma^{(\fr)},\varphi,\chi).\] 
In the following, we will use the notation $P_{\chi}^0(\pi(\varsigma^{(\fr)})f)$ instead of $P(\varsigma^{(\fr)},\varphi,\chi)$ to emphasis the dependence on $x_0$.
Then in view of the $S$-version of Waldspurger formula \cite[Thm.~1.9]{CST} we have  
  \[\frac{P_\chi^0(\pi(\varsigma^{(\fr)})f) P^0_{\chi^{-1}}(\ov{\pi}(\varsigma^{(\fr)}\tau)\ov{f})}{[O_{K,\fr_1}^\times:O^\times]^2}{\rm N}(\fr_1)^{-1}=\frac{\kappa_{\fr_1}^2}{2^c}\cdot \frac{\Gamma_{\Sigma}(k)\pair{f,\ov{f}}_{\wh{R}^\times}}{\pi^k(\phi,\phi)_{U_0(\mathfrak{N})}}\sqrt{|D_K|}\cdot L^{(\Sigma_0)}(1/2,\pi_K\otimes\chi)\chi_{S_1^+\bs S_0}(\fn^+).\]

We now analyse the left hand side. 
Since $\pi=\ov{\pi}$, by the multiplicity one of $f$ (cf.~Lemma~\ref{lm:t-mult}), we have $\ov{f}=a f$ for some non-zero constant $a\in \BC$. 
Hence 
\[\begin{aligned}
  \frac{P_{\chi}^0(\pi(\varsigma^{(\fr)})f) P_{\chi^{-1}}^0(\ov{\pi}(\varsigma^{(\fr)}\tau)\ov{f})}{\pair{f,\ov{f}}}
  =&\frac{P_\chi^0(\pi(\varsigma^{(\fr)})\Psi(x_0\otimes\varphi)) P_{\chi^{-1}}^0(\pi(\varsigma^{(\fr)}\tau){\Psi(x_{0}\otimes\varphi)})}{\pair{\Psi(x_0\otimes\varphi),\Psi(x_{0}\otimes\varphi)}_{R}}\\
  =&\frac{P_\chi^0(\pi(\varsigma^{(\fr)})\Psi(x_0\otimes\varphi))P_{\chi}^0(\pi(J\varsigma^{(\fr)}\tau)\Psi(v_{0}\otimes\varphi))}{\pair{\Psi(x_0\otimes\varphi),\Psi(x_{0}\otimes\varphi)}_{R}}\\
  =&\frac{\varepsilon(\pi)}{\prod_{v\in S_0}\varepsilon(\pi_v)}\frac{P_\chi^0(\pi(\varsigma^{(\fr)})\Psi(x_0\otimes\varphi))P_{\chi}^0(\pi(\varsigma^{(\fr)})\Psi(x_0\otimes\varphi))}{\pair{\Psi(x_0\otimes\varphi),\Psi(x_{0}\otimes\varphi)}_{R}}\quad (\text{Lemma \ref{sc}}).
\end{aligned}\]
{Here the second equality follows from the {automorphy of $f$} and the fact that $c(t)=J^{-1}t J$ for any $t\in \BA_{K}^\times$.}

The proof concludes.

\end{proof}
\subsection{$\ell$-integrality of Rankin--Selberg $L$-values}
\label{S:ellopt}
{In this subsection we present an $\ell$-integrality of Rankin--Selberg $L$-values in anticyclotomic twist family, based on the Waldspurger formula in \S\ref{S:central.W} and $\ell$-integrality of quaternionic modular forms in \S\ref{S:defform}.}

Let the setting be as in \S\ref{ss:set}. 
We first introduce normalised test vectors.\label{ss:ntv}

Let $\ell$ be a prime that is coprime to the primes in $S_0$. Fix embeddings $\iota_\ell:\ov{\BQ}\hookrightarrow \ov{\BQ}_{\ell}$ and $\iota:\ov{\BQ}\hookrightarrow \BC$.
Let $L$ be a sufficiently large number field {as in \S\ref{avatar}} {so that it contains the Hecke field of $\pi$}. Let
$O_{\pi,\ell}\subset\C_\ell$ be the completion 
of the ring of integers of $L$ with respect to the prime $\fl|\ell$ of $L$ determined by the embedding $\iota_\ell$.

Recall that $x_0=\otimes_{\sigma\in \Sigma}x_{0_\sigma}\in \ \bigotimes_{\sigma\in \Sigma} V_{k_\sigma-2}(O_{\pi,\ell})$ where $x_{0_\sigma}=X^{\frac{k_\sigma-2}{2}}Y^{\frac{k_\sigma-2}{2}}$ is 
$\ell$-primitive {and $K_\sigma^\times$-invariant (cf.~\S\ref{algrep}).}
From now on, fix \[x:=\rho_{k,\ell}(\varsigma_{\ell})^{-1}x_0\] with $\varsigma_{\ell}$ as in \S\ref{gpt}, let $\pi_f$ be the finite part of $\pi$ and let $$\varphi\in {\rm M}_{k}(\wh{R}^\times, L)[\pi_f]^{\prod_{v|\mathfrak{N}^-, v\notin S_0}{\iota_{\varsigma_v} K_v^\times}}$$be a basis such that $f=\Psi(x_0\otimes \varphi)$ is the test vector as in Definition \ref{D:tv}. We choose $\wh{\varphi}$ to be $\ell$-optimal with respect to $x$ in the sense of Definition \ref{opt:ellava}.

For brevity, the present choice of $f$ is said to be $\ell$-optimal.

The above normalization is crucial for $\ell$-integrality of Rankin--Selberg $L$-values. 
We have the associated period
\[
  \Omega_{\pi}:=\frac{\pi^k(\phi,\phi)_{U_0(\mathfrak{N})}}{\Gamma_{\Sigma}(k)\pair{f,f}_{\wh{R}^\times}},
  \]
where the Petersson norms of $\phi$ and $f$ are normalized as in \S\ref{ss:set-W}.

\begin{prop}\label{lbd:RS}{Let $(\pi,\chi)$ be a self-dual pair as in \S\ref{setting:Wald}.}
  {Let $\ell$ be a rational prime that is  coprime to the primes in  $S_0$.}
  Then for any $\chi\in \CX_{S_0}$ with conductor $\fr$ such that $[O_{K,\fr_1}^\times:O^\times]=1$ for $\fr_1=\prod_{\substack{v\in S_0\\ \text{non-split}}}\fp_v^{r_v}$, we have

\[
  {v}_{\ell}\left(2^r \cdot \frac{\sqrt{|D_K|}\cdot L^{(\Sigma_0)}(\frac{1}{2}, \pi_{ K}\otimes \chi)}{\Omega_{\pi}}\right)\geq 0,\]where {$2^r:=\frac{\kappa_{\fr_1}^2}{2^c}$ {and $\Sigma_0$ is} as in Theorem \ref{T:central.W}.}\end{prop}
  \begin{proof} 
    {
    Since $x_0$ is $K_\infty^\times$-invariant, we have \[\begin{aligned}
    \pair{x_0, \varphi(x_\fr(a))}_{k-2\Sigma}
    =\pair{x, \wh{\varphi}(x_\fr(a))}_{k-2\Sigma}.
    \end{aligned}\]}
     {So in view of the $\ell$-optimality of $f$, the toric period 
$P(\varsigma^{(\fr)},\varphi,\chi)$ defined in \S\ref{ss:set-W}  is $\ell$-integral. }Hence, the assertion follows from the explicit Waldspurger formula as in ~Theorem~\ref{T:central.W}.
  \end{proof}
  \begin{remark}
  If the norm of $\fr_1$ is sufficiently large, then $[O_{K,\fr_1}^\times:O^\times]=1.$
  \end{remark}

\section{Non-vanishing of Rankin--Selberg $L$-values modulo $\ell$: CM case}\label{s:nv}
For an odd prime $\ell$, a prime $\fp$ of a totally real number field $F$ coprime to $\ell$ and a CM quadratic extension $K$ of $F$, 
this section presents mod $\ell$ non-vanishing of Rankin--Selberg $L$-values in the $\fp$-anticyclotic twist family of a CM Hilbert newform over $F$. 
{The main results are Theorems \ref{T:b.W} and \ref{T:3.W} if $\fp$ is inert or split in $K$ respectively. We refer to ~\S\ref{ss:lp-nvs} for an outline of the strategy. }

\subsection{Key tools}\label{s:kt}
\subsubsection{Equidistribution of special points}\label{SS:Uniform_CM}We describe an equidistribution of special points on a Shimura set based on Ratner's ergodicity of unipotent flows (cf.~\cite{Vatsal_Cornut:Documenta}).

Let $F$ be a totally real field. 
Fix a totally definite quaternion algebra $B$ over $F$, and a finite prime $\fp$ of $F$ such that $B_\fp$ splits. Fix an $F$-embedding $$\iota: K\hookrightarrow B$$ of a totally imaginary quadratic field extension $K$ of $F$.
Let $F_+^\x\subset F^\times$ be the subgroup of totally positive elements.
Write $\ov{K}^\x$, $\ov{B}^\x$, $\ov{F}_+^\x$ for the closure of $K^\x$, $B^\x$, $F_+^\x$ in $\wh{K}^\x$, $\wh{B}^\x$, $\wh{F}^\x$ respectively.

Put 
$$\CMspace_K:=\ol{K}^\x\bksl \wh{B}^\times , \, \quad X:=\ol{B}^\x\bksl \wh{B}^\times , \, \quad Z:=\ol{F}_+^\times\bksl \wh{F}^{\times}.
$$
The group $\wh{B}^\times$ acts on via right multiplication on ${\rm CM}_K$ and $X$, and via multiplication by the reduced norm $\rm{N}$ on the third space. Similarly, there is a left action of  $\wh{K}^\times$ on ${\rm CM}_K$ and $Z$. 
Let {$i:\CMspace_K\to X$} be the natural  map and $c:X\to Z$ the map induced by the reduced norm.

Let $$\CS\subset \ov{K}^\times\bs\wh{K}^\times$$ be a finite subset of elements that are pairwise distinct modulo {$\wh{K}^{(\fp),\times}\cdot (K^\times F_\fp^\times)$}.
Let $U$ be an open compact subgroup of $\wh{B}^\times$. 
Put
\[X(\CS,U)=\prod_{\tau\in \CS} X/U\text{ and }Z(\CS,U)=\prod_{\tau\in \CS}Z/{\rm{N}}(U).\]
Define
\begin{align*}i_{\CS}:\ &\CMspace_K/U\longrightarrow X(\CS,U),\quad x\mapsto (i(\tau\cdot x))_{\tau\in \CS}\\
\intertext{ and}
c_{\CS}:\ &X(\CS,U)\longrightarrow Z(\CS,U),\quad(x_{\tau})_{\tau\in \CS}\mapsto ({\rm{N}}(x_\tau))_{\tau\in \CS}.\end{align*}

The following key result is a special case of \cite[Cor.\,2.10]{Vatsal_Cornut:Documenta}.
\begin{thm}\label{P:Vatsal_Cornut.W} Let $\CH$ be a $B^\x_\fp$-orbit in $\CMspace_K$ and $\ol{\CH}$ the image of $\CH$ in $\CMspace_K/U$. Then for all but finitely many $x\in\ol{\CH}$, we have
\[i_{\CS}(\wh{O}_K^\x\cdot x)=c_{\CS}^{-1}(\wh{O}_K^\x\cdot \ol{x}),\]
 where $\overline{x}:=c_{\CS}\circ i_{\CS} (x)$.
\end{thm}

In the following we will introduce an explicit $\CS$ for our later application. 

For $\fc$ an integral ideal of $F$ coprime to $\fp$ and $n\geq 0$ an integer, {put ${\rm P}_{\fc,n}:={\rm P}_{\fc\fp^n}=K^\times\wh{F}^\times\bs\wh{K}^\times/\wh{O}_{K,\fc\fp^n}^{\times}$ and ${\rm P}_{\fc\fp^\infty}:=\varprojlim_n {\rm P}_{\fc,n}$.} Let $\Delta$ be the torsion subgroup of ${\rm P}_{\fc\fp^\infty}$ and 
$\Delta^\alg\subset \Delta$  the subgroup of ${\rm P}_{\fc\fp^\infty}$ generated by the image of 
$$\wh{F}^\times\prod_{v|D_{K/F},v\nmid\fp}K^\x_v \prod_{v\mid \fc}O_{K_v}^\times.$$ {Let $D_1$ be a set of representatives of 
$\Delta/\Delta^\alg$ in $\wh{K}^\x$.} 

We recall the following  (cf.~\cite[Sec.~2.2]{Vatsal_Cornut:London}).

\begin{lem}\label{ind}
  The elements in $D_1$ are pairwise distinct modulo image of $\wh{K}^{(\fp),\times}\cdot (K^\times F_\fp^\times)$ in ${\rm{P}}_{\fc\fp^\infty}$.
\end{lem}

\subsubsection{Non-Eisenstein functions: generalities}\label{nonE:gen}
Let $A$ be a commutative ring. 

Let $U$ be an open compact subgroup of $\wh{B}^\times$. Recall that ${\rm M}_2(U,A)$ is the set of functions $$h:X_U:=B^\x\wh{F}^\times \bksl \wh{B}^\times/U\to A$$ and 
${\rm M}_2(A):=\varinjlim_{U\subset \wh{B}^\times}{\rm M}_2(U,A)$ is the space of $A$-valued smooth functions on $B^\x\wh{F}^\times\bksl \wh{B}^\times$. Let $\rho:\wh{B}^\times\to\Aut({\rm M}_2(A))$ denote the right translation of $\wh{B}^\times$. 

Let $B^1$ be the algebraic group over $F$ arising from the norm one elements of $B^\times$. Put 
\[{\rm M}_2(A)_{\Eis}:=\{h\in {\rm M}_2(A)\mid \rho(g_1)h=h\text{ for all $g_1\in B^1(\BA_f)$}\}.\]

\begin{defn} 
A function $h\in {\rm M}_2(A)$ is called Eisenstein if $h\in {\rm M}_2(A)_{\Eis}$, 
i.e. $h(g)=h_1({\rm{N}}(g))$ for a smooth function $h_1:F_+^\x\bksl \wh{F}^\times\to A$.
\end{defn} 

Put
\[{\rm S}_2(A):={\rm M}_2(A)/{\rm M}_2(A)_{\Eis}.\]
Denote by ${\rm S}_2(U,A)$ the image of ${\rm M}_2(U,A)$ in ${\rm S}_2(A)$.

Let $K\hookrightarrow B$ be a totally imaginary quadratic extension of $F$.
If ${\rm{N}}(U)\subset F_{+}^\times {\rm{N}}(\wh{K}^\times)$, then we may write $$X_{U}=X_{U}^+\sqcup X_{U}^-,$$ where $$X_{U}^+:=\{[g]\in X_{U}\ \Big|\ \eta_{K/F}({\rm{N}}(g))=1\}.$$ 
For $f\in {\rm M}_2(U,A)$, write \begin{equation}\label{decomp}f=f^++f^-\end{equation} with $f^\varepsilon$ supported on $X_U^\varepsilon$.
\begin{remark}\label{decom}
{For a higher weight modular form $f\in {\rm M}_{k}(U,A)$, we define a decomposition $f=f^{+}+f^{-}$
analogous to \eqref{decomp}.}
\end{remark}
\subsubsection{An independence of CM values}\label{indeCM}
This subsection describes an application of the equidistribution of special points to a linear independence of CM values. It will be a key to our non-vanishing results.

Let $\fc$ be an integral ideal $F$ coprime to $\fp D_{K/F}D_B$. For $a\in \wh{K}^\times$, let $$x_{\fc,n}(a):=x_{\fc\fp^n}(a)$$ be a family of special points as in \S\ref{sppt} with $S_0$ consisting of prime factors of $\fc$ and $\fp$. 
Let $D_1$ be as in Lemma \ref{ind} and 
$\ell$ a rational prime. 

The main result of this subsection is the following.

\begin{prop}\label{C:Vatsal_Cornut} 
Let $\BF$ be a finite extension of $\BF_{\ell}$.
Let $(\beta_\tau)_{\tau\in D_1}$ be a sequence of elements in $\BF$ such that $\beta_{\tau_1}\neq 0$ for some $\tau_1$.  Let $f\in {\rm M}_2(U,\BF)$ and 
suppose that ${\rm N}(U)\subset F_+^\times{\rm N}(\wh{K}^\times)$. Let $\fc$ be an integral ideal of $F$ coprime to $\fp D_{K/F}D_B$.
\begin{itemize}
\item[a)] Suppose that $\fp$ is inert in $K$ and $f=f^\varepsilon$ is non-Eisenstein. Then there exists an integer $n_0$ such that for any $n>n_0$ of the fixed parity determined by $\varepsilon=(-1)^n\eta_{K/F}(\fc)$, we have
\[\sum_{\tau\in D_1}\beta_\tau\cdot f(x_{\fc,n}(a\tau))\neq 0\] $\text{ for some }a\in \wh{K}^\times$.
\item[b)] Suppose that $\fp$ is split in $K$, and that $f=f^{\sgn(\eta_{K/F}(\fc))}$ is non-Eisenstein. Then there exists an integer $n_0$ such that for any $n>n_0$, we have
\[\sum_{\tau\in D_1}\beta_\tau\cdot f(x_{\fc,n}(a\tau))\neq 0
\] $\text{ for some }a\in \wh{K}^\times$.
\end{itemize}
\end{prop}
\begin{proof} Consider a special point $$P_0:=[\varsigma^{(\fc)}]\in\CMspace_K$$
for $\varsigma^{(\fc)}$ as in \S\ref{gpt} with $S_0$ consisting of prime factors of $\fc$ and $\fp$. Note that $P_0\in X_{U}^{\eta_{K/F}(\fc)}$.
Let $\CH=P_0\cdot B^\x_\fp$ be the $B^\x_\fp$-orbit of $P_0$.  

{ 
For $n\geq 0$, let $\delta_n\in \GL_2(F_\fp)$ be the family of elements as in \S\ref{gpt} and put 
\begin{equation}\label{CM_Iw}
P_n:=P_0
\delta_{n}\in \CH.
\end{equation}}
Note that the images of points $(P_{n})_{n=0,1,\cdots}$  in 
$\ov{K}^\times \bksl \wh{B}^\times/U$ are distinct. 
Hence, by Theorem \ref{P:Vatsal_Cornut.W} and Lemma \ref{ind}, there exists $n_0\gg 0$ such that 
\begin{equation}\label{E:20.W}
i_{D_1}(\wh{O}_K^\x \cdot P_n)=c_{D_1}^{-1}(\wh{O}_K^\x \cdot\ol {P_n})
\end{equation} $\text{ for any }n>n_0$.

Since $f$ is non-Eisenstein, there exist $y,z \in X_U$ such that 
 $$
 f(y) \neq f(z)
 $$
 and $c(y)=c(z)$.
 
 \begin{itemize}
\item[a)] We have $y,z\in X_U^{\varepsilon}$. Note that 
$ c(i(P_{n})) = c(i(P_{n+2})) \pmod{{\rm{N}}(\wh{K}^{\times})} $, and
 $$
 c(y)=c(z) \neq c(i(P_{n})) \pmod{{\rm{N}}(\wh{K}^{\times})} \implies 
  c(y)=c(z) = c(i(P_{n+1}))  \pmod{{\rm{N}}(\wh{K}^{\times})}
$$
  since $\fp$ is inert in $K$. It follows that $$\text{$c(y)=c(z)=c(i(P_{n})) \pmod{{\rm{N}}(\wh{K}^{\times}){\rm{N}}(U)}$ for $n$ of a fixed parity.}$$ 
    In the following we consider $n>n_{0}$ of that parity. 
  
  Replacing $D_1$ by $a'D_1$ for $a' \in \wh{K}^\times$, we may assume that 
  $$c(y)=c(z)=c(i(P_{n})) \pmod{{\rm{N}}(U)}.$$  
  Pick $(w_{\tau})_{\tau \in D_{1}}\in c_{D_{1}}^{-1}(\ov{P}_{n})$. In view of \eqref{E:20.W} there exist $a_{1},a_{2} \in \wh{O}_{K}^\times$ such that 
  $$
  i_{D_{1}}(a_{1}P_{n})=(y,w_{\tau_{2}}, w_{\tau_{3}},\cdots),  \,i_{D_{1}}(a_{2}P_{n})=(z,w_{\tau_{2}},w_{\tau_{3}},\cdots).$$
  Hence 
  $$
  \sum_{\tau \in D_{1}} \beta_{\tau}\cdot f(x_{\fc,n}(a_{1}\tau)) - \sum_{\tau \in D_{1}} \beta_{\tau}\cdot f(x_{\fc,n}(a_{2}\tau)) = \beta_{\tau_{1}} (f(y)-f(z)) \neq 0 ,$$
  and the assertion follows.

\item[b)] Since $\fp$ splits in $K$, we have
$ c(i(P_{n})) = c(i(P_{n+1})) \pmod{{\rm{N}}(\wh{K}^{\times}){\rm{N}}(U)} $. 

In view of the hypothesis we may suppose that \[c(y)=c(z)=c(i(P_n))\pmod{{\rm{N}}(U)}\] 
and $f(y)\neq f(z)$. Then the assertion follows just as in part a). 

\end{itemize}
\end{proof} 
\begin{remark} Our non-vanishing problems concern toric periods of a CM Hilbert modular form of arbitrary weight. Via congruences, the non-vanishing is closely related to that of a parallel weight two  modular form as in the proposition.  
\end{remark}

\subsubsection{Non-Eisenstein functions: under Hecke action}
The following properties of non-Eisenstein functions will be used in our non-vanishing arguments.

Let $U \subset \wh{B}^\times$ be an open compact subgroup of the form $\displaystyle\prod_{v\nmid \infty}U_v$.
Let $A$ be a commutative ring.
\begin{lem}\label{L:Ihara2.W}Let $v$ be a finite place such that $B_v$ splits  and $U_v= U_0(\fp_v^k)_v$ for some $k$. For $\beta_1,\cdots ,\beta_s\in A$, define $\CR\in\End({\rm M}_2(A))$ by
\[\CR=1+\sum_{i=1}^s\beta_i\cdot \rho\left(
\begin{pmatrix}
\varpi_v^{-i} & \\
 & 1\\
\end{pmatrix}
\right).\]
Then $\CR:{\rm S}_2(U,A)\to{\rm S}_2(A)$ is an injective map (cf.~\cite[Lem.~5.5]{CH}).
\end{lem}
  
Let $A$ be the ring of integers of a finite extension of $\BQ_\ell$ and $\varpi$ a uniformiser of $A$.
{We have the following key mod $\ell$ refinement of Lemma~\ref{locavetest}.}

\begin{lem}\label{tnn} {Let $R\subset B$ be an $O$-order.}
Let $\wh{\varphi}\in \CM_{k}(\wh{R}^\times,\Frac(A))$ be 
  the $\ell$-adic avatar of an eigenform $\varphi$ with level $\wh{R}^\times$. {Let $K\subset B$ be an $F$-subalgebra, which is a totally imaginary quadratic extension of $F$} and $\fq$ be a prime of $F$ non-split in $K$
 and $q$ the cardinality of the residue field of $F_\fq$.
  Suppose that $\ell\nmid q(q^2-1)$, and 
  that $\varphi=\varphi_\fq\otimes\varphi^{(\fq)}$ with $\varphi_\fq$ is the newform\footnote{That is the representation $\pi_\fq$ generated by $\varphi_\fq$ is an irreducible representation of $B_\fq^\times$ for which the $R_\fq^\times$-invariant subspace is of dimensional $1$.}
   with level $R_\fq^\times$. Let $\chi_\fq$ be a character of $K_\fq^\times/F_\fq^\times$ valued in $A$. Suppose that 
  \[\gamma:=\frac{1}{\vol(K_\fq^\times/F_\fq^\times)}\int_{K_\fq^\times/F_\fq^\times}\frac{(\pi_\fq(t)\varphi_\fq,\varphi_\fq)_\fq}{(\varphi_\fq,\varphi_\fq)_\fq}\chi_\fq(t)d^\times t\] is an $\ell$-adic unit, where $(\ ,\ )_\fq$ is a 
  non-degenerate $B_\fq^\times$-invariant pairing on $\pi_\fq$. 
  Moreover, suppose that $\wh{\varphi}$ is $\ell$-optimal\footnote{That is, $\pair{\wh{\varphi}, x}_{k-2\Sigma}\pmod{\varpi}$ is a well defined non-zero weight two form.} with respect to some primitive $x\in P_{k-2\Sigma}(A)$. Put
{  \[H=\sum_{\iota\in K_\fq^\times/(F_\fq^\times V_\fq)}\chi_\fq(t) \pi_\fq (t)\wh{\varphi} \]} for $V_\fq:= K_\fq^\times\cap R_\fq^\times$.
Then $H$ is also $\ell$-optimal with respect to $x$.
\end{lem}
\begin{proof}
  Let $\pi$ be the irreducible cuspidal automorphic representation generated by $\varphi$.
Let $R_\fq'\subset R_\fq$ be a suborder such that $R_\fq'^\times$ stabilizes $H$. Put $$H'=\sum_{g\in R_\fq^\times/R_\fq'^\times} \pi_\fq(g) H\in \pi^{\wh{R}^\times}.$$ If $\pair{x,H'}_{k-2\Sigma}$ is non-zero modulo $\varpi$, then so is $\pair{x,H}_{k-2\Sigma}$. In the following we establish the former non-vanishing. 

Let $\pair{\ ,\ }$ be the $G(\BA_f^{(\ell)})$-invariant pairing on {$\varinjlim_{U^{\ell}\subset\wh{B}^{(\ell),\times}}\CM_k(U^{(\ell)}\cdot R_\ell^\times,\Frac(A))$ such that the restriction to $\CM_k(\wh{R}^\times,\Frac(A))$} is given by
\[\pair{f,h}=\sum_{[g]\in B^\times\wh{F}^\times\bs \wh{B}^\times/\wh{R}^\times}\frac{\pair{f(g),h(g)}_{k-2\Sigma}}{w_g},\] where 
$$R_g:=g\wh{R}g^{-1}\cap B \text{ and } w_g:=[R_g^\times:O^\times].$$
Note that the latter is a finite set. In view of the assumption that $\ell\nmid q(q^2-1)$, we have $$w_g \in A^\times,$$ {since $R_g^1:=\{x\in R_g^\times\ |\ {\rm N}(x)=1\}\subset R_{g,\text{tor}}^\times$ has prime factors dividing $q(q^2-1)$ and $R_g^\times/(O^\times\cdot R_g^1)$ is a $2$-group.}

 Note that $\pair{x, H'}_{k-2\Sigma}$ is non-zero modulo $\varpi$ if and only if there exists $H''\in \CM_k(\wh{R}^{\times},\Frac(A))$ with $H''(\wh{B}^\times)\subset A \rho(R_\ell^\times)x$ such that 
\[v_\ell (\pair{H'', H'})= 0.\] 
We have 
\[\begin{aligned}
\pair{H'',H'}
  =&\pair{H'', \sum_{g\in R_\fq^\times/R_\fq'^\times}\pi_\fq(g)H}\\
  =& \# R_\fq^\times/R_\fq'^\times\pair{H'',H}\\
  =& \# R_\fq^\times/R_\fq'^\times\pair{H'', \sum_{t\in K_\fq^\times/(F_\fq^\times V_\fq)}\chi_\fq(t)\pi_\fq(t)\wh{\varphi}}\\
  =&\# R_\fq^\times/R_\fq'^\times \cdot \gamma \#K_\fq^\times/(F_\fq^\times V_\fq) \cdot \pair{H'', \wh{\varphi}}.\\
\end{aligned}
\]
Here the last equality follows from the multiplicity one of invariant pairings on $\pi$ (cf.~\cite[Lem.~2.6]{Jacquet_Langlands:GLtwo}) and the definition of $\gamma$.

Note that $\# R_\fq^\times/R_\fq'^\times$, $\gamma$ and $\#K_\fq^\times/(F_\fq^\times V_\fq)$ are $\ell$-adic units {as $\ell\nmid q(q^2-1)$}. Recall that $\pair{\wh{\varphi},x}_{k-2\Sigma}$ is non-zero modulo $\varpi$. Hence, the existence of $H''\in \CM_k(\wh{R}^\times,\Frac(A))$ such that $H''(\wh{B}^\times)\subset A \rho(R_\ell^\times)x$ and 
$\pair{\wh{\varphi},H''}$ is non-zero modulo $\varpi$ follows, concluding the proof. 

\end{proof}

\subsubsection{Non-Eisenstein functions: CM case} 
This endoscopic case exhibits peculiar features as in~Lemma~\ref{ee} of this subsection.

Let the settings be as in \S\ref{nonE:gen}. Henceforth, we suppose that 
\begin{equation}\label{lev}
{\rm{N}}(U)\subset F_{+}^\times {\rm{N}}(\wh{K}^\times).
\end{equation} Then restricting to $X^\varepsilon_{U}$  
for $\varepsilon\in \{\pm\}$, we define spaces of non-Eisenstein forms ${\rm S}_2(U,A)^\varepsilon$ and ${\rm S}_2(A)^\varepsilon$. For $f\in {\rm M}_2(U,A)$, recall that $f^\varepsilon$ is its restriction to $X^{\varepsilon}_{U}$.

We typically take $A$ to be the ring of integers of a finite extension of $\BQ_\ell$. {Let $\varpi$ be a uniformiser of $A$ and $\BF=A/(\varpi)$ the residue field. }
 
\begin{defn}
{A form $f\in {\rm M}_2(U,\BF)$ has CM by $K$ if
 $$T_{v}f= f\cdot a_{v}\pmod{\varpi}$$ for all but finitely many finite places $v$, where $a_v\in A$ is the Hecke eigenvalue under Hecke operator $T_v$ of the theta series associated to a Hecke character $\lambda$ over $K$.} 
     \end{defn}
     Our mod $\ell$ non-vanishing of Rankin--Selberg $L$-values in the CM case crucially relies on the following non-Eisenstein property.  
     
\begin{lem}\label{ee} Suppose that 
  \begin{itemize}
       \item[\tiny{$\bullet$}] $(\ell, 2D_{K/F})=1$, 
     \item[\tiny{$\bullet$}] {$f\in {\rm M}_2(U,\BF)$} has CM by $K$,   
    \item[\tiny{$\bullet$}] {$f^\varepsilon$} is non-zero {for some} $\varepsilon\in \{\pm \}$.
  \end{itemize}  Then {$f^\varepsilon$} is non-Eisenstein. 
\end{lem}
\begin{proof}
Pick a place $v$ so that 
\begin{itemize}
\item[\tiny{$\circ$}] $\ell \nmid q_v+1$,
\item[\tiny{$\circ$}] $v$ is inert in $K$, 
\item[\tiny{$\circ$}] $f$ is a $T_v$-eigenform and spherical at $v$. 
\end{itemize}
In view of the first hypothesis, such a $v$ exists.
Note that the $T_v$-eigenvalue of $f$ is $0$. In fact, for $v$ inert in $K$, the unramified representation of $\GL_2(F_v)$ associated to an unramified character $\chi_v$ over $K_v^\times$ such that $\chi_v|_{F_v^\times}=\eta_{K_v/F_v}$ is of the form $$\pi(\nu,\nu^{-1}\eta_{K_v/F_v})$$ {for $\nu$ a primitive 
unramified quadratic character of $F_v^\times$ (cf.~\cite[Thm.~4.7]{Jacquet_Langlands:GLtwo}).} In particular, the $T_v$-eigenvalue of a spherical vector is $0$.

Now assume that 
$f^\varepsilon$ is Eisenstein.

By the hypothesis, we have $f^{\varepsilon}(z)\neq 0$ for some $z\in X_U^{\varepsilon}$. Then 
note that 
\[0= T_v f\left(z\begin{pmatrix}1&\\&\varpi_v^{-1}\end{pmatrix}\right)= (q_v+1)f^{\varepsilon}(z)\neq 0\in \BF, \]
where the second congruence just follows from the definition $T_q=\sum_{i=1}^{q_v+1}u_i$ for $u_i\in \GL_2(F_v)$ with ${\rm{N}}(u_i)=\varpi_v$, and the hypothesis that $f^\varepsilon$ is Eisenstein. A contradiction. 
\end{proof}

{The following is a refinement of Lemma~\ref{L:Ihara2.W}.}

\begin{lem}\label{ee3}  {Let $U \subset \wh{B}^\times$ be an open subgroup of the form $\prod_{v<\infty}U_v$.}
Let $v$ be a place of $F$ unramified in $K/F$ {such that $B_v$ splits} and $U_v= U_0(\fp_v^k)_v$ for some $k$. For a commutative ring $A$, let $\beta_1,\cdots ,\beta_s\in A$ and $\CR\in\End({\rm M}_2(A)^\varepsilon)$ be the endomorphism defined by
  \[\CR=1+\sum_{i=1}^s\beta_i\cdot \rho\left(
  \begin{pmatrix}
  \varpi_v^{-\kappa i} & \\
   & 1\\
  \end{pmatrix}
  \right),\quad \kappa=\begin{cases}2,&\quad \text{$v$ inert in $K$,}\\
  1,&\quad \text{$v$ splits in $K$.}\end{cases}\]
  Then the map $\CR:{\rm S}_2(U,A)^\varepsilon \to{\rm S}_2(A)^\varepsilon $ is an injection.
  \end{lem}
  \begin{proof}
    Let $f^\varepsilon \in {\rm S}_2(U,A)^\varepsilon$ be so that $\CR f^\varepsilon\in {\rm M}_2(A)^\varepsilon_\text{Eis}$. In the following\footnote{ See also the proof of~\cite[Lem.~5.5]{CH}.} we show that $f^\varepsilon\in {\rm M}_2(U,A)_{\text{Eis}}^\varepsilon$.
      
 Let $N(O_v)$ be the subgroup of upper triangular matrices. Note that $N(O_v)\subset \Stab(f^\varepsilon)$. Put $$u=\begin{pmatrix}
  \varpi_v^{-1}&\\&1
\end{pmatrix}$$ and 
$$P(u)=-\sum_{i=1}^s\beta_i\cdot \rho(u^{\kappa (i-1)}
).$$ Then $P\in \End({\rm M}_2(A)_{}^\varepsilon)$. 
By the assumption, we have $(1-\rho(u^{\kappa})P(u))f^\varepsilon\in {\rm M}_2(A)^\varepsilon_\text{Eis}$, and so 
\[(1-\rho(u^{\kappa n})P(u)^n)f^\varepsilon\in {\rm M}_2(A)^\varepsilon_\text{Eis}\]
for any $n\geq 1$. 

Note that $(1-\rho(u^{\kappa n})P(u)^n)f^\varepsilon$ and $\rho(u^{\kappa n})P(u)^nf^\varepsilon$ are fixed by $\begin{pmatrix}
  1&x\\&1
\end{pmatrix}$ for any $x\in F_v$ with $\varpi_v^{2n}x\in O_v$ since $u^{-2n}\begin{pmatrix}
  1&x\\&1
\end{pmatrix}u^{2n}=\begin{pmatrix}
  1& \varpi_v^{2n}x\\&1
\end{pmatrix}$. Thus $f^\varepsilon$ is fixed by $\begin{pmatrix}
  1&x\\&1
\end{pmatrix}$ for all $x\in F_v$. By smoothness, the same holds for $\begin{pmatrix}
  1&\\y&1
\end{pmatrix}$ for some $y\in F_v^\times$. Therefore, $f^\varepsilon$ is fixed by $w_0=\begin{pmatrix}
  &y^{-1}\\ -y&
\end{pmatrix}$, 
and in turn by 
\[\begin{pmatrix}
  1&\\x&1
\end{pmatrix}=w_0\begin{pmatrix}
  1&-y^2x\\&1
\end{pmatrix}w_0^{-1}\]for all $x\in F_v$. 

Hence $\SL_2(F_v)$ fixes $f^\varepsilon$ and $f^\varepsilon(tg)=f^\varepsilon(g)$ for all $t\in B^1(F_v)$. By strong approximation, $B^1$ is dense in $B^1(\BA_{f}^{(v)})$, thus $f^\varepsilon(tg)=f^\varepsilon(g)$ for all $t\in B^1(\BA_{f}^{(v)})$. It follows that $f^\varepsilon(tg)=f^\varepsilon(g)$ for all $t\in B^1(\BA_f)$ and hence $f^\varepsilon(gt)=f^\varepsilon(g)$ for all $t\in B^1(\BA_f)$, concluding the proof.  
  \end{proof}

\subsection{Set-up}\label{set:nvrs}
We introduce set-up for the rest of the section. 

\subsubsection{}\label{set}
{Let $\Sigma$ be a CM type of $K$.}
Let $\lambda$ be a self-dual Hecke character over $K$ of infinity type $\Sigma+\kappa(1-c)$ with $\kappa\in \BZ_{\geq 0}[\Sigma]$, {and $\phi_\lambda$ the associated {$\GL_{2}(F)$}-theta series, {which is a Hilbert modular newform} of weight $k=2\Sigma+2\kappa$ and level $U_0(\mathfrak{N}_{\lambda})$ for $\mathfrak{N}_{\lambda}:=D_{K/F}{\rm{N}}_{K/F}(\Cond(\lambda))$.} 

Let $B$ the totally definite quaternion algebra over $F$ so that its discriminant is given by $$D_{B}=\prod_{\eta_{K_{v}/F_v}(-1)=-1}\fp_v$$ 
(cf.~Lemma~\ref{lm:disc}). 
Let $G=B^\times$ and $\pi=\pi_\lambda$ be the cuspidal automorphic representation of $G(\BA)$ associated to $\phi_{\lambda}$ with conductor $\mathfrak{N}_{\lambda}$, {whose existence follows from the Tunnell--Saito theorem (cf.~\S\ref{S:conTS})}.

Let $\fp$ be a prime ideal of $F$ and $p$ the rational prime below $\fp$. Let $\fc$ be an ideal coprime to $p \mathfrak{N}_{\lambda}$. Fix a character $\chi_0$ of ${\rm{G}}_\fc$ of conductor $\fc$ that is trivial on $\prod_{v|\mathfrak{N}_{\lambda}^-, v\nmid\fp}K_v^\times$.
Let $\Xi_\fp$ be the set of finite order characters of $\Gamma_{\fp}$.
Let $m_\lambda=\ord_{\fp}\cond(\lambda_{\fp})$ be the exponential conductor of $\lambda_\fp$. For any $\nu\in\Xi_\fp$ with $\ord_\fp\cond(\nu_\fp)>m_\lambda$, the pair $(\pi_\lambda,\chi_0\nu)$ is self-dual with epsilon factor $+1$ (cf.~Proposition~\ref{char}{~with~$S_0=\{\fp\}\sqcup\{\fq\ |\ \fq|\fc\}$}). 
Define
$$
\Xi_{\lambda\chi_0,\fp}^{\pm}=\{ \nu \in \Xi_{\fp}| \, \varepsilon(\lambda\chi_0\nu)=\pm 1\}.
$$

{Let $\ell$ be a rational prime.} For $\pi=\pi_\lambda$, let $\Omega_\lambda:=\Omega_\pi$ be the {$\ell$-optimally normalized} period as in \S\ref{ss:ntv}.
In this section we study $\ell$-indivisibility of the $L$-{values} \[L^{\alg}(1/2,\pi_{\lambda,K}\otimes \chi_0\nu):=\frac{ \sqrt{|D_K|} \cdot L(1/2, \pi_{\lambda, K}\otimes \chi_0\nu)}{\Omega_{\lambda}},\] 
for primes $\ell$ such that {$(\ell, 2pD_{K/F})=1$} as $\nu\in \Xi_{\lambda\chi_0,\fp}^{+}$ varies.

For $\fp\nmid \mathfrak{N}_{\lambda}^-$, note that the test vector and the period $\Omega_\lambda$ does not depend on $\fp$ (cf.~Remark~\ref{uniform}) . Whenever $\fp|\mathfrak{N}_{\lambda}^-$, we will use the notation
$\Omega_\lambda^{\{\fp\}}$ to emphasise the dependence on $\fp$.

{
In the rest of the section, we will assume: 
\begin{itemize}
  \item $\fp\nmid D_{K/F}$,
  \item $(\ell, 2\fp D_{K/F})=1$.
\end{itemize}}

\subsubsection{Epsilon factors} 
Note that $$L(s,\pi_{\lambda,K}\otimes\chi_0\nu)=L(s{+1/2},\lambda\chi_0\nu)L(s{+1/2},\lambda(\chi_0\nu)^{-1}).$$ 
We have the following explicit description of the global epsilon factor $\varepsilon(\lambda\chi_0\nu)$ and hence that of $\Xi_{\lambda\chi_0,\fp}^{\pm}$.

\begin{lem}\label{epsi:twist}\label{twep}Suppose that $\nu\in \Xi_{\fp}$ has exponential conductor $n> m_\lambda$ at $\fp$. Then {we have}
   \[\varepsilon(\lambda\chi_0\nu)=\begin{cases}(-1)^{n-m_\lambda}\varepsilon(\lambda\chi_0),\quad &{\text{if $\fp$ is inert in $K$}}, \\
  \varepsilon(\lambda\chi_0),\quad &\text{if $\fp$ is  split in $K$}, \end{cases} 
  \] and \[\varepsilon(\lambda\chi_0)=\varepsilon(\lambda)\eta_{K/F}(\fc).\]\end{lem}
\begin{proof}

  Let $\psi: F\bs \BA\ra \BC^\times$ be a non-trivial additive character and for each place $v$ of $F$, put $\psi_{K_v/F_v}=\psi_v\circ {\rm T}_{K_v/F_v}$.
Let $\mu=\lambda\chi_0$ and $\mu^*=\mu\cdot|\cdot|^{1/2}_{\BA_K}$ be the associated unitary character.

We have the following case by case analysis. 

  \begin{itemize}
    \item  [\tiny{$\bullet$}] If $v|\infty$ or $v|\cond(\mu_v^*)$ is non-split in $K/F$, then   $$\varepsilon(\mu_v^*,\psi_{K_v/F_v})=\varepsilon(\mu_v^*\nu_{v},\psi_{K_v/F_v})$$ since $\nu_{v}=1$.
    \item [\tiny{$\bullet$}] Suppose that $v\nmid\fp$ and $v$ is unramified in $K/F$. For $w|v$, let $c_w$ and $n_w$ be the exponential conductor of $\mu_w$ and $\psi_{K_w/F_w}$ respectively. Then we have
    \[\begin{aligned}
      \varepsilon(\mu_v^*\nu_{v},\psi_{K_v/F_v})=&\prod_{w|v}\varepsilon(\mu_w^*\nu_{w},\psi_{K_w/F_w})\\
      =&\prod_{w|v}\varepsilon(\mu_w^*,\psi_{K_w/F_w})\nu_{w}(\varpi_w^{c_w+n_{w}})\\
      =&\prod_{w|v}\varepsilon(\mu_w^*,\psi_{K_w/F_w})\\
      =&\varepsilon(\mu_v^*,\psi_{K_v/F_v}),
    \end{aligned}\]where the second equality follows from definition of the epsilon factor and the third from the fact that $\nu_v|_{F_v^\times}=1$.
    \item [\tiny{$\bullet$}] For $v=\fp$ we have 
    \[\begin{aligned}
      \varepsilon(\mu_v^*\nu_{v},\psi_{K_v/F_v})=\begin{cases}
        (-1)^{n-m_\lambda}\varepsilon(\mu_v^*,\psi_{K_v/F_v})\nu_{v}(\kappa) ,&\quad \text{$\fp$ is inert in $K/F$},\\
        \varepsilon(\mu_v^*,\psi_{K_v/F_v})\nu_{v}(\kappa) ,&\quad \text{$\fp$ splits in $K/F$},
      \end{cases}
    \end{aligned}\]
    where $\kappa\in K_\fp^\times$ such that $\ov{\kappa}=-\kappa$ (cf.~\cite[Prop.~3.7]{MS}).

  \end{itemize}
  
    In view of the above expressions for local epsilon factors, and noting {that} $\nu_{\fp}(\kappa)=1$ as $\kappa$ is torsion in $K_\fp^\times/F_{\fp}^\times$, the first part of the lemma follows. 
 The second part follows by a similar\footnote{Just note that under the condition on $\chi_0$, we have $\prod_{v\mid\fc}\chi_{0,v}(\kappa)=+1$.} argument.
  
\end{proof}
\subsection{Theta elements}\label{ss:lp}
Let the setting be as in \S\ref{set:nvrs}. In particular $\ell\neq p$ is a rational prime. 
\subsubsection{Definition}\label{theta:ele} 
 Let $f=\Psi(x_0\otimes\varphi)$ be the test vector and $\wh{\varphi}$ the $\ell$-adic avatar of $\varphi$ as in \S\ref{ss:ntv}, where we take $S_0=\fp\sqcup\{\fq\ |\ \fq|\fc\}$ in Definition~\ref{D:tv}. Recall that $\wh{\varphi}$ is $\ell$-optimal with respect to $x=\rho_{k,\ell}(\varsigma_{\ell})^{-1}x_0$ and $\pair{x,\wh{\varphi}(\cdot)}_{k-2\Sigma}\in O_{\pi,\ell}$. The latter  is the $\ell$-adic coefficient ring introduced in \S\ref{ss:ntv}, which we enlarge to also {contain} the image of $\chi_0$.

 {Let ${\rm G}_{\fc,n}:={\rm G}_{\fc\fp^n}$ be the Galois group of the  ring class extension of $K$ of conductor $\fc\fp^n$ and ${\rm{G}}_{\fc\fp^\infty}=\varprojlim_n {\rm G}_{\fc,n}$.
The geometrically normalised reciprocity law induces isomorphisms
\begin{equation}\label{art1}r:{\rm{P}}_{\fc,n}=K^\times\wh{F}^\times\bs\wh{K}^\times/\wh{O}_{K,\fc\fp^n}^{\times}\simeq {\rm{G}}_{\fc\fp^n}\end{equation}
and 
\begin{equation}\label{art2}r:{\rm{P}}_{\fc\fp^\infty}=K^\times\wh{F}^\times\bs\wh{K}^\times/\wh{O}_{K,\fc}^{\times,(\fp)}\simeq {\rm{G}}_{\fc\fp^\infty}.\end{equation}  }
{In this section, we will use the canonical identification $r$ and will omit the notation $r$.}
Denote by\[
[\cdot]_{\fc,n}: \wh{K}^\times\ra {\rm G}_{\fc,n},\quad t\mapsto [t]_{\fc,n}\]
the natural projection map. 
 For $t\in \wh{K}^\times$, let $x_{\fc,n}(t)=x_{\fc\fp^n}(t)$ be a family of CM points as in \S\ref{sppt}.

\begin{defn}
For integers $n\geq m_\lambda$, 
the $n$-th theta element associated to the self-dual pair $(f, \chi_0)$ is given by
\begin{align*}\Thetam_{\chi_0,n}(\pi_\lambda):=
\displaystyle 
 \displaystyle \sum_{t\in {\rm G}_{\fc,n}}\chi_0(t)\pair{x,\wh{\varphi}(x_{\fc,n}(t))}_{{k-2\Sigma}}\cdot [t]_{\fc,n} \in O_{\pi,\ell}[{\rm G}_{\fc,n}].
\end{align*} 
\end{defn}
The theta element is related to toric periods by the following.
\begin{lem}\label{theta-tp}
For any character $\nu\in \Xi_\fp$ of {exponential conductor $n:=\ord_\fp\cond(\nu)\geq \ord_\fp \mathfrak{N}_\lambda$}, we have
\[\nu(\Thetam_{\chi_0,n}(\pi_\lambda))=P(\varsigma^{(\fc\fp^n)},f,\chi_0\nu).\]where $P(\varsigma^{(\fc\fp^n)},f,\chi_0\nu)$ is the toric period defined as in \eqref{tperd}.
\end{lem}
\begin{proof}
Note that 
\[\begin{aligned}
  \nu(\Thetam_{\chi_0,n}(\pi_\lambda))=& \sum_{t\in {\rm G}_{\fc,n}}\pair{x,\wh{\varphi} (x_{\fc,n}(t))}_{k-2\Sigma}\cdot \nu\chi_0(t)\\
  =& \sum_{t\in {\rm G}_{\fc,n}}\pair{\rho_{k,\ell}(t_\ell \varsigma_{\ell}) x,{\varphi} (x_{\fc,n}(t))}_{k-2\Sigma}\cdot \nu\chi_0(t)\\
  =& \sum_{t\in {\rm G}_{\fc,n}}\pair{ x_0,{\varphi} (x_{\fc,n}(t))}_{k-2\Sigma}\cdot \nu\chi_0(t)\\
  =&P(\varsigma^{(\fc\fp^n)},f,\nu\chi_0), 
\end{aligned}\]where the second last equality follows from the fact that $x_0$ is a $K^\times$-invariant vector {(cf.~\S\ref{algrep})}.
\end{proof}

\subsubsection{}\label{reduce}
In view of the explicit Waldspurger formula as in ~Theorem~\ref{T:central.W} and Lemma~\ref{theta-tp}, we have an explicit relation between the theta element and the central Rankin--Selberg $L$-values. 
So it suffices to consider $\ell$-indivisibility of the specialisations $$\nu(\Theta_{\chi_0,n}(\pi_\lambda))$$ as $\nu \in \Xi_{\lambda\chi_0,\fp}^+$ varies.

{Recall that $\mathfrak{N}_{\lambda}$ is the conductor of $\pi_\lambda$.}
We often fix a Hecke character $\chi_0$ of ${\rm G}_{\fc}$ such that the following conditions hold, which are collectively referred to as the condition \eqref{Ass(D)}:

\begin{equation}\tag{D}\label{Ass(D)}
\begin{aligned}
\text{\tiny{$\bullet$}}\ &\text{$\fc=\cond(\chi_0)$ is coprime to $\ell \fp \mathfrak{N}_{\lambda}$ and only divisible by places inert in $K/F$, }\\
\text{\tiny{$\bullet$}}\ &\text{$\ell\nmid \prod_{v|\cond(\chi_0)}(q_v^2-1)$},\\ 
\text{\tiny{$\bullet$}}\ &\text{$\chi_{0,v}=1$ for $v|\mathfrak{N}_{\lambda}^-$ {but} $v\nmid \fp$.}\\
\end{aligned}
\end{equation}
\subsection{$(\ell,p)$ non-vanishing}\label{deg1}

\subsubsection{Inert case} 
\label{rk1int}
Let $\fp$ be a prime of $F$ that is inert in $K$.

For $\nu\in \Xi_\fp$ with $n:=\ord_{\fp}\cond(\nu_\fp)>m_\lambda=\ord_{\fp}\cond(\lambda_{\fp})$, we have
$$
\varepsilon(\lambda\chi_0\nu)=(-1)^{n-m_\lambda}\varepsilon(\lambda\chi_0) 
$$ by Lemma \ref{epsi:twist}. It follows that $\# \Xi_{\lambda\chi_0,\fp}^{+}=\infty$.

Our main result is the following. 

\begin{thm}\label{T:b.W}  Let $K$ be a CM field and $F$ its maximal totally real subfield. 
  Let $\lambda$ be a self-dual Hecke character over $K$ and {$\sigma_\lambda$} the associated $\GL_{2}(F)$-cuspidal automorphic representation of conductor $\mathfrak{N}_{\lambda}$. Let $\fp$ be a degree one prime of $F$ that is inert in $K$ and $\ell$ a prime such that $(\ell, 2\fp D_{K/F})=1$. Let $\chi_0$ be a finite order anticyclotomic Hecke character over $K$ such that the condition \eqref{Ass(D)} holds. Then for all but finitely many $\nu\in\Xi_{\lambda\chi_0,\fp}^{+}$, we have 
  $$
  {v}_{\ell}\left(\frac{ \sqrt{|D_K|} \cdot L(1/2, \sigma_{\lambda, K}\otimes\chi_0\nu)}{\Omega_{\lambda}^{\cdot}}\right)=0
  $$ for
\[\Omega_{\lambda}^{\cdot}:=\begin{cases}
  \Omega_{\lambda}&\quad \fp\nmid \mathfrak{N}_{\lambda}^-,\\
\Omega_{\lambda}^{\{\fp\}}&\quad \fp\mid \mathfrak{N}_{\lambda}^-.\\
  \end{cases}\]
  
  \end{thm}

\begin{proof} 
Let $p$ be the rational prime below $\fp$.
 Choose a finite {extension} $\CO$ of $\BZ_\ell$ in $\ov{\BZ}_\ell$ so that $\CO$ contains $O_{\pi_{\lambda},\ell}$, the image of ${\chi_0}$ {and $\mu_{p}$}, where $\mu_{p^k}\subset \ov{\BQ}^\times$ is the group of $p^k$-th roots of unity. 
 
 {Let $s\geq 1$ be the integer such that $\CO^\times[p^{\infty}]=\mu_{p^s}\subset \ov{\BQ}$. Let $\varpi$ be a uniformiser of $\CO$ and put ${\bf k}_\ell=\CO/\varpi {\CO}\subset\ov{\BF}_{\ell}$. Since $\fp\nmid \ell$, the reduction map $\CO\ra {\bf k}_\ell, x\mapsto x\pmod{\varpi}$ induces an isomorphism  \begin{equation}\label{rootu}\CO^\times[p^\infty]\simeq {\bf k}_\ell^\times[p^\infty]\end{equation}} of abelian groups.

We have an $\ell$-optimal test vector $f=\Psi(x_0\otimes\varphi)$ as in~\S\ref{theta:ele}.
For the order $R$ in the definition of $f$, let
$$U=\{g\in \wh{R}^{\times}\ |\ g_\fl\equiv 1\pmod{\fl},\forall \fl|\ell\}.$$
In particular, $\pair{x,\wh{\varphi^{}} (g)}_{k-2\Sigma}\pmod{\varpi}\in {\rm M}_2(U, {\bf k}_\ell)$ is a non-zero CM form. 

Fix the identification between relative Picard groups and ring class Galois groups as in \eqref{art1} and \eqref{art2}.
Recall that we have 
\[\begin{aligned}
  \Thetam_{\chi_0,n}(\pi_\lambda)=&
\displaystyle 
 \displaystyle \sum_{t\in {\rm G}_{\fc,n}}\chi_0(t)\pair{x,\wh{\varphi}(x_{\fc,n}(t))}_{k-2\Sigma}\cdot [t]_{\fc,n}.
\end{aligned}\]
and in view of the explicit Waldspurger formula as in ~Theorem~\ref{T:central.W}, Proposition \ref{lbd:RS} and Lemma \ref{theta-tp}, it suffices to show that \begin{equation}\label{nvv}{v}_{\ell}(\nu(\Theta_{\chi_0,n}(\pi_\lambda))) = 0\end{equation} for all but finitely many $\nu \in \Xi_{\lambda\chi_0,\fp}^+$. 
{In the following we first reduce this non-vanishing to that of the right hand side of ~\eqref{avered} by considering the Galois average of toric periods.}
\vskip2mm
\underline{Galois average}.
{ Let $\nu\in \Xi_\fp$ be a character of conductor $\fp^n$ with {$n>m_\lambda$}.} Let {${\bf k}_\ell(\nu)\subset \ov{\BF}_{\ell}$} be the subfield generated by the values of $\nu$ over ${\bf k}_\ell$. Since ${\bf k}_\ell$ {contains $p$-th roots of unity}, $d_\nu:=[{\bf k}_\ell(\nu):{\bf k}_\ell]$ is a $p$-power. For a $p$-power root of unity $\zeta\in {\bf k}_\ell(\nu)$, we have
\begin{equation}\label{eqt}{\rm Tr}_{{\bf k}_\ell(\nu)/{\bf k}_\ell}(\zeta)=\begin{cases}0,& \zeta\not\in {\bf k}_\ell,\\
d_\nu\zeta,& \zeta\in {\bf k}_\ell.\end{cases}\end{equation}
Let ${\rm{G}}_{\fc,n}'$ be the subgroup of ${\rm G}_{\fc,n}$ given by 
\[{\rm G}_{\fc,n}'=\{t\in {\rm G}_{\fc,n}\ |\  \nu(t)\in \mu_{p^s}\}.\]

In view of \eqref{eqt} and \eqref{rootu}, for $a\in \wh{K}^\times$, we have 
\begin{equation}\label{avered}
{\rm Tr}_{{\bf k}_\ell(\nu)/{\bf k}_\ell}(\nu\chi_{0}(a^{-1})\cdot\nu(\Theta_{\chi_0,n}(\pi_\lambda))\pmod{\varpi})=d_\nu\cdot\sum_{t\in {\rm G}_{\fc,n}'}\chi_{0}\nu(t)\pair{x,\wh{\varphi}(x_{\fc,n}(at))}_{k-2\Sigma} \pmod{\varpi}.
\end{equation}
Hence the non-vanishing \eqref{nvv} is a consequence of that of the right hand side of \eqref{avered} for some auxiliary $a$, which will be the focus of the rest of the proof. 
\vskip2mm 
\underline{The structure of ${\rm G}_{\fc,n}'$}.
In this paragraph we analyse the structure of ${\rm G}_{\fc,n}'$ under our assumption that
$\deg \fp=1$. 
A key result:  
 ${\rm G}_{\fc,n}'$ only depends on $n$ and $\#{\rm G}_{\fc,n}'$ is independent of $n$.
 
 Let $\Delta\subset {\rm{G}}_{\fc\fp^\infty}$ be the torsion subgroup.
{We have the short exact sequence 
\begin{equation}\label{sact}0\ra\Delta\ra {\rm{G}}_{\fc\fp^\infty}\rightarrow\Gamma_{\fp}\ra 0,\end{equation}where the third map is natural quotient map{ and $\Gamma_\fp\simeq \BZ_p$.}
For $n>0$, let $\Gamma_{\fp,n}$ be the quotient of $\Gamma_\fp$ by image of $1+\fp^n\CO_{K_\fp}$ under the identification \eqref{art1} and the quotient map in \eqref{sact}. Then for $n\gg0$, \eqref{sact} induces a short exact sequence 
\begin{equation}\label{sact2}
0\ra\Delta\ra {\rm{G}}_{\fc,n}\ra \Gamma_{\fp,n}\ra 0,
\end{equation}where the map ${\rm{G}}_{\fc,n}\ra \Gamma_{\fp,n}$ is the natural quotient map.}
{Since $\deg\fp=1$, we have {\begin{equation}\label{degree1}\{\gamma\in \Gamma_{\fp,n}\ |\ \nu(\gamma)\in \mu_{p^s}\subset \ov{\BQ}\}=\Gamma_{\fp,n}[p^s].\end{equation}}}
Note that there exists $n_0$ such that for $n\geq n_0$, \eqref{sact2} holds and the natural map 
\begin{equation}\label{emb}\frac{1+\fp^{n-s}O_{K_\fp}}{(1+\fp^{n}O_{K_\fp})(1+\fp^{n-s})}\ra {\rm{P}}_{\fc,n}\simeq {\rm{G}}_{\fc,n}\end{equation}
is injective, {where the identification ${\rm{P}}_{\fc,n}\simeq {\rm{G}}_{\fc,n}$ is as in \eqref{art2}.} So we have canonical identification
\begin{equation}\label{sact3}\frac{1+\fp^{n-s}O_{K_\fp}}{(1+\fp^{n}O_{K_\fp})(1+\fp^{n-s})}\simeq \Gamma_{\fp,n}[p^s]\end{equation} under the the map \eqref{emb} and the quotient map in \eqref{sact2}.
Then it follows that
\[{\rm{G}}_{\fc,n}'=\Delta\times \frac{1+\fp^{n-s}O_{K_\fp}}{(1+\fp^{n}O_{K_\fp})(1+\fp^{n-s})}\subset {\rm G}_{\fc,n}.\]
{Hence, ${\rm{G}}_{\fc,n}'$ depends only on $n$ and its cardinality is independent of $n$.}

Also note that 
$$\Delta=D_1\times \Delta^{\text{alg}},$$ where $D_1$, $\Delta^{\text{alg}}$ are as in \S\ref{SS:Uniform_CM}. {Define a subgroup \begin{equation}\label{subgp}{\rm H}_{\fc,n}:=\Delta^{\alg}\cdot \frac{1+\fp^{n-s}O_{K_\fp}}{(1+\fp^{n}O_{K_\fp})(1+\fp^{n-s})}\subset {\rm{G}}_{\fc,n}',\end{equation} which is the image of an open compact subgroup of $K_S^\times/F_S^\times$ for some finite set of places $S$ of $F$.}

\vskip2mm
\underline{A Hecke-modification of the test vector {outside $\fp$}}. {In this paragraph we introduce a Hecke-modification of the test vector at places dividing $\fc$, which is eigen under the $\Delta^{\text{alg}}$-action.}

Put $V_0^{(\fp)}=\prod_{v|\fc}O_{K_v}^\times/O_{K_v,\fc_v}^\times$. Note that $\# V_0^{(\fp)}$ is an $\ell$-adic unit.
Define 
\begin{equation}\label{partei}H_\ell(g)=\frac{1}{\# V_0^{(\fp)}}\sum_{u\in V_0^{(\fp)}}\chi_0(u)
  \pair{x,\wh{\varphi^{}} (g\cdot \iota_{\varsigma_\fc^{(\fc)}}(u))}_{k-2\Sigma},\quad g\in \wh{B}^\times.\end{equation} In view of the assumption 
$\ell\nmid\prod_{v|\fc}(q_v^2-1)$, Lemma \ref{tnn} and \cite[Prop.~3.12]{CST}\footnote{{Note that \cite[Prop.~3.12]{CST} implies that the $\gamma$ in Lemma \ref{tnn} is an $\ell$-adic unit since $\ell\nmid\prod_{v|\fc}q_v(q_v^2-1)$.}}, 
it follows that $H_\ell(g)\pmod{\varpi}\in {\rm M}_2(U', {\bf k}_\ell)$ is also a non-zero CM form, where   
  $U' \subset \wh{B}^\times$ is an open compact subgroup such that $U'^{(\fc)}=U^{(\fc)}$. 
 By the definition of $\Delta^{\alg}$, we have 
 \begin{equation}\label{pav0}\sum_{t\in {\rm G}_{\fc,n}}\chi_0\nu(t)\pair{x,\wh{\varphi}(x_{\fc,n}(at))}_{k-2\Sigma} \equiv\sum_{t\in {\rm G}_{\fc,n}}\chi_0\nu(t)H_\ell(x_{\fc,n}(at))\pmod{\varpi},\qquad \forall a\in \wh{K}^\times\end{equation}
 and 
\begin{equation}\label{pav}\sum_{t\in {\rm G}_{\fc,n}'}\chi_{0}\nu(t)\pair{x,\wh{\varphi}(x_{\fc,n}(at))}_{k-2\Sigma}\equiv\sum_{t\in {\rm G}_{\fc,n}'}\chi_{0}\nu(t)H_\ell(x_{\fc,n}(at))\pmod{\varpi},\qquad \forall a\in \wh{K}^\times.\end{equation}

In view of the definition of $f$, note that $H_\ell(x_{\fc,n}(\cdot))$ is $\Delta^{\text{alg}}$-eigen with eigen-character $\chi_0^{-1}$, and 
$\#\Delta^{\text{alg}}$ is an $\ell$-adic unit by our assumption on $\ell$. Thus for $\nu\in \Xi_\fp$ with $n=\ord_\fp\cond(\nu)\geq n_0$, we have
\begin{equation}\label{pav1}\sum_{t\in {\rm G}_{\fc,n}'}\chi_{0}\nu(t)H_\ell(x_{\fc,n}(at))\equiv \#\Delta^{\text{alg}}\sum_{t\in D_1\times \frac{1+\fp^{n-s}O_{K_\fp}}{(1+\fp^{n}O_{K_\fp})(1+\fp^{n-s})}}\chi_{0}\nu(t)H_\ell(x_{\fc,n}(at))\pmod{\varpi},\qquad a\in \wh{K}^\times.\end{equation}
\vskip2mm
\underline{A partition of $\Xi_{\fp}$ and Hecke-modification of the test vector at $\fp$}.
In this paragraph we partition $\Xi_{\fp}$ into finitely many subfamilies indexed by {primitive additive characters 
$\mu: \BZ/p^s\BZ\twoheadrightarrow  \mu_{p^s}\subset \ov{\BQ}$}, and introduce a Hecke-modification $H_{\ell,\mu}$ of $H_\ell$ at $\fp$ which is eigen under the action of  $\frac{1+\fp^{n-s}O_{K_\fp}}{(1+\fp^{n}O_{K_\fp})(1+\fp^{n-s})}$.

We always identify $F_{\fp}$ with $\BQ_p$ and the matrices appearing in the Hecke action are seen as elements of {$B^\times(F_\fp)\simeq\GL_2(\BQ_p)$, where the isomorphism is as in \S\ref{S:totq}}. We choose a local uniformiser of $F_\fp$ to be $p$.

Let $\{1,\alpha\}$ be a relative integral basis of $O_{K_\fp}$ as in ~\S\ref{spepts}. Note that there  exists a sufficiently large integer $n_1\geq s$ such that for any integer $n\geq n_1$ the following holds: 
we have an isomorphism \[\BZ/p^s\BZ\simeq \frac{1+\fp^{n-s}O_{K_\fp}}{(1+\fp^{n}O_{K_\fp})(1+\fp^{n-s})},\quad u\mapsto 1+\alpha p^{n-s} u.\] For any $u\in \BZ/p^s\BZ=O_\fp/\fp^s$, we have
\begin{equation}\label{CMH}
      \iota(1+\alpha p^{n-s}u)\varsigma_{\fp}^{(n)}\equiv \varsigma_{\fp}^{(n)}\begin{pmatrix}
        1&\frac{u}{p^s}\\& 1
      \end{pmatrix}\pmod{U_0(\mathfrak{N}_{\lambda, \fp})_\fp}
    \end{equation} 
    as elements in $\GL_2(F_\fp)$. {Put}
    \[\nu(1+\alpha p^{n-s}u)=\zeta_\nu^u\] where \[\zeta_\nu:=\nu(1+\alpha p^{n-s})\in \mu_{p^s}\subset \ov{\BQ}^\times\] 
    {is a primitive $p^s$-th root of unity}.

In view of \eqref{CMH} and the definition of $\zeta_\nu$, it thus follows that  

\begin{equation}\label{unco}
\sum_{t\in\frac{1+p^{n-s}O_{K_\fp}}{(1+p^{n}O_{K_\fp})(1+\fp^{n-s})}}\nu(t)H_\ell(x_{\fc,n}(at))= \sum_{y\in\BZ/p^s\Z}\zeta_\nu^y H_{\ell}\left(x_{\fc,n}(a\tau)
\begin{pmatrix}
1&\frac{y}{p^s}\\
0&1\\
\end{pmatrix}\right)\qquad \forall a\in \wh{K}^\times.
\end{equation}
Consequently, in view of \eqref{avered}, \eqref{pav} and \eqref{pav1}, for {$n> m_0:=\max\{n_0, n_1\}$} and {$a\in \wh{K}^\x$}, we have 
\begin{equation}\label{E:9.W0}\begin{aligned}
&{\rm Tr}_{{\bf k}_\ell(\nu)/{\bf k}_\ell}(\nu\chi_{0}(a^{-1})\cdot\nu(\Theta_{\chi_0,n}(\pi_\lambda)\pmod{\varpi})\\
=&\#\Delta^{\text{alg}}d_\nu\cdot \sum_{\tau\in D_1}{\chi_0}(\tau)\sum_{y\in\BZ/p^s\Z}\zeta_\nu^y H_{\ell}\left(x_{\fc,n}(a\tau)
\begin{pmatrix}
1&\frac{y}{p^s}\\
0&1\\
\end{pmatrix}\right)\pmod{\varpi}).
\end{aligned}
\end{equation}

{Let $$\mu:\BZ/p^s\BZ\twoheadrightarrow \mu_{p^s}\subset\ov{\BQ}^\times$$ be a given character.} Let $\Xi_{\fp}(\mu)\subset \Xi_{\fp}$ be the subset consisting of characters $\nu$ such that $\ord_\fp\cond(\nu)> m_0$ and $\zeta_\nu^{\cdot}=\mu(\cdot)$. Define 
$$\Xi_{\lambda\chi_0, \fp}^\pm(\mu)=\Xi_{\lambda\chi_0, \fp}^\pm\cap \Xi_{\fp}(\mu).$$ Note that $\Xi_{\lambda\chi_0, \fp}^+(\mu)$ and $\Xi_{\lambda\chi_0, \fp}^-(\mu)$ have infinitely elements by Lemma~\ref{epsi:twist} since $\fp$ is inert.

Define
\begin{equation}\label{tv-tw}
 H_{\ell,\mu}(g)=\frac{1}{p^s}\sum_{y\in\BZ/p^s\BZ}\mu(y)\pi_{\lambda,\fp}\left(
\begin{pmatrix}
1& \frac{y}{p^s}\\
0 &1\\
\end{pmatrix}
\right)
H_{\ell}(g),\quad  g\in\wh{B}^\times.
\end{equation}
{Note that for each $\nu\in \Xi_\fp(\mu)$ and $n:=\ord_\fp\cond(\nu)$, the form 
$H_{\ell,\mu}(x_{\fc,n}(\cdot))$ is $(\chi_0\nu)^{-1}$-eigen under the action of ${\rm H}_{\fc,n}$ as in \eqref{subgp}.} {In view of the eigen-property,  
the remark just below \eqref{subgp} and Lemma~\ref{locavetest}, it follows that 
$$H_{\ell,\mu}\neq 0.$$}
By definition, we have  
$ H_{\ell,\mu}\pmod{\varpi}\in {\rm M}_2(\wtd{U}',{\bf k}_\ell(\nu))$ for
$\wtd{U}':=U'^{(\fp)}\cdot U_1(\fp^{k_0})$. Let $\nu\in \Xi_{\lambda\chi_0, \fp}(\mu)$. 
Then in view of \eqref{E:9.W0}, we have 
\begin{equation}\label{E:10'.W} 
{\rm Tr}_{{\bf k}_\ell(\nu)/{\bf k}_\ell}(\nu\chi_0(a^{-1})\cdot\nu(\Theta_{\chi_0,n}(\pi_\lambda))\equiv \#\Delta^{\alg} p^s d_\nu\cdot\sum_{\tau\in D_1}{\chi_0}(\tau) H_{\ell,\mu}(x_{\fc,n}(a\tau))\pmod{\varpi}.
\end{equation}
\vskip2mm
\underline{Being non-Eisenstein}.
We next show that $ H_{\ell,\mu}\pmod{\varpi}$ is non-Eisenstein. 

{To begin, \eqref{lev} holds by Lemma \ref{lm:lev}.} 
Therefore, since $\pi_\lambda$ is CM and $(\ell,2D_{K/F})=1$, the non-vanishing of $H_\ell\pmod{\varpi}$ implies that it is non-Eisenstein by Lemma \ref{ee}. A simple calculation shows that
\begin{equation}\label{sum0}\begin{aligned}
    \sum_{t\in(\Z/p^s\Z)^\x}\rho
  \left(\begin{pmatrix}
  t & \\
   & 1\\
  \end{pmatrix}\right) H_{\ell,\mu}=&p^{s}H_{\ell}-{p^{s-1}}\rho\left(\begin{pmatrix}
      p^{-1}& \\ & 1
  \end{pmatrix}\right)U_\fp 
      H_{\ell}.
  \end{aligned}\end{equation} 
 If $\lambda_\fp$ is {unramified}, we may further write the above equality as 
  \begin{equation}\label{sum}\sum_{t\in(\Z/p^s\Z)^\x}\rho
  \left(\begin{pmatrix}
  t & \\
   & 1\\
  \end{pmatrix}\right) H_{\ell,\mu}= p^{s}H_{\ell}-{p^{s-1}}\rho\left(\begin{pmatrix}
    p^{-1}& \\ & 1
\end{pmatrix}\right)T_\fp 
    H_{\ell}+{p^{s-1}}\rho\left(\begin{pmatrix}
      p^{-2}& \\ & 1
  \end{pmatrix}\right) 
      H_{\ell}.\end{equation}
Since $H_\ell$ is {either} a $U_\fp$-eigenform or a $T_\fp$-eigenform with an $\ell$-integral eigenvalue and {$H_\ell\pmod{\varpi}$ is non-Eisenstein}, it follows that $ H_{\ell,\mu}\pmod{\varpi}$ is non-Eisenstein by Lemma \ref{L:Ihara2.W}. Then we may suppose that $ H_{\ell,\mu}^{\varepsilon}\pmod{\varpi}$ is non-Eisenstein for some $\varepsilon\in \{+,-\}$.
\vskip2mm
\underline{The non-vanishing}.
In view of \eqref{E:10'.W} and Proposition \ref{C:Vatsal_Cornut}, it thus follows that 
$$
{v}_{\ell}(\nu(\Theta_{\chi_0,n}(\pi_\lambda))) = 0
$$
for all but finitely many {$\nu \in \Xi_{\lambda\chi_0, \fp}^{{\varepsilon_0}}(\mu)$} such that $n:=\ord_\fp\cond(\nu)$ satisfies 
  \begin{equation}\label{comsign}\varepsilon=(-1)^n\eta_{K/F}(\fc).\end{equation} 
Note that  $\varepsilon_0:=\varepsilon(\lambda\chi_0\nu)=(-1)^{n-m_{\lambda}}\eta_{K/F}(\fc)\varepsilon(\lambda)$ 
  by Lemma~\ref{epsi:twist}. So the condition \eqref{comsign} is equivalent to \begin{equation}\label{comsign2}\varepsilon_0=(-1)^{m_\lambda}\varepsilon(\lambda){\varepsilon}.\end{equation}

If $\varepsilon_0= -$, then $$L(1, \lambda\chi_0\nu)\neq  0$$ for all but finitely many 
$\nu \in \Xi_{\lambda\chi_0, \fp}^{-}(\mu)$. But these central $L$-values vanish since $\varepsilon(\lambda\chi_0\nu)=-1$ for any 
$\nu\in \Xi_{\lambda\chi_{0},\fp}^{-}(\mu)$ and $\#\Xi_{\lambda\chi_{0},\fp}^{-}(\mu)=\infty$. Therefore $\varepsilon_0=+$ and the parity $n$ in Proposition \ref{C:Vatsal_Cornut} satisfies 
$(-1)^{n-m_{\lambda}}\varepsilon(\lambda\chi_{0})=+1$. Moreover, it follows from \eqref{comsign2} that ${H}_{\ell,\mu}^{\sgn((-1)^{m_{\lambda}}\varepsilon(\lambda))}\pmod{\varpi}$ is non-Eisenstein and ${H}_{\ell,\mu}^{\sgn(-(-1)^{m_{\lambda}}\varepsilon(\lambda))}\pmod{\varpi}$ is Eisenstein, {where $\sgn(\pm 1):=\pm$.}
\end{proof}

\begin{remark}\label{geq1}If $\deg\fp>1$, then the group ${\rm G}_{\fc,n}'$ appearing in \eqref{avered} is more complex: {it has a significant contribution at $\fp$ given by $\ker(\nu)$ since $\Gamma_{n,\fp}$ is a direct sum of $\deg\fp$-many cyclic $p$-groups}. In particular $\#{\rm G}_{\fc,n}'$ also depends on $n$.
Consequently, the existence of an analogous common test vector $H_{\ell,\mu}$ as in the above proof is not apparent, and it will be considered in the sequel \cite{BHTY}.
\end{remark}
\subsubsection{Restriction of test vector to components of Shimura set}\label{S:restest}

We describe some consequences of the proof of Theorem \ref{T:b.W}, which lead to non-Eisenstein-ness of the test vector on specific component of the Shimura set determined by the epsilon factor and 
will be crucially used in our non-vanishing arguments.

Let $\varphi\in {\rm M}_k(\wh{R}^\times,\ov{\BQ})$ be the test vector as in \S\ref{theta:ele} associated to $\lambda$. {Suppose that $\fp$ is coprime to $\mathfrak{N}_{\lambda}^-$. Then $\varphi$ is independent of $\fp\fc$ in the sense of ~Remark~\ref{uniform}.} 
{Let $\varphi^\pm$ be the restriction of $\varphi$ to $X_{\wh{R}}^\pm$.}
\begin{prop}\label{nv}
If $\varepsilon(\lambda)=\pm1$, then $\varphi^\pm\neq 0$ and $\varphi^\mp = 0$.
\end{prop}
\begin{proof}
In the following we choose an auxiliary prime $\ell $ coprime to $2D_{K/F}$ and $\Cond(\lambda)$. Without loss of generality, we assume 
$\varphi$ to be $\ell$-optimal since $\varphi$ is {non-zero}.

As in Theorem \ref{T:b.W}, let {$\fp\nmid \ell$} be a degree one prime of $F$ inert in $K$, and let $\chi_{0}=1$, and $\fc=1$. {We also assume $\pi_{\lambda,\fp}$ is unramified.}

We identify $F_{\fp}$ with $\BQ_p$ and the matrices appears in the Hecke action are seen as element of {$B^\times(F_\fp)\simeq\GL_2(\BQ_p)$ as in \S\ref{S:totq}.} 
In view of \eqref{sum} and the fact that $T_\fp$ has eigenvalue $0$ on newforms in the unramified supercuspidal representation $\pi_{\lambda,\fp}$, it follows that
\[\sum_{t\in(\Z/p^s\Z)^\x}\rho
 \left( \begin{pmatrix}
  t & \\
   & 1\\
  \end{pmatrix}\right)\ H_{\ell,\mu}\equiv p^{s}H_{\ell}+{p^{s-1}}\rho\left(\begin{pmatrix}
      p^{-2}& \\ & 1
  \end{pmatrix}\right) 
      H_{\ell}\pmod{\varpi}.\]
Recall that \[
      H_{\ell}:= \pair{x,\wh{\varphi}}_{k-2\Sigma},\]{where we may and do  suppose that $x=x_0$ in \S\ref{S:ellopt} since $\ell$ is auxiliary.}
Therefore, 

\begin{equation}\label{eq:tw}
p^{-s}\sum_{t\in(\Z/p^s\Z)^\x}\rho
\left(\begin{pmatrix}
t & \\
 & 1\\
\end{pmatrix}\right)
 H_{\ell,\mu}^{}\equiv \pair{x,\left(1+\frac{1}{p}\rho \left(\begin{pmatrix}
  p^{-2} & \\
   & 1\\
  \end{pmatrix}\right)\right)\wh{\varphi}}_{k-2\Sigma}\pmod{\varpi}.
  \end{equation}

  Put $\varepsilon$ for the sign of $\varepsilon(\lambda)$. As seen in the proof of Theorem~\ref{T:b.W}, ${H}_{\ell,\mu}^{-\varepsilon}\pmod{\varpi}$ is Eisenstein but ${H}_{\ell,\mu}^{\varepsilon}\pmod{\varpi}$ is non-Eisenstein.
 Moreover, $\sum_{t\in(\Z/p^s\Z)^\x}\rho
 \left(\begin{pmatrix}
 t & \\
  & 1\\
 \end{pmatrix}\right)
  H_{\ell,\mu}^{}\pmod{\varpi}$ is non-Eisenstein. Therefore the proof shows that 
 $\sum_{t\in(\Z/p^s\Z)^\x}\rho
\left(\begin{pmatrix}
 t & \\
  & 1\\
 \end{pmatrix}\right)
  H_{\ell,\mu}^{-\varepsilon}\pmod{\varpi}$ is Eisenstein, and in turn \[\sum_{t\in(\Z/p^s\Z)^\x}\rho
 \left(\begin{pmatrix}
 t & \\
  & 1\\
 \end{pmatrix}\right)
  H_{\ell,\mu}^{\varepsilon}\pmod{\varpi}\] non-Eisenstein.

In view of \eqref{eq:tw} and Lemma \ref{ee3}, it follows that $\pair{x,\wh{\varphi^\varepsilon}}_{k-2\Sigma} \pmod{\varpi}$ is non-Eisenstein and $\pair{x,\wh{\varphi^{-\varepsilon}}}_{k-2\Sigma}\pmod{\varpi}$ Eisenstein. Hence $\pair{x,\wh{\varphi^{-\varepsilon}}}_{k-2\Sigma}\pmod{\varpi}$ is zero by Lemma \ref{ee}.  

Suppose that $\wh{\varphi}^{-\varepsilon}\neq 0$. Choose a prime $\ell\gg 0$ such that $\wh{\varphi^{-\varepsilon}}\pmod{\varpi}$ is non-zero and $$\ell\Sigma> k-2\Sigma\in \BZ[\Sigma].$$ 
{Then it follows from Remark~\ref{rmk:opt} that} the non-vanishing of $\wh{\varphi^{-\varepsilon}}\pmod{\varpi}$ is equivalent to $\ell$-optimality of
$\wh{\varphi^{-\varepsilon}}\pmod{\varpi}$ with respect to $x$. In view of the preceding paragraph this is a contradiction.
 \end{proof}

We now describe an application to the split case. 
 Let $\fp\nmid \ell$ be a {degree one} prime of $F$ split in $K$.
 { Let $\chi_0$ be a finite order anticyclotomic Hecke character over $K$ such that the condition \eqref{Ass(D)} holds.} As in the inert case \eqref{tv-tw}, define ${H}_{\ell,\mu}$ with {respect} to $\fp$, $\chi_0$, an integer $s\geq 1$ and a primitive character $\mu$ of $\BZ/p^s\BZ$.

\begin{cor}\label{ep1}Let $\ell$ be a prime such that $(\ell,2D_{K/F})=1$. If $\varepsilon(\lambda)=\pm 1$, then ${H}_{\ell,\mu}^{\pm}\pmod{\varpi}$ is non-Eisenstein and ${H}_{\ell,\mu}^{\mp}=0$.
 \end{cor}

\begin{proof} 
{Let $\varepsilon\in \{\pm\}$ denote the sign of $\varepsilon(\lambda)$.
By Proposition \ref{nv}, $
  \pair{x,\wh{\varphi^{\varepsilon}}}_{k-2\Sigma}\pmod{\varpi}$ is non-Eisenstein and $$
  \pair{x,\wh{\varphi^{-\varepsilon}}}_{k-2\Sigma} = 0.$$
Since $\fc$ is only divisible by inert primes in the extension $K/F$, in view of \eqref{partei},
if $\pair{x,\wh{\varphi^{-\varepsilon}}}_{k-2\Sigma}$ vanishes then so does $H_\ell^{-\varepsilon}$. Furthermore, in view of Lemma \ref{tnn}  and \cite[Prop.~3.12]{CST}, if $
  \pair{x,\wh{\varphi^{\varepsilon}}}_{k-2\Sigma}\pmod{\varpi}$ is non-Eisenstein then so is $H_\ell^{\varepsilon}\pmod{\varpi}$.}
  
{To see the assertions for ${H}_{\ell,\mu}^\varepsilon\pmod{\varpi}$ and ${H}_{\ell,\mu}^{-\varepsilon}$, just note that ${H}_{\ell,\mu}$ and $H_\ell$ are related as in the proof of Theorems \ref{T:b.W} (see~equations~\eqref{sum0}~and~\eqref{sum}). 
Since $\fp$ is split in $K$ and $H_\ell^{-\varepsilon} =0$, it thus follows that 
${H}_{\ell,\mu}^{-\varepsilon}=0$. The assertion for ${H}_{\ell,\mu}^\varepsilon\pmod{\varpi}$ then follows by applying Lemma~\ref{ee3} to  $H_{\ell}^\varepsilon\pmod{\varpi}$.}
\end{proof}

\subsubsection{Split case} 
\label{rk1sp}

Let $\fp$ be a prime of $F$ that is split in $K$.

For $\nu\in \Xi_\fp$ with $\ord_{\fp}\cond(\nu_\fp)>\ord_{\fp}\cond(\lambda_{\fp})$, we have
$$
\varepsilon(\lambda\chi_{0}\nu)=
 \varepsilon(\lambda\chi_{0}) 
$$ by Lemma \ref{epsi:twist}. It follows that $\# \Xi_{\lambda\chi_0,\fp}^{+}=\infty$ if and only if $\varepsilon(\lambda\chi_{0})=+1$.

Our main result is the following.

\begin{thm}\label{T:3.W} 
Let $K$ be a CM field and $F$ its maximal totally real subfield. 
  Let $\lambda$ be a self-dual Hecke character over $K$ and $\sigma_\lambda$ the associated $\GL_{2}(F)$-cuspidal automorphic representation of conductor $\fN_{\lambda}$. Let $\fp$ be a degree one prime of $F$ that is split in $K$ and $\ell$ a prime such that $(\ell, 2\fp D_{K/F})=1$. Let $\chi_0$ be a finite order {anticyclotomic} Hecke character over $K$ such that the condition \eqref{Ass(D)} holds.
  Then for all but finitely many $\nu\in\Xi_{\lambda\chi_{0},\fp}^{+}$, we have
  \[{v}_{\ell}\left(\frac{\sqrt{|D_K|}\cdot L(1/2, \sigma_{\lambda, K}\otimes\chi_{0}\nu)}{\Omega_{\lambda}^{}}\right)=0.\]
  
  \end{thm}

\begin{proof} 
The argument is similar to the proof of \ref{T:b.W}, whose notation will appear below. 
In particular, we similarly define ${H}_{\ell,\mu}$ and $\Theta_{\chi_{0},n}(\pi_\lambda)$.
We may assume that $\varepsilon(\lambda\chi_{0})=+1$ as otherwise $\#\Xi_{\lambda\chi_{0},\fp}^+<\infty$. 

Now under the assumption $\varepsilon(\lambda\chi_{0})=+1$ note that ${H}^{\eta_{K/F}(\prod_{v|\fc}\varsigma_v^{(\ord_v\fc)})}_{\ell,\mu}\pmod{\varpi}$ is non-Eisenstein by Corollary~\ref{ep1} and the fact that 
$$\varepsilon(\lambda\chi_{0})=\varepsilon(\lambda)\eta_{K/F}(\fc)=\varepsilon(\lambda)\eta_{K/F}(\prod_{v|\fc}\varsigma_v^{(\ord_v\fc)})$$ (cf.~Lemma~\ref{twep}). Therefore, in light of Proposition \ref{C:Vatsal_Cornut}, we conclude that 
$$
{v}_\ell(\nu(\Theta_{\chi_{0},n}(\pi_\lambda)))= 0
$$
for all but finitely many $\nu \in \Xi_{\lambda\chi_{0},\fp}^+(\mu)$ for each primitive character $\mu$ of $\BZ/p^s\BZ$.

\end{proof}
\subsubsection{Consequence on non-vanishing of Hecke L-values}
As a consequence, we have the following.
\begin{cor}\label{maincor}
Let $K$ be a CM field and $F$ its maximal totally real subfield. 
  Let $\lambda$ be a self-dual Hecke character over $K$ and $\fp$ a degree one prime of $F$ unramified in $K$.
  Then we have 
    \[L(1,\lambda\nu)\neq 0 \] for all except finitely many 
    $\nu \in \Xi_{\lambda,\fp}^+$.
\end{cor}
\begin{proof}This just follows from Theorem \ref{T:b.W} and Theorem \ref{T:3.W} for $\fp$ inert and split in $K/F$ respectively, by taking $\chi_0=1$ and by choosing an auxiliary prime {$\ell$} such that $(\ell,2\fp D_{K/F})=1$.\end{proof}

\section{Non-vanishing of Hecke $L$-values modulo $\ell$}\label{s:nv-Hecke}

This section establishes mod $\ell$ non-vanishing of Hecke $L$-values building on the Rankin–Selberg results in \S\ref{s:nv}. The main result is Theorem \ref{mm}. A bridge among the two non-vanishing arises from a comparison of quaternionic and CM periods (cf.~Theorem~\ref{cpp}), which constitutes the core of this section.

 \subsection{Backdrop}
\subsubsection{Setting and strategy}
Let $K$ be a CM field and $F$ the maximal totally real subfield.
Let $\Sigma$ be a CM type of $K$. Let $\lambda$ be a self-dual Hecke character over $K$ in the sense of \eqref{theta} with infinity type $\Sigma+\kappa(1-c)$ for $\kappa\in \BZ_{\geq 0}[\Sigma]$.

Let $\fp$ be a prime of $F$ and $p$ the rational prime below $\fp$. Recall that $\Gamma_{\fp}=\Gal(K_{\infty}/K)$ is the Galois group of the anticyclotomic $\BZ_p^{[F_\fp:\BQ_p]}$-extension of $K$ unramified outside $\fp$, 
$\Xi_\fp$ the set of $\ov{\BQ}^\times$-valued finite order characters of $\Gamma_{\fp}$ and
$$
\Xi_{\lambda,\fp}^{+}=\{ \nu \in \Xi_{\fp}| \, \varepsilon(\lambda\nu)=+1 \}. 
$$
{Fix an embedding $\iota_\infty: \ov{\BQ}\hookrightarrow \BC$.}
Let $\ell$ be a prime and fix an embedding $\iota_\ell:\ov{\BQ}\hookrightarrow \ov{\BQ}_{\ell}$. Suppose that $\ell$ is ordinary for the extension $K/F$, and $\Sigma$ is an $\ell$-ordinary CM type. 
Then there are CM periods $$\Omega_\infty\in (\BC^\times)^\Sigma$$ associated to a \Neron differential on an abelian scheme over $\ov{\BZ}_{(\ell)}$ of CM type $(K,\Sigma)$ (cf.~\S\ref{in1}). 

In the following,  we consider mod $\ell$ non-vanishing of the normalised $L$-values\[L^{\alg}(1,\lambda\nu)=[O_K^\times:O^\times]\prod_{w\in \Sigma_\ell }G((\lambda\nu)_w)\cdot\frac{\pi^\kappa\Gamma_\Sigma(\Sigma+\kappa) \cdot L(1,\lambda\nu)}{\Omega_\infty^{\Sigma+2\kappa}}\in \ov{\BZ}_{(\ell)}\] as $\nu\in \Xi_{\lambda,\fp}^+$ varies (cf.~\S\ref{in1}).

\begin{remark}
 If $\ell\nmid 2D_F$, then $\ell\nmid [O_K^\times:O^\times]$.
\end{remark}

{For the rest of this section, we will assume that 
\begin{itemize}
  \item $\fp$ is a degree one prime of $F$ that unramified in $K/F$,
  \item $\ell\nmid 2\fp D_F$ is ordinary.
\end{itemize}}

{Let $\phi_\lambda$ be the associated $\GL_{2}(F)$-theta series of weight $k=2\Sigma+2\kappa$ and level $U_0(\mathfrak{N}_{\lambda})$ for
$\mathfrak{N}_{\lambda}=D_{K/F}{\rm N}_{K/F}(\Cond(\lambda))$.} Let $B$ be the totally definite quaternion algebra over $F$ such that  \[\varepsilon(B_v)=\eta_{K_v/F_v}(-1)\] 
for any place $v$ of $F$ (cf.~Lemma~\ref{lm:disc}). 
Let $\pi_\lambda$ be the cuspidal irreducible automorphic representation of $B^\times(\BA)$ associated to {$\phi_{\lambda}$.} Let $
f_{\lambda}\in \pi_\lambda^{\wh{R}^\times}$ be an $\ell$-optimal toric vector defined as in \S\ref{ss:ntv}, {and we take $S_0$ as in \S\ref{ss:ntv} to be coprime to $\mathfrak{N}_{\lambda}^-$. Then $\phi_\lambda$ does not depends on $S_0$}.

{Note that for any finite order Hecke character $\chi$ over $K$, we have a factorisation
\begin{equation}\label{$L$-fac}
L(1/2,\pi_{\lambda,K}\otimes\chi)=L(1,\lambda\chi)L(1,\lambda\chi^{-1}) 
\end{equation}}
of $L$-values. 
{In the context of Rankin--Selberg $L$-values the period}
\[\Omega_{\lambda}:=\Omega_{\pi_\lambda}=\frac{\pi^k(\phi_\lambda,\phi_\lambda)_{U_0(\mathfrak{N}_{\lambda})}}{\Gamma_{\Sigma}(k)\pair{f_{\lambda}, f_{\lambda}}_{\wh{R}^\times}}\]
arises naturally (cf.~\S\ref{ss:ntv}). 

In light of the factorisation \eqref{$L$-fac}, a basic problem is to compare the periods 
$\Omega_\lambda$ and $\Omega_\infty^{\Sigma+2\kappa}$. This is one of the main goals of the present section. Upon establishing the period comparison, the non-vanishing of Hecke $L$-value will basically follow from that of the Rankin--Selberg $L$-values in \S\ref{s:nv} and a lower bound for $\ell$-divisibility of the Hecke $L$-values.

\subsubsection{Lower bound for $\ell$-divisibility of Hecke $L$-values and preliminary mod $\ell$ non-vanishing}
\label{lowerbd}
This subsection describes a lower bound for the $\ell$-adic valuation of Hecke $L$-values  in terms of local invariants defined below and a preliminary mod $\ell$ non-vanishing result, based on Hida's approach \cite{Hi1,Hi2,Hi3,Hi4}. In turn, we obtain a lower bound for the $\ell$-adic valuation of  the ratio of quaternionic and CM periods.

For a prime $v|{\rm{N}}_{K/F}(\Cond(\lambda))$ of $F$ non-split in $K$, put \[\mu_\ell(\lambda_v)=\inf_{x\in K_v^\times}{v}_{\ell} (\lambda_v(x)-1).\]

\begin{thm}\label{lm:lb} 
Let $K$ be a CM field, $F$ its maximal totally real subfield and $\Sigma$ a CM type. 
  Let $\lambda$ be a self-dual Hecke character over $K$ of infinity type $\Sigma+\kappa(1-c)$ for $\kappa\in \BZ_{\geq 0}[\Sigma]$. Let {$\ell\nmid 2D_F$} be a prime such that $\Sigma$ is an $\ell$-ordinary CM type. Let $\fp$ be a degree one prime of $F$ with $\fp\nmid\ell D_{K/F}$.
\begin{itemize}
  \item [i)] {We} have  \[{v}_{\ell}\left(L^{\alg}(1,\lambda\nu)\right)\geq \sum_{\substack{v|{\rm{N}}_{K/F}(\Cond(\lambda)) \text{ inert}\\ v\nmid\fp}}\mu_{\ell}(\lambda_v) \]
  for all except finitely many $\nu\in \Xi_{\lambda,\fp}^+$.
  \item [ii)]Suppose that $\# \Xi_{\lambda,\fp}^+=\infty$. Then there exist infinitely many $\nu\in \Xi_{\lambda,\fp}^+$ such that 
    \[{v}_{\ell}\left(L^{\alg}(1,\lambda\nu)\right)= \sum_{\substack{v|{\rm{N}}_{K/F}(\Cond(\lambda)) \text{ inert}\\ v\nmid\fp}}\mu_{\ell}(\lambda_v). \]
\end{itemize}
    \end{thm}
    \begin{proof}
   For $\ell$ coprime to the conductor of $\lambda$ and $\fp$ split in $K$, this is the content of \cite[Prop.~5.1~\&~Thm.~A]{Hsieh:nv}. For the remaining cases, {see} \cite[Thm.~3.15,~Prop.~4.10~\&~Thm.~1.5]{He}. {See also the remark at the beginning of Section 6.1 of \cite{Hsieh:nv} for more explanation about part i).}

    \end{proof}

\begin{lem}\label{lm:lb2}
  Let $\lambda$ be a self-dual Hecke character over a CM field $K$ of infinity type $\Sigma+\kappa(1-c)$ for $\kappa\in \BZ_{\geq 0}[\Sigma]$. Let {$\ell\nmid 2D_F$} be a prime such that $\Sigma$ is an $\ell$-ordinary CM type. Then
  \[{v}_{\ell}\left(\frac{\pi^{2\kappa}\Gamma_\Sigma(\Sigma+\kappa)^2 \cdot \Omega_{\lambda}}{\Omega_\infty^{2\Sigma+4\kappa}}\prod_{w\in \Sigma_\ell }G(\lambda_w)^2\right)\geq 2\sum_{\substack{v|{\rm{N}}_{K/F}(\Cond(\lambda)) \text{ inert}}}\mu_{\ell}(\lambda_v) .\]
    \end{lem}
    \begin{proof}
Choose an auxiliary degree one prime $\fp$ of $F$ inert in $K$ and coprime to $\ell \mathfrak{N}_{\lambda}$.

By Theorem~\ref{T:b.W}, we have 
  {\[{v}_{\ell}\left(\frac{\sqrt{|D_K|}\cdot L(1,\lambda\nu)L(1,\lambda\nu^{-1})}{\Omega_{\lambda}^{}}\right)={v}_{\ell}\left(\frac{\sqrt{|D_K|}\cdot L(1/2, \pi_{\lambda, K}\otimes\nu)}{\Omega_{\lambda}^{}}\right)=0\]} for all but finitely many $\nu\in\Xi_{\lambda,\fp}^{+}.$ Note that since $\fp$ is inert in $K$, we have $\Xi_{\lambda,\fp}^+=\infty$ by Lemma \ref{epsi:twist}.
  In view of part (i) of Theorem \ref{lm:lb}, we can further take $\nu\in \Xi_{\lambda,\fp}^+$ such that 
    \[{v}_{\ell}\left(L^{\alg}(1,\lambda\nu)\right)\geq \sum_{\substack{v|{\rm{N}}_{K/F}(\Cond(\lambda)) \text{ inert}\\ v\nmid\fp}}\mu_{\ell}(\lambda_v) \] and \[{v}_{\ell}\left(L^{\alg}(1,\lambda\nu^{-1})\right)\geq \sum_{\substack{v|{\rm{N}}_{K/F}(\Cond(\lambda)) \text{ inert}\\ v\nmid\fp}}\mu_{\ell}(\lambda_v). \] 
    {Note that the assumption on $\ell$ implies that $D_K$ is an $\ell$-adic unit.}
    Hence, the result follows.
    \end{proof}

\begin{remark} \label{crtmu0}\
 Suppose that $v$ is a finite place of $F$ non-split in $K$, and $\lambda_v$ is ramified and self-dual. Let $\ell$ be an odd prime coprime to $v$. Suppose further that: whenever $v$ is inert in $K$ and $\lambda_v$ has exponential conductor $1$, we have $\ell\nmid (q_v+1)$. Then 
$$\mu_{\ell}(\lambda_v)=0.$$
\end{remark}
\subsection{Comparison of periods and $(\ell,p)$ non-vanishing}\label{prnv}
This subsection links the quaternionic and CM periods (cf.~Theorem~\ref{cpp}). 

We first describe the strategy. In light of the $(\ell,\fp)$ non-vanishing of Rankin–Selberg $L$-values established in \S\ref{s:nv} and the lower bound for the $\ell$-adic valuation of Hecke $L$-values as in \S\ref{lowerbd}, the sought after non-vanishing of Hecke $L$-values and the comparison of periods are  equivalent. 
Another key observation is that the comparison of periods 
essentially does not depend on the prime $\fp$. Consequently, it suffices to show that the lower  bound for Hecke $L$-values is an   
equality for two Hecke $L$-values appearing in the decomposition \eqref{$L$-fac} for an auxiliary split prime. Such a preliminary mod $\ell$ non-vanishing can be deduced from Hida's approach (cf.~\S\ref{lowerbd}).

We begin with the following.

\begin{prop}\label{cpnv1} Let $K$ be a CM field, $F$ its maximal totally real subfield and $\Sigma$ a CM type. 
Let $\lambda$ be a self-dual Hecke character over $K$ of infinity type $\Sigma+\kappa(1-c)$ for $\kappa\in \BZ_{\geq 0}[\Sigma]$ and $\fN_{\lambda}=D_{K/F}{\rm N}_{K/F}(\Cond(\lambda))$. Let {$\ell\nmid 2D_F$} be a prime such that $\Sigma$ is an $\ell$-ordinary CM type. Then the following are equivalent.
\begin{itemize}
 \item [a)] We have
 \[{v}_{\ell}\left(\frac{\pi^{2\kappa}\Gamma_\Sigma(\Sigma+\kappa)^2 \cdot \Omega_{\lambda}}{\Omega_\infty^{2\Sigma+4\kappa}}\prod_{w\in \Sigma_\ell }G(\lambda_w)^2\right)=2\sum_{\text{$v|{\rm{N}}_{K/F}(\Cond(\lambda))$ inert}}\mu_\ell(\lambda_v).\]
\item [b)]
Let $\fp\nmid \ell \mathfrak{N}_{\lambda}^-$ be a degree one prime of $F$. Let $\fq$ be a degree one prime of $F$ inert in $K$ and  coprime to $\ell\fp\mathfrak{N}_{\lambda}$ such that $\ell\nmid (q^2-1)$, where $q$ is the cardinality of residue field of $F_\fq$. Let $\chi_0\in \Xi_\fq$ be\footnote{{Since $\fq$ is inert in $K$, there exists $\chi_0\in \Xi_{\fq}$ such that $\varepsilon(\lambda\chi_0)=+1$ (cf.~Lemma~\ref{epsi:twist}) and hence $\#\Xi_{\lambda\chi_0,\fp}^+=\infty$.} If $\varepsilon(\lambda)=+1$, then we can take $\chi_0=1$ and require no assumption on $\fq$.
} a Hecke character such that $\#\Xi_{\lambda\chi_0,\fp}^+=\infty$. Then 
 \[{v}_{\ell}\left(L^{\alg}(1,\lambda\chi_0\nu)\right)= \sum_{\substack{v|{\rm{N}}_{K/F}(\Cond(\lambda)) \text{ inert}\\}}\mu_{\ell}(\lambda_v) \]
 and 
 \[{v}_{\ell}\left(L^{\alg}(1,\lambda(\chi_0\nu)^{-1})\right)= \sum_{\substack{v|{\rm{N}}_{K/F}(\Cond(\lambda)) \text{ inert}\\}}\mu_{\ell}(\lambda_v) \]
 for all except finitely many $\nu\in \Xi_{\lambda\chi_0,\fp}^+$. 
 \end{itemize}
\end{prop}

\begin{proof}
  Let $\fq$ be as in part b). 
  
  Based on the mod $\ell$ non-vanishing of Rankin--Selberg $L$-value as in Theorems \ref{T:b.W} and \ref{T:3.W} for the self-dual pair $(\pi_{\lambda},\chi_0\nu)$, we have 
   
  \begin{equation} \label{nv11}
  {v}_{\ell}\left(\frac{ \sqrt{|D_K|} \cdot L(1,\lambda\chi_0\nu)L(1,\lambda(\chi_0\nu)^{-1})}{\Omega_{\lambda}^{\cdot}}\right)=0
  \end{equation} 
  for all but finitely many $\nu\in\Xi_{\lambda\chi_0,\fp}^{+}$.
  Thus part b) implies part a). 
  
  Suppose that part a) holds. Then by the lower bound for Hecke $L$-value as in Theorem \ref{lm:lb}
  we have \begin{equation}\label{l1}{v}_{\ell}\left(L^{\alg}(1,\lambda\chi_0\nu)\right)\geq \sum_{\substack{v|{\rm{N}}_{K/F}(\Cond(\lambda)) \text{ inert}\\}}\mu_{\ell}(\lambda_v) \end{equation}
 and 
 \begin{equation}\label{l2}{v}_{\ell}\left(L^{\alg}(1,\lambda(\chi_0\nu)^{-1})\right)\geq \sum_{\substack{v|{\rm{N}}_{K/F}(\Cond(\lambda)) \text{ inert}\\}}\mu_{\ell}(\lambda_v), \end{equation} where we have used the fact that ${v}_{\ell}(\prod_{w\in \Sigma_\ell}G(\lambda_{w}\chi_{0,w}))={v}_{\ell}(\prod_{w\in \Sigma_\ell}G(\lambda_{w}))$ since $\chi_0$ is unramified at $\ell$, and $\mu_{\ell}(\lambda_v\chi_{0,v})=\mu_{\ell}(\lambda_v)$ for $v|{\rm N}_{K/F}(\Cond(\lambda))$ inert in $K$.
  Hence, in light of \eqref{nv11}, \eqref{l1} and \eqref{l2}, part b) follows.

\end{proof}

We have the following comparison of periods.
\begin{thm}\label{cpp}
Let $K$ be a CM field and $\Sigma$ a CM type. 
  Let $\lambda$ be a self-dual Hecke character over $K$ of infinity type $\Sigma+\kappa(1-c)$ for $\kappa\in \BZ_{\geq 0}[\Sigma]$. Let {$\ell\nmid 2D_F$} be a prime such that $\Sigma$ is an $\ell$-ordinary CM type.  If $\varepsilon(\lambda)=-1$, suppose that $\ell\neq 3$.
Then we have
\begin{equation}\label{com1}{v}_{\ell}\left(\frac{\pi^{2\kappa}\Gamma_\Sigma(\Sigma+\kappa)^2 \cdot  \Omega_{\lambda}^{}}{\Omega_\infty^{2\Sigma+4\kappa}}\prod_{w\in \Sigma_\ell }G(\lambda_w)^2\right)=2\sum_{\text{$v|{\rm{N}}_{K/F}(\Cond(\lambda))$ inert}}\mu_\ell(\lambda_v).\end{equation}
\end{thm}

\begin{proof}
  Choose an auxiliary degree one prime $\fq$ of $F$ inert in $K$ such that 
  \begin{itemize}
  \item[\tiny{$\bullet$}]the rational prime below it is coprime to $\ell {\rm{N}}_{K/F}(\Cond(\lambda))$ and,
  \item[\tiny{$\bullet$}]$\ell\nmid (q^2-1)$. 
  \end{itemize}Take $\chi_0\in \Xi_{\fq}$ such that 
  \[\varepsilon(\lambda\chi_0)=+1.\]Such a 
  $\chi_0$ exists since $\fq$ is inert in $K$ by the epsilon factor formula (cf.~Lemma~\ref{rk1int}).  If $\varepsilon(\lambda)=+1$, then 
  we simply take
  $\chi_0=1$ and  if $\varepsilon(\lambda)=-1$, we take $\chi_0$ to be sufficiently ramified. {Note that the assumption $\ell\neq 3$ implies the existence of a prime $\fq$ with $\ell\nmid q^2-1$}.

    Choose an auxiliary degree one prime $\fq'$ of $F$ such that 
    \begin{itemize}
    \item[\tiny{$\bullet$}] $\fq'$ splits in $K$ and,
    \item[\tiny{$\bullet$}]the rational prime $q'$ below it is coprime to $q\ell{\rm N}_{K/F}(\Cond(\lambda))$ and has odd residue characteristic.
    \end{itemize}
    
    Since $\varepsilon(\lambda\chi_0)=\varepsilon(\lambda\chi_0^{-1})=+1$, applying the mod $\ell$ non-vanishing as in part (ii) of Theorem \ref{lm:lb} for $\lambda\chi_0$ and $\lambda\chi_0^{-1}$,
  there exist characters $\nu_1, \nu_2\in \Xi_{\fq'}$ such that 
    \begin{equation}\label{lowhec}{v}_{\ell}\left(L^{\alg}(1,\lambda\chi_0\nu_1)\right)=\sum_{\substack{\text{$v|{\rm{N}}_{K/F}(\Cond(\lambda\chi_0))$ inert}}}\mu_\ell(\lambda_v\chi_{0,v})\end{equation}
    and
    \begin{equation}\label{lowhec1}{v}_{\ell}\left(L^{\alg}(1,\lambda\chi_0^{-1}\nu_2)\right)=\sum_{\substack{\text{$v|{\rm{N}}_{K/F}(\Cond(\lambda\chi_0^{-1}))$ inert}}}\mu_\ell(\lambda_v\chi_{0,v}^{-1}).\end{equation}
    Since $\fq'$ has odd residual characteristic, one can find\footnote{{That is, $\nu_0=\sqrt{\nu_1}\sqrt{\nu_2}$ and $\nu=\sqrt{\nu_1}\sqrt{\nu_2}^{-1}$, where $\sqrt{\nu_i}$ is a square-root of $\nu_i$ in $\Xi_{\fq'}$.}} $\nu_0,\nu\in \Xi_{\fq'}$ such that $\nu_0\nu=\nu_1$ and $\nu_0{\nu}^{-1}=\nu_2$.  Then we have 
    \[L(1/2,\pi_{\lambda\nu_0,K}\otimes\nu\chi_0)=L(1,\lambda\chi_0\nu_1)L(1,\lambda\chi_0^{-1}\nu_2).\]
  
    In the following we will show that 
    \begin{equation}\label{cmpd}{v}_{\ell}\left(\frac{\pi^{2\kappa}\Gamma_\Sigma(\Sigma+\kappa)^2 \cdot \Omega_{\lambda\nu_0}^{}}{\Omega_\infty^{2\Sigma+4\kappa}}\prod_{w\in \Sigma_\ell }G(\lambda_w\nu_{0,w})^2\right)=2\sum_{\text{$v|{\rm{N}}_{K/F}(\Cond(\lambda\nu_0))$ inert}}\mu_\ell(\lambda_v\nu_{0,v}),
    \end{equation}
    i.e. part (a) of Proposition \ref{cpnv1} for the Hecke character 
    $\lambda\nu_0$. Consequently, a series of implications ensues: 
     part (b) of Proposition \ref{cpnv1} for Hecke  
     characters $\lambda\nu_0\chi_0$ and $\lambda\nu_0\chi_0^{-1}$ holds by the same proposition (prime $\fp$ therein being $\fq'$ and $\chi_0, \fq$ being as above). 
    Then the latter implies 
    part (b) of Proposition \ref{cpnv1} for Hecke characters $\lambda\chi_0$ and $\lambda\chi_0^{-1}$ (prime $\fp$ therein being $\fq'$ and $\chi_0, \fq$ {as} above). 
Hence, part (a) of Proposition \ref{cpnv1} for the Hecke character $\lambda$ follows and so 
    \begin{equation}\label{}{v}_{\ell}\left(\frac{\pi^{2\kappa}\Gamma_\Sigma(\Sigma+\kappa)^2 \cdot \Omega_{\lambda}^{}}{\Omega_\infty^{2\Sigma+4\kappa}}\prod_{w\in \Sigma_\ell }G(\lambda_w)^2\right)=2\sum_{\text{$v|{\rm{N}}_{K/F}(\Cond(\lambda))$ inert}}\mu_\ell(\lambda_v).\end{equation}
    
    Henceforth we consider \eqref{cmpd}. 
    
    Since $\fq'$ splits in $K$, we can further choose Hecke {characters} 
    $\nu_1$ and $\nu_2$ such that 
    \begin{itemize}
    \item[\tiny{$\bullet$}]$\ord_{\fq'}\cond(\nu) \geq \ord_\fq'\cond (\pi_{\pi\nu_0})$ and,
    \item[\tiny{$\bullet$}] $[O_{K,\cond(\nu)}^\times:O^\times]=1$. 
    \end{itemize}
    Recall that
   \begin{equation}\label{lbd}{v}_{\ell}\left(\frac{L(1/2, \pi_{\lambda\nu_0, K}\otimes\chi_0\nu)}{\Omega_{\lambda\nu_0}}\right)\geq 0\end{equation}
   (cf.~Proposition~\ref{lbd:RS}).
In view of \eqref{lbd}, \eqref{lowhec} and \eqref{lowhec1}, we have 
  \[\begin{aligned}
    &{v}_{\ell}\left(\frac{\pi^{2\kappa}\Gamma_\Sigma(\Sigma+\kappa)^2\Omega_{\lambda\nu_0}}{\Omega_\infty^{2\Sigma+2\kappa}}\prod_{w\in \Sigma_\ell }G(\lambda_w\chi_{0,w})G(\lambda_w\chi_{0,w}^{-1})\right)\\
    \leq& \sum_{\substack{\text{$v|{\rm{N}}_{K/F}(\Cond(\lambda\chi_0))$ inert}}}\mu_\ell(\lambda_v\chi_{0,v})+\sum_{\substack{\text{$v|{\rm{N}}_{K/F}(\Cond(\lambda\chi_0^{-1}))$ inert}}}\mu_\ell(\lambda_v\chi_{0,v}^{-1})\\
    =&2\sum_{\text{$v|{\rm{N}}_{K/F}(\Cond(\lambda))$ inert}}\mu_\ell(\lambda_v).
     \end{aligned}\]
     Here the last equality follows from 
\[\sum_{\text{$v|{\rm{N}}_{K/F}(\Cond(\lambda))$ inert}}\mu_\ell(\lambda_v)=\sum_{\text{$v|{\rm{N}}_{K/F}(\Cond(\lambda))$ inert}}\mu_\ell(\lambda_v\chi_{0,v})\] since $\chi_{0,v}$ is trivial for $v|\Cond(\lambda)$ non-split in $K$. Also note that $\mu_\ell(\lambda_\fq\chi_{0,\fq}^{-1})=0$ if $\chi_{0,\fq}$ is ramified and $\ell\nmid (q^2-1)$.
     
  On the other hand, we have mod $\ell$ non-vanishing of Rankin--Selberg $L$-values:
  \[{v}_{\ell}\left(\frac{L(1/2, \pi_{\lambda\nu_0, K}\otimes\chi_0\mu)}{\Omega_{\lambda\nu_0}}\right)=0\] for all but finitely many $\mu\in \Xi_{\fq'}$. 
  Together with the lower bound for the $\ell$-adic valuation of Hecke $L$-values as in part~(i)~of~Theorem~\ref{lm:lb}, we then have\footnote{Here we use the fact that an unramified finite order twist at $\ell$ does not change the $\ell$-adic valuation of the underlying Gauss sum.} 
   \[ {v}_{\ell}\left(\frac{\pi^{2\kappa}\Gamma_\Sigma(\Sigma+\kappa)^2  \cdot \Omega_{\lambda\nu_0}}{\Omega_\infty^{2\Sigma+4\kappa}}\prod_{w\in \Sigma_\ell }G(\lambda_w\chi_{0,w})G(\lambda_w\chi_{0,w}^{-1})\right)\geq 2\sum_{\text{$v|{\rm{N}}_{K/F}(\Cond(\lambda))$ inert}}\mu_\ell(\lambda_v).\]

   {Hence,} the proof of \eqref{cmpd} concludes.
\end{proof}
\subsection{Mod $\ell$ non-vanishing}
{The main result of this section is the following.}
\begin{thm}\label{mm}
Let $K$ be a CM field, $F$ the maximal totally real subfield and $\Sigma$ a CM type. 
  Let $\lambda$ be a self-dual Hecke character over $K$ of infinity type $\Sigma+\kappa(1-c)$ for $\kappa\in \BZ_{\geq 0}[\Sigma]$. Let {$\ell\nmid 2D_F$} be a prime such that $\Sigma$ is an $\ell$-ordinary CM type. If $\varepsilon(\lambda)=-1$, suppose that $\ell\neq 3$.
Let $\fp\nmid \ell {\rm N}_{K/F}(\Cond(\lambda))^-$ be a degree one prime of $F$. Then 
    \[{v}_{\ell}\left(L^{\alg}(1,\lambda\nu)\right)= \sum_{\substack{v|{\rm{N}}_{K/F}(\Cond(\lambda)) \text{ inert}}}\mu_{\ell}(\lambda_v) \]for all except finitely many $\nu\in \Xi_{\lambda,\fp}^+$.
\end{thm}

\begin{proof}
We may assume that $\varepsilon(\lambda)=+1$ if $\fp$ splits in $K$. 

Recall the period relation 
\[{v}_{\ell}\left(\frac{\pi^{2\kappa}\Gamma_\Sigma(\Sigma+\kappa)^2 \cdot  \Omega_{\lambda}^{}}{\Omega_\infty^{2\Sigma+4\kappa}}\prod_{w\in \Sigma_\ell }G(\lambda_w)^2\right)=\sum_{\text{$v|{\rm{N}}_{K/F}(\Cond(\lambda))$ inert}}\mu_\ell(\lambda_v) \] (cf.~Theorem~\ref{cpp}). 
Thus the result follows from Proposition~ \ref{cpnv1} by taking the initial self-dual Hecke character to be $\lambda$, $\fp$ to be the given prime and 
$\chi_0=1$.
      
  \end{proof}
\section{CM Iwasawa main conjecture}\label{s:IMC} 
The main results of this section are Theorem \ref{snv} on simultaneous mod $p$ non-vanishing of Hecke $L$-values,  and Theorem~\ref{mcc}
on Eisenstein congruence divisibility towards the Iwasawa main conjecture for a CM field, which completes Hsieh's proof  \cite{Hs2}.

In this section we switch\footnote{as Iwasawa theory is traditionally a $p$-adic theory} the prior role of primes $\ell$ and $p$. 
\subsection{Simultaneous mod $p$ non-vanishing of Hecke $L$-values} Let $K/F$ be a $p$-ordinary CM quadratic extension for an odd prime $p$. Let $\Sigma$ be a $p$-ordinary CM type of $K$. Fix embeddings $\iota_{\infty}:\ov{\BQ}\hookrightarrow \BC$ and $\iota_{p}:\ov{\BQ}\hookrightarrow \ov{\BQ}_p$. 

Let $\chi$ be an algebraic Hecke character over $K$ of infinity type $k\Sigma+\kappa(1-c)$ for $\kappa=\sum_{\sigma\in \Sigma}\kappa_\sigma\sigma\in \BZ_{\geq 0}[\Sigma]$ and $k\Sigma+\kappa\in \BZ_{>0}[\Sigma]$. Put 
\[L^{\alg}(1, \chi):=[O_K^\times:O^\times]\prod_{w\in \Sigma_p }G(\chi_w)\cdot \frac{\pi^\kappa\Gamma_\Sigma(k\Sigma+\kappa)L(1,\chi)}{\Omega_\infty^{k\Sigma+2\kappa}} \ov{\BZ}_{(p)}\cap \ov{\BQ}\] as in \eqref{eq:alg}.

We say that $\chi$ is {residually} self-dual if \[\wh{\chi}|_{\wh{F}^\times}\omega_F\equiv \eta_{K/F}\pmod{\fm_p}\] for $\wh{\chi}$ the $p$-adic avatar of $\chi$ (cf.~\cite[Def.~1.1]{HT}), {$\omega_F$  the Teichm\"uller character over $F$} and $\fm_p$ the maximal ideal of {$\ov{\BZ}_{p}$}.

Fix a rational prime $\ell\neq p$ and a prime $\fl|\ell$ of $F$. Let $\Gamma_\fl$ be the maximal anticyclotomic $\BZ_\ell$-free extension of $K$ unramified outside $\fl$ and let $\Xi_\fl$ be the set of $\ov{\BQ}^\times$-valued finite order characters of $\Gamma_{\fl}$.

{In this section we assume that $\chi$ is either self-dual or 
not residually self-dual. Then the invariant  \[{\mu}_{p}(\chi):=\sum_{\substack{v|{\rm{N}}_{K/F}(\Cond(\chi))\ \text{non-split}\\v\nmid \fl}}\mu_{p}(\chi_v) \] often provides a lower bound for the $p$-adic valuation of $L^{\alg}(1, \chi\mu)$ for $\nu\in \Xi_{\fl}$, which is generically an equality.} 
The following is a main result of \cite{Hi4,Hsieh:nv,He} (cf.~\cite[Thm.~1.5]{He}).
\begin{thm}\label{nsnv}
Let $\chi$ be an algebraic Hecke character over a $p$-ordinary CM quadratic extension $K/F$ which is either self-dual or not residually self-dual. Let 
$\fl\nmid p$ be a degree one prime of $F$. 
{Suppose} that
\begin{itemize}
  \item  [\tiny{$\bullet$}] $p\nmid 2D_F$,
  \item  [\tiny{$\bullet$}]  $\fl\nmid  \Cond(\chi|_{\BA^\times})$,
  \item [\tiny{$\bullet$}] { If $\chi$ is self-dual and $\fl$ splits in $K$, then $\varepsilon(\chi)=+1$.}
\end{itemize}
Then there exist infinitely many $\nu\in \Xi_\fl$ such that 
 \[{v}_p \left(L^{\alg}(1,\chi\nu)\right)=\mu_{p}(\chi).\] 
\end{thm}
In combination with our non-vanishing result in the self-dual case (cf.~Theorem~\ref{mainthm}), we obtain the following simultaneous non-vanishing.
\begin{thm}\label{snv}
Let $\chi_1$ and $\chi_2$ be algebraic Hecke characters over a $p$-ordinary CM quadratic extension $K/F$ such that $\chi_1$ is self-dual and $\chi_2$ not residually self-dual. Let 
$\fl\nmid p$ be a degree one prime of $F$. 
{Suppose} that
\begin{itemize}
  \item  [\tiny{$\bullet$}] $p\nmid 6D_F$,
  \item  [\tiny{$\bullet$}]  $\fl\nmid {\rm N}_{K/F}(\Cond(\chi_1))^-\cdot \Cond(\chi_2|_{\BA^\times})$,
  \item [\tiny{$\bullet$}] {$\varepsilon(\chi_1)=+1$ if $\fl$ splits in $K/F$.}
\end{itemize}
Then there exist infinitely many $\nu\in \Xi_\fl$ such that 
\[{v}_p \left(L^{\alg}(1,\chi_1\nu)\right)=\mu_{p}(\chi_1)\]and \[{v}_p \left(L^{\alg}(1,\chi_2\nu^{-1})\right)=\mu_{p}(\chi_2).\] 
\end{thm}
\begin{proof}
By {Theorem \ref{mm}}, we have \begin{equation}\label{nv1}{v}_p \left( L^{\alg}(1,\chi_1\nu)\right) =\mu_{p}(\chi_1)\end{equation} for all except finitely many $\nu\in \Xi_{\chi_1,\fl}^+$.
On the other hand, as a consequence of Theorem \ref{nsnv}, we have \begin{equation}\label{nv2}{v}_p\left( L^{\alg}(1,\chi_2\nu^{-1})\right)=\mu_{p}(\chi_2)\end{equation} for infinitely many $\nu\in \Xi_{\fl}$, as explained below. 

Note that the third condition is equivalent to  $\Xi_{\chi_1,\fl}^+=\infty$.
So the result follows from the claim that \eqref{nv2} holds for infinitely many $\nu\in \Xi_{\chi_1,\fl}^+$, which is the content of the rest of the proof.

{We first give a description of $\Xi_{\chi_1,\fl}^+$.} Since $\chi_1$ is self-dual, the second assumption implies that $\chi_1|_{F_\fl^\times}$ is unramified, which is {equivalent} to the prime $ \fl$ being unramified in $K$. In particular,  if $\fl$ splits in $K$, then $\Xi_{\chi_1,\fl}^+$ and $\Xi_{\fl}$ differ by finitely many elements, and if $\fl$ is inert in $K$, then the exponential conductors of $\nu\in \Xi_{\chi_1,\fl}^+$ at $\fl$ have the {same} parity $\pmod{2}$  (cf.~Lemma~\ref{epsi:twist}).

Since $\chi_2$ is not residually self-dual, the argument in \cite[\S5.3.2]{He} shows that \eqref{nv2} holds for infinitely many $\nu\in \Xi_{\chi_1,\fl}^+$, concluding the proof.

\end{proof}
\begin{remark}\label{hyp'}\
  \begin{itemize}
      \item [i)] If $\chi$ is not residually self-dual, then $\mu_{p}(\chi)=0$ if for each prime $\fq\parallel \Cond(\chi)$ of {$K$ non-split in $K/F$} and coprime to $\fl$, we have $p\nmid({\rm N}_{K/\BQ}(\fq)-1)$. 
\item [ii)] Similarly, if $\chi$ is self-dual, then $\mu_{p}(\lambda)=0$
if $p$ is odd and for each prime $\fq\nmid\fl$ of $F$ inert in $K$ such that $\fq\parallel \cond(\lambda_\fq)$, we have 
$p\nmid({\rm N}_{F/\BQ}(\fq)+1)$. 
\item[iii)] Based on the results of \cite{He}, it is also possible to allow $\chi_2$ to be residually self-dual. We restrict to the above version since it suffices for our application to the CM main conjecture.

  \end{itemize}

\end{remark}

\subsection{Iwasawa Main conjecture for CM fields}
\subsubsection{Set-up}

Let $K/F$ be a $p$-ordinary CM quadratic extension for an odd prime $p$. Let $h_K^-$ be the relative class number. Fix embeddings $\iota_\infty:\ov{\BQ}\hookrightarrow \BC$ and 
$\iota_p: \ov{\BQ}\hookrightarrow \ov{\BQ}_{p}$. 
Let $\Sigma$ be a $p$-ordinary CM type of $K$, which gives rise to a $p$-adic CM type $\Sigma_p$ via the embedding $\iota_p$.

 Let $K_\infty$ be the compositum of the cyclotomic $\BZ_p$-extension and the anticyclotomic $\BZ_p^{[F:\BQ]}$-extension of $K$. Put $\Gamma_K=\Gal(K_\infty/K)$. Let $K'$ be a finite abelian extension of $K$ that contains $K(\mu_p)$ and is disjoint with $K_\infty$. Put $\Delta=\Gal(K'/K)$ and $K_\infty'=K_\infty K'$. Let $\psi:\Delta\ra \ov{\BZ}_p^\times$ be a finite order Hecke character over $K$. Put $R=\ov{\BZ}_p^{\text{un}}[\psi]$ and $\Lambda=R[\![\Gamma_K]\!]$. Let $\fm_{\Lambda}$ be the maximal ideal of $\Lambda$.

 Let $M_\Sigma$ be the maximal $p$-abelian $\Sigma_p$-ramified extension
of $K_\infty'$. A primary Iwasawa-theoretic object is the module: $$X_\Sigma:=\Gal(M_\Sigma/K'_\infty)\otimes_{\BZ_p[\Delta][\![\Gamma_K]\!]} R[\Delta][\![\Gamma_K]\!].$$ 
Define $X_\Sigma^{(\psi)}$ as the maximal $\psi$-isotypic quotient of $X_\Sigma$, which is a finitely generated torsion $\Lambda$-module, and let \[F_{\Sigma}(\psi)\in \Lambda\] be its characteristic power series.
In the analytic realm we have the associated Katz $p$-adic $L$-function 
\[L_{\Sigma}(\psi)\in \Lambda,\] which interpolates
the algebraic part of critical Hecke $L$-values associated to twists of $\psi$ by certain characters of $\Gamma_K$ (cf.~\cite{Kz,HT}). 

The CM Iwasawa main conjecture posits the following equality of ideals of the Iwasawa algebra $\Lambda$. 
\begin{conj} We have
\[(F_{\Sigma}(\psi))= (L_{\Sigma}(\psi)).\] 

\end{conj}
Hsieh's Eisenstein congruence approach \cite{Hs2} to the CM Iwasawa Main conjecture  begins by realizing $L_\Sigma(\psi)$ as constant term of a $\Lambda$-adic Eisenstein series $\CE(\psi)$ on the unitary group $U(2,1)$ over $K$. One then seeks to construct Selmer cocycles given the vanishing of $\CL_\Sigma(\psi)$ at a prime ideal of $\Lambda$. A key necessity for the latter: mod $p$ non-vanishing of $\CE(\psi)$.

In \cite{Hs2} this $p$-primitivity is approached by considering the Fourier--Jacobi expansion of $\CE(\psi)$. Specifically, a period formula for the Fourier--Jacobi coefficients reduces the $p$-primitivity to a simultaneous mod $p$ non-vanishing of two Hecke $L$-values in an anticyclotomic $\ell$-adic family for an auxiliary prime $\ell\neq p$ such that one of the underlying Hecke characters is self-dual and the other is not residually self-dual. 

At precisely this stage the argument of \cite{Hs2} is erroneous: it is based on the mod $p$ non-vanishing of anticyclotomic Hecke $L$-values for all but finitely many $\nu$ for a given Hecke character as originally claimed in \cite{Hi1,Hsieh:nv}. Due to a subsequently found gap in Hida's strategy towards Zariski density of CM points on self-product of the mod $p$ Hilbert modular Shimura variety and extra conditions in its correction \cite[Cor.~2.12]{Hi4}, the mod $p$ non-vanishing is only known to hold for infinitely many $\nu$ for a given Hecke character. In particular, the intersection of the set of anticyclotomic characters $\nu$ for which the mod $p$ non-vanishing holds for a given pair of Hecke characters may well be empty. Based on our mod $p$ non-vanishing of anticyclotomic Hecke $L$-values for self-dual characters, the proof of simultaneous mod $p$ non-vanishing can still be completed, as detailed in the next subsection (cf.~Theorem~\ref{snv}).

\subsubsection{Main results} Let $\Psi: \Gal(K_{\infty}'/K)\ra \Lambda$ be the $\Lambda$-adic deformation of $\psi$.

We have the associated ordinary $\Lambda$-adic Eisenstein series $$\CE(\psi):=\CE^{\text{ord}}(\Psi|1,\fn)$$ on $U(2,1)$ {as} in \cite[Cor.~5.15]{Hs2}, where $\fn$ is an auxiliary prime to $p$ ideal of $O$ satisfying the conditions (a1)-(a3) {\color{green}as} in \cite{Hs2}.  
It is constructed from the pullback of a Siegel Eisenstein series on $U(2,2)$, whose constant term is essentially the Katz $p$-adic $L$-function $L_{\Sigma}(\psi)$ (cf.~\cite[Prop.~6.6]{Hs2}).

One of our main results is the following. 

\begin{thm}\label{nv:eis} Let $K/F$ be a $p$-ordinary CM quadratic extension and $\Sigma_p$ a $p$-adic CM type. 
For $K'$ a finite extension of $K$ as above, let $\psi:\Delta:=\Gal(K'/K) \ra \ov{\BZ}_p^\times$ be a character. Suppose that
  \begin{itemize}
    \item [\tiny{$\bullet$}] $p\nmid 6 h_K^- D_F \# \Delta$,
    \item [\tiny{$\bullet$}] $\psi$ is unramified at $\Sigma_p^c$ and $\psi\omega_K^{-a}$ is unramified at $\Sigma_p$ for some integer $a\nequiv 2\pmod{p-1}$.
  \end{itemize} Then the Eisenstein series $\CE(\psi)$ is $p$-primitive, i.e. one of its Fourier--Jacobi coefficients is non-zero modulo $\fm_{\Lambda}$.
\end{thm}
\begin{proof}
 In the following we give a proof based on \cite{Hs2} and we emphasise the utility of our simultaneous mod $p$ non-vanishing of Hecke $L$-values as in~Theorem~\ref{snv}. Specifically, we give an exposition\footnote{We refer the reader to \cite[\S7]{Hs2} for details.} of arguments of \cite{Hs2} which reduce the assertion to a special case of Theorem \ref{snv}. We follow the notation of \cite{Hs2}.

Let $\chi=\wh{\varepsilon}\circ\Psi$ be a specialization of $\Psi$, where 
$\varepsilon$ is of infinity type $-k\Sigma$ for $k\geq 4$ and $\wh{\varepsilon}$ the $p$-adic avatar. Let $E^{\text{ord}}(\chi|1,\fn)$ be the corresponding specialization of $\CE^{\text{ord}}(\Psi|1,\fn)$.
Then the Fourier--Jacobi coefficients of $E^{\text{ord}}(\chi|1,\fn)$ are $p$-integral holomorphic Theta functions. The space of $p$-integral Theta functions has a basis of $p$-primitive Theta functions $\Theta$ which are 
eigenforms for Shintani operators (cf.~\cite[Prop.~7.3]{Hs2}).
For the $p$-primitivity of Fourier--Jacobi coefficients $F_a^m$, it suffices to show that the normalized inner product 
\begin{equation}
\frac{(F_{a}^m,\Theta)}{(\Theta,\Theta)} \in \ov{\BZ}_{p}^\times
\end{equation} for a $\Theta$.

Let $\Theta$ be a $p$-primitive Theta function such that it is an eigenform for Shintani operators and 
the associated unitary self-dual Hecke character $\kappa\nu$ has epsilon factor $+1$ (cf.~\cite[Prop.~7.3]{Hs2}), {where $$\kappa\in \CX_0(\fa,m)$$ is a fixed Hecke character as in \cite[Sec.~7.1]{Hs2} such that $\kappa|\cdot|_{\BA_K^\times}^{-1/2}$ is self-dual in the sense of \eqref{theta}}, $\nu$ is an auxiliary finite order anticyclotomic character on $\Gamma_\fl\simeq \BZ_\ell$ for an auxiliary degree one primes $\fl$ of $F$ split in $K$.

{In the next paragraph, we will use the notation `$\doteq$' to denote an equality up to a $p$-unit.}

By~\cite[Thm.~7.9]{Hs2}, we have
\begin{equation}\frac{(F_{a}^m,\Theta)}{(\Theta,\Theta)} \doteq L^{\alg}(1,(\chi\kappa\nu)^{-1}|\cdot|_{\BA_K^\times}^{1/2})\ov{I(\Theta)},\end{equation} where $I(\Theta)$ is the $U(1)$-period of $\Theta$ which is $p$-integral by the main theorem of complex multiplication and arithmetic {theory} of line bundles on Abelian varieties (cf.~\cite[Sec.~7.3]{Hs2}), $\ov{I(\Theta)}$ is the complex conjugate of $I(\Theta)$.
Moreover, it follows from $p$-integral version of the explicit Rallis inner product formula \cite{Yang97,MS02} for the pair $(U(1),U(1))$ that 
\begin{equation}\frac{|I(\Theta)|^2}{(\Theta,\Theta)} \doteq L^{\alg}(1,\kappa\nu |\cdot|_{\BA_K^\times}^{-1/2})\end{equation} 
(cf.~\cite[Thm.~7.11]{Hs2}).
Consequently, it suffices to show that the right hand side of 
 \[\frac{(F_{a}^m,\Theta)}{(\Theta,\Theta)}|I(\Theta)|^2
 \doteq 
 L^{\alg}(1,(\chi\kappa\nu)^{-1}|\cdot|_{\BA_K^\times}^{1/2})\cdot L^{\alg}(1,\kappa\nu |\cdot|_{\BA_K^\times}^{-1/2})\]
 is a $p$-unit. 

 Note that $\kappa\nu$ is self-dual and $\chi\kappa\nu$ is not residually self-dual. 
Hence, the result follows from Theorem \ref{snv} by {choosing} an auxiliary degree one prime $\fl$ of $F$ split in $K$, Hecke characters $\chi$ and $\kappa$ such that  
$$\mu_{p}((\chi\kappa)^{-1}|\cdot|_{\BA_K^\times}^{1/2})=\mu_p(\kappa|\cdot|_{\BA_K^\times}^{-1/2})=0$$  as in the proof of~\cite[Prop.~7.13~\&~Prop.~7.16]{Hs2}.
\end{proof} \begin{remark}\label{rmk:hyp} 
It may be possible to remove the hypothesis $p \nmid  \Delta$ 
{since} our non-vanishing results allow $p$ and the conductor of Hecke character to have common prime factors (cf.~\cite[Rmk.~after~Thm.~2]{Hs2} and~{Remark}~\ref{hyp'} (iii)).
\end{remark}

\begin{thm}\label{mcc} Let $K/F$ be a $p$-ordinary CM quadratic extension and $\Sigma_p$ a $p$-adic CM type. 
For $K'$ a finite extension of $K$ as above, let $\psi:\Delta:=\Gal(K'/K) \ra \ov{\BZ}_p^\times$ be a character.
 Suppose that
  \begin{itemize}
    \item  [\tiny{$\bullet$}] $p\nmid 6 h_K^- D_F \# \Delta$,
    \item  [\tiny{$\bullet$}]$\psi$ is unramified at $\Sigma_p^c$ and $\psi\omega_K^{-a}$ is unramified at $\Sigma_p$ for some integer $a\nequiv 2\pmod{p-1}$.
  \end{itemize}
Then we have 

  \[L_{\Sigma}(\psi) \big{|} F_{\Sigma}(\psi).\]

\end{thm}
\begin{proof}
In view of the $p$-primitivity of $\CE(\psi)$ as in Theorem~\ref{nv:eis}, the Eisenstein congruence approach on the unitary group $U(2,1)$ yields: the Katz $p$-adic $L$-function divides\footnote{An auxiliary Deligne--Ribet $p$-adic $L$-function also appears in the argument of \cite{Hs2}, which is innocuous due to Wiles' proof of main conjecture for totally real number fields.} 
the Eisenstein ideal measuring congruences between $\CE(\psi)$ and the $\Lambda$-adic cusp forms on $U(2,1)$ (cf.~\cite[Cor.~7.20]{Hs2}). On the other hand, by Ribet's lemma, the Eisenstein ideal divides the characteristic ideal of the Selmer group (cf.~\cite[Cor.~8.14]{Hs2}).
\end{proof}

\end{document}